\documentclass[a4paper,10pt]{amsart}

\usepackage[top=2.5cm, bottom=2.5cm, left=2.5cm, right=2.5cm]{geometry}

\usepackage{nomencl}

\usepackage{amssymb,amsmath,amsthm,ascmac,array,bm}
\usepackage[dvipdfmx]{graphicx}
\usepackage{color}

\newtheorem{thm}{Theorem}[section]

\newtheorem{lem}[thm]{Lemma}
\newtheorem{prop}[thm]{Proposition}
\newtheorem{cor}[thm]{Corollary}
\newtheorem{ex}[thm]{Example}
\newtheorem{rem}[thm]{Remark}
\newtheorem{ass}[thm]{Assumption}

\numberwithin{equation}{section}
\allowdisplaybreaks

\newcommand{\mca}{\mathcal{A}}

\newcommand{\mcc}{\mathcal{C}}

\newcommand{\mcf}{\mathcal{F}}
\newcommand{\mcl}{\mathcal{L}}

\newcommand{\mfb}{\mathfrak{b}}

\newcommand{\mbbd}{\mathbb{D}}

\newcommand{\mbbn}{\mathbb{N}}

\newcommand{\mbbr}{\mathbb{R}}

\newcommand{\mbby}{\mathbb{Y}}

\newcommand{\mbbrp}{\mathbb{R}_{+}}

\newcommand{\al}{\alpha} \newcommand{\lam}{\lambda} \newcommand{\ep}{\epsilon} 
\newcommand{\vp}{\varphi} \newcommand{\del}{\delta} \newcommand{\sig}{\sigma}
\newcommand{\D}{\Delta} \newcommand{\Sig}{\Sigma} \newcommand{\gam}{\gamma}
 \newcommand{\Gam}{\Gamma}

\newcommand{\p}{\partial}
\newcommand{\ra}{\rangle} \newcommand{\la}{\langle}
\newcommand{\wc}{\Rightarrow} 
\newcommand{\cil}{\xrightarrow{\mcl}} 
\newcommand{\scl}{\xrightarrow{\mcl_{s}}} 
\newcommand{\cip}{\xrightarrow{p}} 
\newcommand{\asc}{\xrightarrow{a.s.}} 

\newcommand{\argmax}{\mathop{\rm argmax}}

\def\ds#1{\displaystyle{#1}}

\def\nn{\nonumber}

\def\pr{\mathbb{P}} \def\E{\mathbb{E}}

\newcommand{\sqllf}{\mathbb{H}}

\def\diag{{\rm diag}}

\def\sumj{\sum_{j=1}^{n}}

\def\lp{L\'evy process}
\def\lm{L\'evy measure}
\def\ld{L\'evy density}

\def\wp{Wiener process}
\def\cadlag{c\`adl\`ag}

\def\tz{\theta_{0}}
\def\tes{\hat{\theta}_{n}}
\def\aes{\hat{\alpha}_{n}}

\def\ges{\hat{\gam}_{n}}

\title[Estimation of locally stable L\'{e}vy driven SDE]
{Non-Gaussian quasi-likelihood estimation of SDE driven by locally stable L\'{e}vy process}

\author{Hiroki Masuda}
\address{Faculty of Mathematics, Kyushu University, 744 Motooka Nishi-ku Fukuoka 819-0395, Japan}
\email{hiroki@math.kyushu-u.ac.jp}

\date{\today}

\begin{document}
\setlength{\baselineskip}{4.5mm}

\begin{abstract}
We address estimation of parametric coefficients of a pure-jump L\'{e}vy driven univariate stochastic differential equation (SDE) model, which is observed at high frequency over a fixed time period. It is known from the previous study \cite{Mas13as} that adopting the conventional Gaussian quasi-maximum likelihood estimator then leads to an inconsistent estimator.
In this paper, under the assumption that the driving {\lp} is locally stable, we extend the Gaussian framework into a non-Gaussian counterpart, by introducing a novel quasi-likelihood function formally based on the small-time stable approximation of the unknown transition density. The resulting estimator turns out to be asymptotically mixed normally distributed without ergodicity and finite moments for a wide range of the driving pure-jump {\lp es}, showing much better theoretical performance compared with the Gaussian quasi-maximum likelihood estimator. Extensive simulations are carried out to show good estimation accuracy. The case of large-time asymptotics under ergodicity is briefly mentioned as well, where we can deduce an analogous asymptotic normality result.
\end{abstract}

\maketitle

\section{Introduction}\label{sec_intro}

Stochastic differential equation (SDE) driven by a {\lp} is one of basic models to describe time-varying physical and natural phenomena. There do exist many situations where non-Gaussianity of distributions of data increments, or of a residual sequence whenever available, is significant in small-time, making diffusion type models observed at high frequency somewhat inappropriate to reflect reality; see \cite{AitJac14} as well as the enormous references therein, and also \cite{GraSam10}. This non-Gaussianity may not be well modeled even by a diffusion with compound-Poisson jumps as well since jump-time points are then rather sparse compared with sampling frequency, so that most increments are approximately Gaussian except for intervals containing jumps. SDE driven by a pure-jump {\lp} may then serve as a good natural candidate model. For those models, however, a tailor-made estimation procedure seems to be far from being well developed, which motivated our present study.

In this paper, we consider a solution to the univariate Markovian SDE
\begin{equation}
dX_{t}=a(X_{t},\al)dt+c(X_{t-},\gam)dJ_{t}
\label{SDE}
\end{equation}
defined on an underlying complete filtered probability space $(\Omega,\mcf,(\mcf_{t})_{t\in\mbbrp},\pr)$ with
\begin{equation}
\mcf_{t}=\sig(X_{0})\vee\sig(J_{s}; s\le t),
\label{Ft_def}
\end{equation}
where:
\begin{itemize}
\item The initial random variable $X_{0}$ is $\mcf_{0}$-measurable;
\item The driving noise $J$ is a symmetric pure-jump ({\cadlag}) L\'evy process independent of $X_{0}$;
\item The trend coefficient $a:\mbbr\times\Theta_{\al}\to\mbbr$ and scale coefficient $c:\mbbr\times\Theta_{\gam}\to\mbbr$ are assumed to be known except for the $p$-dimensional parameter 
\begin{equation}
\theta:=(\al,\gam)\in\Theta_{\al}\times\Theta_{\gam}=\Theta\subset\mbbr^{p},
\nonumber
\end{equation}
with $\Theta_{\al}\in\mbbr^{p_{\al}}$ and $\Theta_{\gam}\in\mbbr^{p_{\gam}}$ being bounded convex domains.
\end{itemize}
Our objective here is estimation of $\theta$, when the true value $\theta_{0}=(\al_{0},\gam_{0})\in\Theta$ does exist and the process $X$ is observed only at discrete but high-frequency time instants $t^{n}_{j}=jh$, $j=0,1,\dots,n$, with nonrandom sampling step size
\begin{equation}
h=h_{n}\to 0
\nonumber
\end{equation}
We will mostly work under the bounded-domain asymptotics
\footnote{The equidistance assumption on the sampling times could be removed as soon as the ratios of $\min_{j\le n}(t_{j}-t_{j-1})$ and $\max_{j\le n}(t_{j}-t_{j-1})$ are bounded in an appropriate order. This may be shown by the same line as in \cite[pp.1604--1605]{Mas13as}.}:
\begin{equation}
\text{$T_{n} \equiv T$, \quad i.e. $h=\frac{T}{n}$}
\label{fixed.T}
\end{equation}
for a fixed terminal sampling time $T\in(0,\infty)$, that is, we observe not a complete path $(X_{t})_{t\le T}$ but the time-discretized step process
\begin{equation}
X^{(n)}_{t}:=X_{\lfloor t/h\rfloor h}
\label{disc.X}
\end{equation}
over the period $[0,T]$; see Section \ref{sec_ergodic} for the large-time asymptotics where $T_{n}\to\infty$ under the ergodicity.

Due to the lack of a closed-form formula for the transition distribution, a feasible approach based on the {\it genuine} likelihood function is rarely available. 
In this paper, we will introduce a \textit{non-Gaussian quasi-likelihood function}
\footnote{Non-Gaussian quasi-likelihoods have not received much attention compared with the popular Gaussian one. Among others, we refer to the recent paper \cite{FanQiXiu14} for a certain non-Gaussian quasi-likelihood estimation of a possibly heavy-tailed GARCH model, and also to \cite{ZhuLin11} for self-weighted Laplace quasi-likelihood in a time series context.}, which extends the prototype mentioned in \cite{Mas_Dublin} and \cite{Mas15jjss},
under the {\it locally (symmetric) $\beta$-stable property} of $J$ in the sense that
\begin{equation}
\mcl(h^{-1/\beta}J_{h}) \wc S_{\beta},\qquad h\to 0,
\label{A_J1}
\end{equation}
where $S_{\beta}$ stands for the standard symmetric $\beta$-stable distribution corresponding to the characteristic function
\begin{equation}
\vp_{0}(u):=\exp(-|u|^{\beta});
\nonumber
\end{equation}
among others, we refer to \cite{JanWer94}, \cite{SamTaq94}, \cite{Sat99} and \cite{Zol86}) for comprehensive accounts of general stable distributions.
It is known from \cite[Proposition 1]{BerDon97} that, as long as the linear scaling $\ep_{h}J_{h}$ for some $\ep_{h}\to 0$ is concerned, the strictly stable distribution is the only possible asymptotic distribution. Many locally stable {\lp es} with finite variance can exhibit large-time Gaussianity (i.e. central limit effect) in addition to the small-time non-Gaussianity.
In the main results, we will assume the locally stable property \eqref{A_J1} with a stronger mode (see Lemmas \ref{key.lemma} and \ref{key.lemma.cf}) and that the stability index $\beta$ is known with
\begin{equation}
\beta\in[1,2).
\nonumber
\end{equation}
It should be noted that the value $\beta$ is also known as the Blumenthal-Getoor activity index defined by
\begin{equation}
\beta:=\inf\bigg\{b\ge 0: \int_{|z|\le 1}|z|^{b}\nu(dz)<\infty\bigg\},
\nn
\end{equation}
which measures degree of $J$'s jump activity.

The proposed maximum quasi-likelihood estimator $\tes=(\aes,\ges)$ has the property that
\begin{equation}
\left(\sqrt{n}h^{1-1/\beta}(\aes-\al_{0}),~\sqrt{n}(\ges-\gam_{0})\right)
\label{intro.amn}
\end{equation}
is asymptotically mixed normally distributed under some conditions, which extends the previous works \cite{Mas11} and \cite{Mas13as} that adopting the Gaussian quasi-maximum likelihood estimator; we refer to \cite[Section 2]{Mas16.arxiv} for some formal comparisons.
In particular, the convergence \eqref{intro.amn} clarifies that the activity index $\beta$ determines the rate of convergence of estimating the trend parameter $\al$; note that
\begin{equation}
\sqrt{n}h^{1-1/\beta}=T^{1-1/\beta}n^{(2-\beta)/(2\beta)}\to\infty
\nonumber
\end{equation}
as $n\to\infty$. It should be emphasized that this estimator can be much more efficient compared with the Gaussian maximum quasi-likelihood estimator studied in \cite{Mas13as}. Most notably, unlike the case of diffusions, we can estimate not only the scale parameter $\gam$ but also the trend parameter $\al$, with the explicit asymptotic distribution in hand; see \cite{Gob02} for the related local asymptotic normality result. To prove the asymptotic mixed normality, we will take a doubly approximate procedure based on the Euler-Maruyama scheme combined with the stable approximation of $\mcl(h^{-1/\beta}J_{h})$ for $h\to 0$. Our result provides us with the first systematic methodology for estimating the possibly non-linear pure-jump L\'{e}vy driven SDE \eqref{SDE} based on a non-Gaussian quasi-likelihood.

\medskip

Here are a couple of further remarks on our model.

\begin{enumerate}

\item The model is semiparametric in the sense that we do not completely specify the {\lm} of $\mcl(J)$, while supposing the parametric coefficients; of course, the {\lm} is an infinite-dimensional parameter, so that $\beta$ alone never determines the distribution $\mcl(J)$ in general. In estimation of $\mcl(X)$, it would be desirable (whenever possible) to estimate the parameter $\theta$ with leaving the remaining parameters contained in {\lm} as much as unknown. The proposed quasi-likelihood, termed as (non-Gaussian) stable quasi-likelihood, will provide us with a widely applicable tool for this purpose.

\item It is assumed from the very beginning that $\beta<2$ so that $J$ contains no Gaussian component. Normally, the simultaneous presence of a non-degenerate diffusion part and a non-null jump part makes the parametric-estimation problem much more complicated. The recent papers \cite{JinKonLiu12} and \cite{KonLiuJin15} discussed usefulness of pure-jump models. Although they are especially concerned with financial context, pure-jump processes should be useful for model building in many other application fields where non-Gaussianity of time-varying data is of primary importance. For example, econometrics, signal processing, population dynamics, hydrology, radiophysics, turbulence, biological molecule movement, noise-contaminated biosignals, and so on;
we refer to \cite{AitJac14}, \cite{CosBocCauFer16}, and \cite{FagAmiUns14} for some recent related works.

\item Finally, our model \eqref{SDE} may be formally seen as a continuous-time analogue to the discrete-time model
\begin{equation}
X_{j} = a(X_{j-1},\al) + c(X_{j-1},\gam)\ep_{j},\qquad j=1,\dots,n,
\nn
\end{equation}
where $\ep_{j}$ are i.i.d. random variables. By making use of the locally stable property \eqref{A_J1}, our model setup enables us to formulate a flexible and unified estimation procedure, which cannot be shared with the discrete-time counterpart. The bounded-domain asymptotics \eqref{fixed.T} makes it possible to ``localize'' the event, sidestepping both stability (such as the ergodicity) and moment condition on $\mcl(J_{1})$. Instead, in order to deduce the asymptotic mixed normality we need much more than the (martingale) central limit theorem with Gaussian limit.
Fortunately, we have the very general tool to handle this, that is, Jacod's characterization of conditionally Gaussian martingales (see \cite{GenJac93} and \cite{Jac97}, and also Section \ref{sec_localization}), which in particular can deal with the SDE \eqref{SDE} when $J$ is a pure-jump {\lp}.

\end{enumerate}

\medskip

The following conventions and basic notations are used throughout this paper.
We will largely suppress the dependence on $n$ from the notations $t^{n}_{j}$ and $h$.
For any process $Y$,
\begin{equation}
\D_{j}Y=\D^{n}_{j}Y:=Y_{t_{j}}-Y_{t_{j-1}}
\nonumber
\end{equation}
denotes the $j$th increments, and we write
\begin{equation}
g_{j-1}(v)=g(X_{t_{j-1}},v)
\nonumber
\end{equation}
for a function $g$ having two components, such as $a_{j-1}(\al)=a(X_{t_{j-1}},\al)$. 
For a variable $x=\{x_{i}\}$, we write $\p_{x}=\{\frac{\p}{\p x_{i}}\}_{i}$, $\p_{x}^{2}=\{\frac{\p^{2}}{\p x_{i}\p x_{j}}\}_{i,j}$, and so forth, with omitting the subscript $x$ when there is no confusion;
given a function $f=f(s_{1},\dots,s_{k}):\,S_{1}\times\dots\times S_{k}\to\mbbr^{m}$ with $S_{i}\subset\mbbr^{d_{i}}$, we write $\p_{s_{1}}^{j_{1}}\dots\p_{s_{k}}^{j_{k}}f$ for the array of partial derivatives of dimension $m\times(\prod_{l=1}^{k}d_{l}j_{l})$.
The characteristic function of a random variable $\xi$ is denoted by $\vp_{\xi}$.
For any matrix $M$ we let $M^{\otimes 2}:=MM^{\top}$ ($\top$ denotes the transpose). 
We use $C$ for a generic positive constant which may vary at each appearance,
and write $a_{n}\lesssim b_{n}$ when $a_{n}\le Cb_{n}$ for every $n$ large enough. 
Finally, the symbols $\cip$ and $\cil$ denote the convergences in $\pr$-probability and in distribution, respectively;
all the asymptoics below will be taken for $n\to\infty$ unless otherwise mentioned.

\medskip

The paper is organized as follows. We first describe the basic model setup in Section \ref{sec_setup}. The main results are presented in Section \ref{sec_SQL}, followed by numerical experiments in Section \ref{sec_simulations}. 
Section \ref{sec_lem.proofs} presents the proofs of the criteria for the key assumptions given in Section \ref{sec_setup}.
Finally, Section \ref{sec_proofs} is devoted to proving the main results.

\section{Basic setup and assumptions}\label{sec_setup}

\subsection{Locally stable L\'{e}vy process}\label{sec_LSLP}

\subsubsection{Definition and criteria}

We denote by
\begin{equation}
g_{0,\beta}(y):= \frac{c_{\beta}}{|z|^{1+\beta}},\qquad z\ne 0,
\nn
\end{equation}
the {\ld} of $S_{\beta}$, where
\begin{equation}
c_{\beta} := \frac{1}{2}\bigg\{\frac{1}{\beta}\Gam(1-\beta)\cos\bigg(\frac{\beta\pi}{2}\bigg)\bigg\}^{-1}
\label{c.beta_def}
\end{equation}
with $c_{1}=\lim_{\beta\to 1}c_{\beta}=\pi^{-1}$ (\cite[Lemma 14.11]{Sat99}).

\begin{ass}[Driving noise structure]
\begin{enumerate}
\item The {\lp} $J$ has no drift, no Gaussian component, and a symmetric {\lm} $\nu$, so that
\begin{equation}
\vp_{J_{t}}(u) = \exp\bigg( t\int (\cos(uz)-1)\nu(dz)\bigg), \qquad u\in\mbbr.
\label{J.lk}
\end{equation}
The {\lm} $\nu$ admits a Lebesgue density $g$ of the form
\begin{equation}
g(z)=g_{0,\beta}(z)\left\{1+\rho(z)\right\},\qquad z\ne 0,
\nonumber
\end{equation}
where $\rho: \mbbr\setminus\{0\}\to[-1,\infty)$ is a measurable symmetric function such that for some constants $\del>0$, $c_{\rho}\ge 0$, and $\ep_{\rho}>0$,
\begin{equation}
|\rho(z)| \le c_{\rho}|z|^{\del},\qquad |z|\le \ep_{\rho},\quad z\ne 0.
\nonumber
\end{equation}
\item 
In addition to (1), $\rho$ is continuously differentiable in $\mbbr\setminus\{0\}$ and the triplet $(c_{\rho},\beta,\del)$ satisfies either
\begin{enumerate}
\item $c_{\rho}=0$ (that is, $\rho(z)=0$ for $|z|\le\ep_{\rho}$), or
\item $c_{\rho}>0$ and $\del>\beta$ with
\begin{equation}
|\rho(z)| + |z\p\rho(z)| \le c_{\rho}|z|^{\del}, \qquad z\ne 0.
\nonumber
\end{equation}
\end{enumerate}
\end{enumerate}
\label{A_J}
\end{ass}

The function $\rho$ controls the degree of ``$S_{\beta}$-likeness'' around the origin.
Also, if in particular $\rho$ is bounded, then $\E(|J_{1}|^{q})<\infty$ for $q\in(-1,\beta)$; see \cite[Theorem 25.3]{Sat99}.

Lemma \ref{key.lemma} below, which will play an essential role in the proof of the main results, shows that Assumption \ref{A_J} ensures not only the locally stable property \eqref{A_J1} but also an $L^{1}$-local limit theorem with specific convergence rate.

\begin{lem}
\begin{enumerate}
\item Let Assumption \ref{A_J}(1) hold with the function $\rho$ being bounded.
Then, the distribution $\mcl(h^{-1/\beta}J_{h})$ for $h\in(0,1]$ admits a positive smooth Lebesgue density $f_{h}$ such that
\begin{equation}
\int |y|^{\kappa}\left|f_{h}(y)-\phi_{\beta}(y)\right|dy \to 0
\label{llt.imp.1}
\end{equation}
for each $\kappa\in(0,\beta)$.
\item Let Assumption \ref{A_J} hold, and assume that there exists a constant $K>0$ such that $g(z)=0$ (equivalently, $\rho(z)= -1$) for every $|z|>K$.
Then, for any $\ep>0$ we have
\begin{equation}
\int \left|f_{h}(y)-\phi_{\beta}(y)\right|dy \lesssim h^{1-\ep}.
\label{key.lemma-1}
\end{equation}
\end{enumerate}
\label{key.lemma}
\end{lem}

\medskip

The additional assumptions on $\rho$ in Lemma \ref{key.lemma} will not be real restrictions, because in the proof of our main result we will truncate the support of $\nu$ in order to deal with possibly heavy-tailed $J$. The localization argument is allowed under the bounded-domain asymptotics \eqref{fixed.T}; see Section \ref{sec_localization}.
It follows from \eqref{key.lemma-1} that under \eqref{fixed.T} we can pick a sufficiently small $\ep>0$ to ensure that
\begin{equation}
\sqrt{n}\int \left|f_{h}(y)-\phi_{\beta}(y)\right|dy \to 0.
\label{llt.imp.2}
\end{equation}
As will be seen later, it is the convergences \eqref{llt.imp.1} and \eqref{llt.imp.2} that are essential for the proofs of the main results.
Assumption \ref{A_J} (also Assumption \ref{A_J.cf} given below) serves as a set of sufficient conditions.

\medskip

Assumption \ref{A_J} is designed to give conditions only in terms of the {\ld} $g$. However, Assumption \ref{A_J}(2) excludes cases where $\rho\in\mcc^{1}(\mbbr\setminus\{0\})$ with $|\p\rho(0+)|>0$ (hence $\del=1$) and $\beta\le 1$; note that we do not explicitly impose that $\beta\in[1,2)$ in Lemma \ref{key.lemma} (and also in Lemma \ref{key.lemma.cf} below).
It is possible to give another set of conditions. 

Denote by $\phi_{\beta}$ the density of $S_{\beta}$, and let
\begin{equation}
\psi_{h}(u) := \log\vp_{h}(u),
\nonumber
\end{equation}
where $\vp_{h}$ denotes the characteristic function of $h^{-1/\beta}J_{h}$:
\begin{equation}
\vp_{h}(u) := \big\{\vp_{J_{1}}(h^{-1/\beta}u)\big\}^{h},\qquad h>0, \quad u\in\mbbr.
\nn
\end{equation}

\begin{ass}[Driving noise structure]
Assumption \ref{A_J}(1) holds with the function $\rho$ being bounded, $\psi_{h}\in\mcc^{1}(\mbbr\setminus\{0\})$, 
and there exist constants $c_{\psi}\ge 0$ and $r\in[0,1]$ and a function $\ep_{\psi}(h)$ such that $\ep_{\psi}(h)\to 0$ as $h\to 0$ and that
\begin{align}
& |\p_{u}\psi_{h}(u)| \lesssim \frac{1}{u}\vee u^{c_{\psi}},\qquad u>0, \label{A_J.cf-1} \\
& \int_{(0,\infty)}u^{r}\vp_{0}(u) \left| \p_{u}\psi_{h}(u) + \beta u^{\beta-1} \right| du \le \ep_{\psi}(h).
\label{A_J.cf-2}
\end{align}
\label{A_J.cf}
\end{ass}

\begin{lem}
Under Assumption \ref{A_J.cf}, we have
\begin{equation}
\int \left|f_{h}(y)-\phi_{\beta}(y)\right|dy \lesssim (\ep_{\psi}(h) \vee h^{(\del/\beta)\wedge 1})^{\frac{\beta}{\beta+r}}.
\nonumber
\end{equation}
In particular, \eqref{llt.imp.2} holds if $\sqrt{n}(\ep_{\psi}(h) \vee h^{(\del/\beta)\wedge 1})^{\frac{\beta}{\beta+r}} \to 0$.
\label{key.lemma.cf}
\end{lem}

Lemmas \ref{key.lemma}(2) and \ref{key.lemma.cf} have no inclusion relation with different domains of the applicability.

\subsubsection{Examples}

\begin{ex}{\rm 
Trivially, Assumption \ref{A_J} is satisfied by the $\beta$-stable driven case ($\mcl(J_{1})=S_{\beta}$) and the whole class of the driving {\lp} considered in \cite{CleGlo15} and \cite{CleGloNgu17}, where $c_{\rho}=0$ (equivalently $\rho\equiv 0$).
See also Remark \ref{rem_asymp.efficiency}.
\label{ex.cgn}
}\qed\end{ex}

In the next two concrete examples, Assumption \ref{A_J.cf} is helpful for verification of \eqref{llt.imp.2} while Assumption \ref{A_J}(2) may not.

\begin{ex}[Symmetric tempered $\beta$-stable {\lp} with $\beta\in[1,2)$]{\rm 
The symmetric exponentially tempered $\beta$-stable {\lp}, which we denote by $TS_{\beta}(\lam)$ for $\lam>0$, is defined through the {\ld}
\begin{equation}
z\mapsto g_{0,\beta}(z)\exp(-\lam|z|);
\nonumber
\end{equation}
we refer to \cite{KawMas11jcam} and the references therein for details of general tempered stable distributions.
When $\mcl(J_{1})=TS_{\beta}(\lam)$, then $\mcl(h^{-1/\beta}J_{h})=TS_{\beta}(\lam h) \wc S_{\beta}$ as $h\to 0$.
Assumption \ref{A_J}(1) is satisfied with $\rho(z)=\exp(-\lam |z|)-1$, hence $\del=1$ for any $\beta<2$.
However, Assumption \ref{A_J}(2) then requires $\beta<1$, which conflicts with the case $\beta\in[1,2)$ of our interest here. We will instead verify Assumption \ref{A_J.cf} for $\beta\in[1,2)$;
the function $\psi_{h}$ is explicitly given by
\begin{equation}
\psi_{h}(u)=
\begin{cases}
\ds{\frac{1}{\pi}\bigg\{ \lam h \log\bigg(1+ \frac{u^{2}}{\lam^{2}h^{2}}\bigg) -2u\arctan\bigg(\frac{u}{\lam h}\bigg)\bigg\}} & (\beta=1) \\[3mm]
\ds{2c_{\beta}\Gam(-\beta) \bigg[ (\lam^{2}h^{2/\beta}+u^{2})^{\beta/2}\cos\bigg\{\beta\arctan\bigg(\frac{u}{\lam h^{1/\beta}}\bigg)\bigg\} - \lam^{\beta}h
\bigg]} & (\beta\in(1,2))
\end{cases}
\nonumber
\end{equation}
First we consider $\beta=1$, where $\p_{u}\psi_{h}(u)=-\frac{2}{\pi}\arctan(\frac{u}{\lam h})$. 
Using the estimate
\begin{equation}
\sup_{y\ge 0} \bigg| \frac{\arctan y - (\pi/2)}{-(1/y)} \bigg| < \infty,
\label{arctan.estimate}
\end{equation}
we have $|\p_{u}\psi_{h}(u) + 1| \lesssim | \arctan(\frac{u}{\lam h}) -\frac{\pi}{2} | \lesssim \frac{h}{u}$. This gives
\begin{equation}
\int_{(0,\infty)}u^{r}\vp_{0}(u) |\p_{u}\psi_{h}(u) + 1| du \lesssim h \int_{(0,\infty)}u^{r-1}e^{-u} du \lesssim h
\nonumber
\end{equation}
Lemma \ref{key.lemma.cf} ensures that $\sqrt{n}\int|f_{h}(y)-\phi_{\beta}(y)|dy \lesssim (nh^{\frac{2}{1+r}})^{1/2}\to 0$, with $r>0$ being small enough;
even when $T_{n}\to\infty$, it suffices for the last convergence to suppose that $nh^{2-\ep_{1}}\to 0$ some $\ep_{1}>0$.

Next we consider $\beta\in(1,2)$ with $r=0$; we may control $r\ge 0$ independently of the case $\beta=1$.
Substituting the expression \eqref{c.beta_def} we see that $\p_{u}\psi_{h}(u) + \beta u^{\beta-1}$ equals the sum of three terms $\D_{h,k}(u)$ ($k=1,2,3$), where
\begin{align}
\D_{h,1}(u) &:= \beta\bigg(\cos\frac{\beta\pi}{2}\bigg)^{-1} \bigg[ \cos\frac{\beta\pi}{2} - 
\cos\bigg\{\beta\arctan\bigg(\frac{u}{\lam h^{1/\beta}}\bigg)\bigg\} \bigg] \frac{u}{(\lam^{2}h^{2/\beta}+u^{2})^{1-\beta/2}},
\nn\\
\D_{h,2}(u) &:= \beta \bigg\{ u^{\beta-1} - \bigg(\frac{u^{2}}{ \lam^{2}h^{2/\beta}+u^{2} }\bigg)^{1-\beta/2} \bigg\},
\nonumber
\end{align}
and $\D_{h,3}(u)$ satisfies that $|\D_{h,3}(u)| \lesssim h^{1/\beta}(\lam^{2}h^{2/\beta}+u^{2})^{\beta/2-1} \lesssim h^{1/\beta} u^{\beta-2}$.
By the mean-value theorem together with \eqref{arctan.estimate}, we derive $|\D_{h,1}(u)| \lesssim \frac{u}{u^{2(1-\beta/2)}}\cdot (h^{-1/\beta}u)^{-1} = h^{1/\beta}u^{\beta -2}$.
Hence, for $\beta>1$,
\begin{equation}
\int_{(0,\infty)}\vp_{0}(u)\big| \D_{h,1}(u) + \D_{h,3}(u) \big| du \le h^{1/\beta} \int_{(0,\infty)} e^{-u^{\beta}} u^{\beta-2}du \lesssim h^{1/\beta}.
\nonumber
\end{equation}
Further, we have
\begin{align}
\int_{(0,\infty)}\vp_{0}(u)|\D_{h,2}(u)|du &\le  \int_{(0,\infty)}\D_{h,2}(u)du \nn\\
&= \int_{(0,\infty)} \beta \bigg\{ u^{\beta-1} - \bigg(\frac{u^{2}}{ \lam^{2}h^{2/\beta}+u^{2} }\bigg)^{1-\beta/2} \bigg\}du \nn\\
&= - \left[ (u^{2}+\lam^{2}h^{2/\beta})^{\beta/2} - (u^{2})^{\beta/2} \right]_{0+}^{\infty} \nn\\
&= -\frac{\lam^{2}h^{2/\beta}}{2}\bigg[ \int_{0}^{1}(u^{2}+s\lam^{2}h^{2/\beta})^{\beta/2-1}ds \bigg]_{0+}^{\infty} \nn\\
&= \frac{1}{2}\lam^{\beta}h \int_{0}^{1}s^{\beta/2-1}ds \lesssim h.
\nonumber
\end{align}
Combining these estimates yields that
\begin{equation}
\int_{(0,\infty)}\vp_{0}(u) |\p_{u}\psi_{h}(u) + \beta u^{\beta-1}| du \lesssim h^{1\wedge (1/\beta)} = h^{1/\beta}.
\nonumber
\end{equation}
Again Lemma \ref{key.lemma.cf} concludes that $\sqrt{n}\int|f_{h}(y)-\phi_{\beta}(y)|dy \lesssim (nh^{2/\beta})^{1/2}\to 0$ if $nh^{2/\beta}\to 0$, which is automatic under \eqref{fixed.T}.
}\qed\label{ex.ts}\end{ex}

\begin{ex}[Symmetric generalized hyperbolic {\lp}]{\rm 
The symmetric generalized hyperbolic distribution \cite{Bar77}, denoted by $GH(\lam,\eta,\zeta)$, is infinitely divisible with the characteristic function
\begin{equation}
u \mapsto \bigg(\frac{\eta^{2}}{\eta^{2}+u^{2}}\bigg)^{\lam/2}\frac{K_{\lam}(\zeta\sqrt{\eta^{2}+u^{2}})}{K_{\lam}(\eta\zeta)},
\nonumber
\end{equation}
where $K_{\lam}$ denotes the modified Bessel function of the third kind with index $\lam\in\mbbr$.
In this example, we will make extensive use of several exact and/or asymptotic properties of $K_{\lam}$,
without notice in most places; we refer to \cite[Chapter 9]{AbrSte92} for details. If $\mcl(J_{1})=GH(\lam,\eta,\zeta)$, then for each $u\in\mbbr$
\begin{equation}
\E(e^{ih^{-1}J_{h}})
= \bigg(\frac{(\eta h)^{2}}{(\eta h)^{2}+u^{2}}\bigg)^{\lam h/2} \bigg(\frac{K_{\lam}\big( (\zeta/h)\sqrt{(\eta h)^{2}+u^{2}} \big)}{K_{\lam}(\eta\zeta)}\bigg)^{h}
\to \exp(-\zeta |u|),\quad h\to 0,
\nonumber
\end{equation}
showing that $J$ is locally Cauchy (for $\zeta=1$). In the sequel we set $\mcl(J_{1})=GH(\lam,\eta,1)$.
By \cite{Rai00} we know that the {\lm} of $GH(\lam,\eta,1)$ admits the density such that
\begin{equation}
z \mapsto \frac{1}{\pi |z|^{2}}\bigg\{ 1 + \frac{\pi}{2}\bigg(\lam+\frac{1}{2}\bigg)|z| + o(|z|)\bigg\},\qquad |z|\to 0.
\nonumber
\end{equation}
Hence Assumption \ref{A_J}(2) fails to hold since $\beta=\del=1$ here (except for the case of $\lam=-1/2$, corresponding to a symmetric normal inverse Gaussian {\lp} \cite{Bar98}).
We will observe that $J$ instead meets Assumption \ref{A_J.cf} for any $(\lam,\eta)\in\mbbr\times(0,\infty)$.
In this case, $\E(|J_{1}|^{q})<\infty$ for any $q>0$.

Direct computations give
\begin{equation}
\p_{u}\psi_{h}(u) + 1 = 1 - \frac{u}{\sqrt{(\eta h)^{2}+u^{2}}}
\frac{K_{\lam+1}}{K_{\lam}}\bigg( \frac{1}{h}\sqrt{(\eta h)^{2}+u^{2}} \bigg).
\label{ex.gh-1}
\end{equation}
The right-hand side is essentially bounded, hence in particular \eqref{A_J.cf-1} holds.

We will verify \eqref{A_J.cf-2} with $r=0$.
Given \eqref{ex.gh-1}, we see that $\int_{(0,\infty)}\vp_{0}(u) \left| \p_{u}\psi_{h}(u) + 1 \right| du \le I'_{h}+I''_{h}$, where
\begin{align}
I'_{h} &:= \int_{(0,\infty)}\vp_{0}(u) \bigg( 1- \frac{u}{\sqrt{(\eta h)^{2}+u^{2}}}\bigg)du, \nn\\
I''_{h} &:= \int_{(0,\infty)}\vp_{0}(u) \frac{u}{\sqrt{(\eta h)^{2}+u^{2}}} 
\bigg|\frac{K_{\lam+1}}{K_{\lam}}\bigg( \frac{1}{h}\sqrt{(\eta h)^{2}+u^{2}} \bigg) -1\bigg|du.
\nonumber
\end{align}
For $I'_{h}$, we divide the domain of integration into $(0,1]$ and $(1,\infty)$ and then derive the following estimates.
\begin{itemize}
\item The $(0,1]$-part can be bounded by
\begin{equation}
\int_{(0,1]}\bigg( 1- \frac{u}{\sqrt{(\eta h)^{2}+u^{2}}}\bigg)du 
= 1-\left[ \sqrt{(\eta h)^{2}+u^{2}} \right]_{0+}^{1} = \frac{2\eta h}{1+\eta h + \sqrt{(\eta h)^{2}+1}} \lesssim h.
\nonumber
\end{equation}
\item The $(1,\infty)$-part equals
\begin{equation}
\int_{(1,\infty)}\vp_{0}(u) \frac{(\eta h)^{2}}{\big(u + \sqrt{(\eta h)^{2}+u^{2}} \big)\sqrt{(\eta h)^{2}+u^{2}}}du
\lesssim h^{2}\int_{(1,\infty)}\vp_{0}(u) \lesssim h^{2}.
\nonumber
\end{equation}
\end{itemize}
Hence we obtain $I'_{h}\lesssim h$.
As for $I''_{h}$, we first make the change of variables:
\begin{equation}
I''_{h} = \int_{(0,\infty)} \vp_{0}(vh) \frac{vh}{\sqrt{\eta^{2}+v^{2}}} 
\bigg|\frac{K_{\lam+1}}{K_{\lam}}\big(\sqrt{\eta^{2}+v^{2}} \big) -1\bigg|dv.
\nonumber
\end{equation}
Just like the case of $I'_{h}$, we look at the $(0,1]$-part and the $(1,\infty)$-part separately.
\begin{itemize}
\item Since we are supposing that $\eta>0$, the $(0,1]$-part trivially equals $O(h)$.
\item The $(1,\infty)$-part is somewhat more delicate. It can be written as
\begin{equation}
I''_{h,1}:=
\int_{(1,\infty)} \vp_{0}(vh) \frac{vh}{\sqrt{\eta^{2}+v^{2}}} 
\bigg| \bigg(\lam+\frac{1}{2}\bigg)\frac{1}{\sqrt{\eta^{2}+v^{2}}} + \sig(v) \bigg|dv,
\nonumber
\end{equation}
where
\begin{equation}
\sig(v) := \frac{K_{\lam+1}}{K_{\lam}}\big(\sqrt{\eta^{2}+v^{2}} \big) -1 -\bigg(\lam+\frac{1}{2}\bigg)\frac{1}{\sqrt{\eta^{2}+v^{2}}},
\nonumber
\end{equation}
which satisfies the property $\sup_{v\ge 1}\left|(\eta^{2}+v^{2})\sig(v)\right| < \infty$.
We then observe that
\begin{align}
I''_{h,1} &\le 
h\bigg( \bigg|\lam+\frac{1}{2}\bigg| \int_{(1,\infty)} \vp_{0}(vh) \frac{v}{\eta^{2}+v^{2}}dv
+ \int_{(1,\infty)} \vp_{0}(vh) \frac{v}{\sqrt{\eta^{2}+v^{2}}}|\sig(v)|dv \bigg) \nn\\
&\lesssim 
h\bigg( \bigg|\lam+\frac{1}{2}\bigg| \int_{(1,\infty)} \vp_{0}(vh) \frac{v}{\eta^{2}+v^{2}}dv
+ \int_{(1,\infty)} v^{-2}dv \bigg) \nn\\
&\lesssim h\bigg( \bigg|\lam+\frac{1}{2}\bigg| \int_{(1,\infty)} \vp_{0}(vh) \frac{v}{\eta^{2}+v^{2}}dv + 1 \bigg).
\nonumber
\end{align}
Further, using the integration by parts and the change of variables we derive
\begin{align}
\int_{(1,\infty)} \vp_{0}(vh) \frac{v}{\eta^{2}+v^{2}}dv
&= -\frac{1}{2}e^{-h}\log(\eta^{2}+1) + \frac{h}{2}\int_{(1,\infty)}e^{-vh}\log(\eta^{2}+v^{2})dv \nn\\
&\lesssim 1 + \int_{(h,\infty)}e^{-x}\log\bigg\{\eta^{2}+\bigg(\frac{x}{h}\bigg)^{2}\bigg\}dx 
\lesssim 1 + \log(1/h).
\nonumber
\end{align}
\end{itemize}
Summarizing the above computations we conclude that
\begin{equation}
\int_{(0,\infty)}\vp_{0}(u) \left| \p_{u}\psi_{h}(u) + 1 \right| du
\lesssim
\begin{cases}
h\log(1/h) & (\lam\ne -1/2) \\
h & (\lam= -1/2)
\end{cases}
\nonumber
\end{equation}
verifying \eqref{A_J.cf-2} with $r=0$. By Lemma \ref{key.lemma.cf} we obtain
\begin{equation}
\int \left|f_{h}(y)-\phi_{\beta}(y)\right|dy \lesssim h\log(1/h) \lesssim h^{a'}
\nonumber
\end{equation}
for any $a'\in(0,1)$.
\label{ex.gh}
}\qed\end{ex}

\subsection{Locally stable stochastic differential equation}

Now let us recall the underlying SDE model \eqref{SDE}.
Denote by $\overline{\Theta}$ the closure of $\Theta=\Theta_{\al}\times\Theta_{\gam}$.

\begin{ass}[Regularity of the coefficients]
\begin{enumerate}
\item The functions $a(\cdot,\al_{0})$ and $c(\cdot,\gam_{0})$ are globally Lipschitz and of class $\mcc^{2}(\mbbr)$, and $c(x,\gam)>0$ for every $(x,\gam)$.
\item $a(x,\cdot)\in\mcc^{3}(\Theta_{\al})$ and $c(x,\cdot)\in\mcc^{3}(\Theta_{\gam})$ for each $x\in\mbbr$.
\item $\ds{\sup_{\theta\in\overline{\Theta}} \bigg\{
\max_{0\le k\le 3}\max_{0\le l\le 2}\bigg(
\left|\p_{\al}^{k}\p_{x}^{l}a(x,\al)\right| 
+ \left|\p_{\gam}^{k}\p_{x}^{l}c(x,\gam)\right|\bigg) + c^{-1}(x,\gam)\bigg\} \lesssim 1+ |x|^{C}}$.
\end{enumerate}
\label{A_coeff}
\end{ass}
The standard theory (for example \cite[III \S 2c.]{JacShi03}) ensures that the SDE admits a unique strong solution 
as a functional of $X_{0}$ and the Poisson random measure driving $J$; in particular, each $X_{t}$ is $\mcf_{t}$-measurable.

\medskip

\begin{ass}[Identifiability]
The random functions $t\mapsto\left(a(X_{t},\al), c(X_{t},\gam)\right)$ and $t\mapsto\left(a(X_{t},\al_{0}), c(X_{t},\gam_{0})\right)$ 
on $[0,T]$ a.s. coincide if and only if $\theta=\tz$.
\label{A_iden}
\end{ass}

\section{Stable quasi-likelihood estimation}\label{sec_SQL}

\subsection{Heuristic for construction}\label{sec_SQL.heuristic}

To motivate our quasi-likelihood, we here present an informal heuristic argument. In what follows we abbreviate $\int_{t_{j-1}}^{t_{j}}$ as $\int_{j}$.
For a moment, we write $\pr_{\theta}$ for the image measures of $X$ given by \eqref{SDE}.
In view of the Euler approximation under $\pr_{\theta}$,
\begin{align}
X_{t_{j}}&=X_{t_{j-1}}+\int_{j}a(X_{s},\al)ds+\int_{j}c(X_{s-},\gam)dJ_{s}
\nonumber \\
&\approx X_{t_{j-1}}+a_{j-1}(\al)h+c_{j-1}(\gam)\D_{j}J,
\nn
\end{align}
from which we may expect that
\begin{equation}
\ep_{j}(\theta)=\ep_{n,j}(\theta):=\frac{\D_{j}X-ha_{j-1}(\al)}{h^{1/\beta}c_{j-1}(\gam)}
\approx h^{-1/\beta}\D_{j}J
\nn
\end{equation}
in an appropriate sense. It follows from the locally stable property \eqref{A_J1} that for each $n$ the random variables $\ep_{1}(\theta),\dots,\ep_{n}(\theta)$ will be approximately i.i.d. with common distribution $S_{\beta}$.

Now assume that the process $X$ admits a (time-homogeneous) transition Lebesgue density under $\pr_{\theta}$, say $p_{h}(x,y;\theta)dy=\pr_{\theta}(X_{h}\in dy|X_{0}=x)$, and let $\E^{j-1}_{\theta}$ denote the expectation operator under $\pr_{\theta}$ conditional on $\mcf_{t_{j-1}}$.
Then, we may consider the following twofold approximation of the conditional distribution $\mcl(X_{t_{j}}|X_{t_{j-1}})$:
\begin{align}
p_{h}(X_{t_{j-1}},X_{t_{j}};\theta)
&=\frac{1}{2\pi}\int\exp(-iuX_{t_{j}}) \int e^{iuy}p_{h}(X_{t_{j-1}},y;\theta)dy \nn\\
&\approx\frac{1}{2\pi}\int\exp(-iuX_{t_{j}})
\E_{\theta}^{j-1}\left[\exp\left\{iu(X_{t_{j-1}}+a_{j-1}(\al)h+c_{j-1}(\gam)\D_{j}J)\right\}\right]du \nn \\
&{}\qquad\text{(Euler approximation)}
\nn\\
&=\frac{1}{2\pi}\int\exp\left\{-iu(\D_{j}X-a_{j-1}(\al)h)\right\} \vp_{h}\big(c_{j-1}(\gam)h^{1/\beta}u\big)du
\nonumber \\
&=\frac{1}{c_{j-1}(\gam)h^{1/\beta}} \frac{1}{2\pi}\int\exp\{-iv\ep_{j}(\theta)\}\vp_{h}(v)dv
\nonumber \\
&=\frac{1}{c_{j-1}(\gam)h^{1/\beta}}f_{h}\left(\ep_{j}(\theta)\right)
\nonumber \\
&\approx \frac{1}{c_{j-1}(\gam)h^{1/\beta}}\phi_{\beta}\left(\ep_{j}(\theta)\right)
\qquad\text{(Locally stable approximation).}
\nn
\end{align}
This informal observation suggests to estimate $\theta_{0}$ by a maximizer of the random function
\begin{equation}
\sqllf_{n}(\theta) := \sumj \log\bigg(\frac{1}{c_{j-1}(\gam)h^{1/\beta}}\phi_{\beta}\left(\ep_{j}(\theta)\right)\bigg),
\label{sql_def}
\end{equation}
which we call the \textit{stable quasi-likelihood}. We then define the {\it stable quasi-maximum likelihood estimator (SQMLE)} as any element $\tes=(\aes,\ges)$ such that
\begin{align}
\tes &\in \argmax_{\theta\in\overline{\Theta}}\sqllf_{n}(\theta)
=\argmax_{\theta\in\overline{\Theta}}\sumj\bigg( -\log c_{j-1}(\gam) +\log\phi_{\beta}\left(\ep_{j}(\theta)\right)\bigg).
\label{sqmle_def}
\end{align}
Since we are assuming that $\overline{\Theta}$ is compact, there always exists at least one such $\tes$.
The heuristic argument for the SQMLE will be verified in Section \ref{sec_amn}.
The SQMLE is the non-Gaussian-stable counterpart to the Gaussian quasi-likelihood previously studied by \cite{Kes97} and \cite{Mas13as} for diffusion and L\'{e}vy driven SDE, respectively.

\medskip

\begin{rem}{\rm 
It may happen, though very rarely, that the density $f_{h}$ of $\mcl(h^{-1/\beta}J_{h})$ is explicit for each $h>0$.
The normal-inverse Gaussian $J$ \cite{Bar98}, which we will use for simulations in Section \ref{sec_sim.nig.J}, is such an example.
In that case, the approximation
\begin{equation}
p_{h}(X_{t_{j-1}},X_{t_{j}};\theta) \approx\frac{1}{c_{j-1}(\gam)h^{1/\beta}}f_{h}\left(\ep_{j}(\theta)\right)
\nonumber
\end{equation}
may result in a better quasi-likelihood since it precisely incorporates information of the driving noise.
Nevertheless and obviously, such an ``exact $\mcl(h^{-1/\beta}J_{h})$'' consideration much diminishes the target class of $J$, and going in this direction entails individual case studies.
}\qed\end{rem}

\subsection{Main result: Asymptotic mixed normality of SQMLE}\label{sec_amn}

For $\mcf$-measurable random variables $\mu=\mu(\omega)\in\mbbr^{p}$ and a.s. nonnegative definite $\Sig=\Sig(\omega)\in\mbbr^{p}\otimes\mbbr^{p}$, we denote by $MN_{p}(\mu,\Sig)$ the $p$-dimensional mixed normal distribution corresponding to the characteristic function
\begin{equation}
v\mapsto \E\bigg\{\exp\bigg(i\mu\cdot v -\frac{1}{2} v\cdot\Sig v\bigg)\bigg\}.
\nonumber
\end{equation}
That is to say, when $Y\sim MN_{p}(\mu,\Sig)$, $Y$ is defined on an extension of the original probability space $(\Omega,\mcf,\pr)$,
and is 
equivalent in $\mcf$-conditional distribution to a random variable $\mu + \Sig^{1/2}Z$ for $Z\sim N_{p}(0,I_{p})$ independent of $\mcf$, where $I_{p}$ denotes the $p$-dimensional identity matrix. Such an (orthogonal) extension of the underlying probability space is always possible.

We introduce the two bounded smooth continuous functions:
\begin{equation}
g_{\beta}(y):=\p_{y}\log\phi_{\beta}(y)=\frac{\p\phi_{\beta}}{\phi_{\beta}}(y),\qquad 
k_{\beta}(y):=1+yg_{\beta}(y).
\nonumber
\end{equation}
We see that $\int g_{\beta}(y)\phi_{\beta}(y)dy = \int k_{\beta}(y)\phi_{\beta}(y)dy=0$, and that $\int g_{\beta}(y) f_{h}(y)dy=0$ as soon as $\int |g_{\beta}(y)| f_{h}(y)dy<\infty$ because $f_{h}$ is symmetric. We also write
\begin{align}
& C_{\al}(\beta) = \int g_{\beta}^{2}(y)\phi_{\beta}(y)dy, \qquad C_{\gam}(\beta)=\int k_{\beta}^{2}(y)\phi_{\beta}(y)dy, \nn\\
& \Sig_{T,\al}(\tz) =\frac{1}{T}\int_{0}^{T}\frac{\{\p_{\al}a(X_{t},\al_{0})\}^{\otimes 2}}{c^{2}(X_{t},\gam_{0})}dt, \qquad 
\Sig_{T,\gam}(\gam_{0}) = \frac{1}{T}\int_{0}^{T}\frac{\{\p_{\gam}c(X_{t},\gam_{0})\}^{\otimes 2}}{c^{2}(X_{t},\gam_{0})}dt.
\nn
\end{align}
The asymptotic behavior of the SQMLE defined through \eqref{sql_def} and \eqref{sqmle_def} is given in the next theorem, which is the main result of this paper. 

\begin{thm}
Suppose that Assumptions \ref{A_J} with $\beta\in[1,2)$, \ref{A_coeff}, and \ref{A_iden} hold. 
Then we have
\begin{equation}
\left(\sqrt{n}h^{1-1/\beta}(\hat{\al}_{n}-\al_{0}),~\sqrt{n}(\hat{\gam}_{n}-\gam_{0})\right) 
\cil MN_{p}\left( 0,\, \Gam_{T}(\tz;\beta)^{-1} \right),
\label{aqmle_amn}
\end{equation}
where
\begin{align}
\Gam_{T}(\tz;\beta) &:= 
\begin{pmatrix}
C_{\al}(\beta)\Sig_{T,\al}(\tz) & 0 \\
0 & C_{\gam}(\beta)\Sig_{T,\gam}(\gam_{0})
\end{pmatrix}.
\nonumber
\end{align}
\label{sqmle.iv_bda_thm}
\end{thm}

In Section \ref{sec_ergodic}, we will deduce the large-time counterpart to Theorem \ref{sqmle.iv_bda_thm} under the ergodicity. In that case the asymptotic distribution is not mixed normal but normal, with the asymptotic covariance matrix taking a completely analogous form.

\medskip

Below we list some immediate consequences of Theorem \ref{sqmle.iv_bda_thm} and some related remarks worth being mentioned.

\begin{enumerate}
\item The asymptotic distribution of $\tes$ is normal if both $x\mapsto \frac{\p_{\gam}c(x,\gam_{0})}{c(x,\gam_{0})}$ and $x\mapsto \frac{\p_{\al}a(x,\al_{0})}{c(x,\gam_{0})}$ are non-random;
this is the case if $X$ is a {\lp}.

\item The estimators $\aes$ and $\ges$ are asymptotically orthogonal, whereas not necessarily independent due to possible non-Gaussianity in the limit.

\item For $\beta\in(1,2)$, we can rewrite \eqref{aqmle_amn} as (recall \eqref{fixed.T})
\begin{align}
& \left( n^{1/\beta-1/2}(\hat{\al}_{n}-\al_{0}),~\sqrt{n}(\hat{\gam}_{n}-\gam_{0}) \right)
\nn\\
& \cil MN_{p}\left(0,\,
\diag\big(T^{-2(1-1/\beta)}\{C_{\al}(\beta)\Sig_{T,\al}(\tz)\}^{-1},\, \{C_{\gam}(\beta)\Sig_{T,\gam}(\gam_{0})\}^{-1}\big) \right).
\nn
\end{align}
If fluctuation of $X$ is virtually stable in the sense that both of the random time averages $\Sig_{T,\al}(\theta_{0})$ and $\Sig_{T,\gam}(\gam_{0})$ 
do not vary so much with the terminal sampling time $T$, then, due to the factor ``$T^{-2(1-1/\beta)}$'', 
the asymptotic covariance matrix of $\aes$ would tend to get smaller (resp. larger) in magnitude for a larger (resp. smaller) $T$.
This feature with respect to $T$ is non-asymptotic.

\item Of special interest is the locally Cauchy case ($\beta=1$), where $\sqllf_{n}$ is fully explicit:
\begin{equation}
\sqllf_{n}(\theta) = -\sumj \bigg\{\log(\pi h) + \log c_{j-1}(\gam) + \log\left(1+\ep_{j}^{2}(\theta)\right)\bigg\}.
\nn
\end{equation}
In this case,
\begin{align}
& \left( \sqrt{n}(\hat{\al}_{n}-\al_{0}),~\sqrt{n}(\hat{\gam}_{n}-\gam_{0}) \right)
\nn\\
&\cil MN_{p}\bigg(0,\, \diag\bigg\{
\bigg(\frac{1}{2T}\int_{0}^{T}\frac{\{\p_{\al}a(X_{t},\al_{0})\}^{\otimes 2}}{c(X_{t},\gam_{0})^{2}}dt\bigg)^{-1},
~\bigg(\frac{1}{2T}\int_{0}^{T}\frac{\{\p_{\gam}c(X_{t},\gam_{0})\}^{\otimes 2}}{c(X_{t},\gam_{0})^{2}}dt
\bigg)^{-1}\bigg\} \bigg).
\nonumber
\end{align}
This formally extends the i.i.d. model from the location-scale Cauchy population, where we have $\sqrt{n}$-asymptotic normality for the maximum-likelihood estimator.
The Cauchy quasi-likelihood has been also investigated in the robust-regression literature; see \cite{MizMul99} and \cite{MizMul02} for a breakdown-point result in some relevant models.
It would be interesting to study their SDE-model counterparts.

\end{enumerate}

\medskip

\begin{rem}{\rm 
As indicated by one of the anonymous reviewers, it would be possible to follow the proof of Theorem \ref{sqmle.iv_bda_thm} for time-inhomogeneous coefficients, say
\begin{equation}
dX_{t}=a(t,X_{t},\al)dt+c(t,X_{t-},\gam)dJ_{t},
\nonumber
\end{equation}
under appropriate regularity conditions on $(t,x,\theta)\mapsto (a(t,x,\al), c(t,x,\gam))$; see \cite{GenJac93} for the case of diffusion.
Further, based on the general criterion for the stable convergence, we could deduce a slightly more general statement where the SQMLE has non-trivial asymptotic bias.
In general, it is however impossible to make an explicit bias correction in a unified manner without specific information of $f_{h}$ (hence of the {\lm} $\nu$).
Even when we have a full parametric form of $\nu$, it may contain a parameter which cannot be consistently estimated unless $T_{n}\to\infty$; see \cite{Mas15LM} for specific examples.
}\qed\end{rem}

\begin{rem}{\rm 
The asymptotic efficiency in the sense of Haj\'{e}k-Le Cam-Jeganathan is of primary theoretical importance (see \cite{vdV98}).
Compared with the diffusion case studied in \cite{Gob01} and \cite{Gob02}, asymptotic-efficiency phenomena for the L\'{e}vy driven SDE \eqref{SDE} when observing \eqref{disc.X} have been less well-known. Nevertheless, for the classical local asymptotic normality property results when $X$ is a {\lp}, one can consult \cite{Mas15LM} for several explicit case studies, and to \cite{IvaKulMas15} for a general locally stable {\lp es}. Moreover, \cite{CleGlo15} and the recent preprint \cite{CleGloNgu17} proved the local asymptotic mixed normality property about the unknown parameters especially when $c(x,\gam)$ is a constant and the {\lm} $\nu$ has a bounded support with a stable-like behavior near the origin. Importantly, the model settings of \cite{CleGlo15} and \cite{CleGloNgu17} can be covered by ours (Example \ref{ex.cgn}), so that the asymptotic efficiency of our SQMLE is assured.
In view of their result and just like the fact that the Gaussian QMLE is asymptotically efficient for diffusions, it seems quite promising that the proposed SQMLE is asymptotically efficient for the general class of SDE \eqref{SDE} driven by a locally $\beta$-stable {\lp}.
\label{rem_asymp.efficiency}
}\qed\end{rem}

Here is a variant of Theorem \ref{sqmle.iv_bda_thm}.

\begin{thm}
Suppose that Assumptions \ref{A_J.cf} holds with $\beta\in[1,2)$ and
\begin{equation}
\sqrt{n}(\ep_{\psi}(h) \vee h^{(\del/\beta)\wedge 1})^{\frac{\beta}{\beta+r}} \to 0.
\nonumber
\end{equation}
Suppose also that Assumptions \ref{A_coeff} and \ref{A_iden} hold, and that
\begin{equation}
\int_{|z|>1}|z|^{q}\nu(dz)<\infty
\nonumber
\end{equation}
for every $q>0$. Then we have \eqref{aqmle_amn}.
\label{sqmle.iv_bda_thm.cf}
\end{thm}

To state a corollary to Theorems \ref{sqmle.iv_bda_thm} and \ref{sqmle.iv_bda_thm.cf}, we introduce the following statistics:
\begin{align}
& \hat{\Sig}_{T,\al,n} := \frac{1}{n}\sumj \frac{\{\p_{\al}a_{j-1}(\aes)\}^{\otimes 2}}{c^{2}_{j-1}(\ges)}, \qquad 
\hat{\Sig}_{T,\gam,n} := \frac{1}{n}\sumj \frac{\{\p_{\gam}c_{j-1}(\ges)\}^{\otimes 2}}{c^{2}_{j-1}(\ges)}.
\nonumber
\end{align}
It turns out in the proof that the quantity $\left((\sqrt{n}h^{1-1/\beta})^{-1}\p_{\al}\sqllf_{n}(\tz),\, n^{-1/2}\p_{\gam}\sqllf_{n}(\tz)\right)$, the normalized quasi-score, $\mcf$-stably converges in distribution (Section \ref{sec_score.scle.proof}), from which the Studentization via the continuous-mapping theorem is straightforward:

\begin{cor}
Under the assumptions of either Theorem \ref{sqmle.iv_bda_thm} or Theorem \ref{sqmle.iv_bda_thm.cf}, we have
\begin{equation}
\left( \big(C_{\al}(\beta)\hat{\Sig}_{T,\al,n}\big)^{1/2} \sqrt{n}h^{1-1/\beta}(\hat{\mu}_{n}-\mu_{0}), \, 
\big(C_{\gam}(\beta)\hat{\Sig}_{T,\gam,n}\big)^{1/2} \sqrt{n}(\hat{\sig}_{n}-\sig_{0}) \right) \cil N_{p}(0,I_{p}).
\label{aqmle_amn_sn}
\end{equation}
\label{sqmle.iv_bda_thm.cor}
\end{cor}

Table \ref{table_qls} summarizes the rates of convergence of the $\beta$-stable maximum quasi-likelihood estimators with $\beta\le 2$, when the target SDE model is
\begin{equation}
dX_{t} = a(X_{t},\al)dt + c(X_{t-},\gam) dZ_{t}
\label{SDE.Z}
\end{equation}
for a driving {\lp} $Z$ with the correctly specified coefficient $(a,c)$; again, note that the Gaussian QMLE requires $T_{n}\to\infty$, which is not necessary for the SQMLE.
We refer to \cite{Mas13spa} for a handy statistic for testing the case (i) against the case (ii) based on the Gaussian QMLE.

\begin{table}[htb]
\centering
\begin{tabular}{lclccclc}
\hline\\[-2mm]
Quasi-likelihood  & & Driving {\lp} $Z$ & & \multicolumn{2}{c}{Rates of convergence} & & Ref. \\
  		     & & & & $\aes$ & $\ges$ & & \\[1mm] \hline
\\[-3mm]
(i) Gauss & & {\wp} & & $\sqrt{nh}$ & $\sqrt{n}$ & & \cite{Kes97} \\[-5pt] 
\\
(ii) Gauss & & {\lp} with jumps & & $\sqrt{nh}$ & $\sqrt{nh}$ & & \cite{Mas11}, \cite{Mas13as} \\[-5pt] 
\\
(iii) Non-Gaussian stable & & Locally $\beta$-stable {\lp} & & $\sqrt{n}h^{1-1/\beta}$ & $\sqrt{n}$ & &  \\[1mm] 
\hline
\end{tabular}
\bigskip
\caption{Comparison of the Gaussian ($\beta=2$) and non-Gaussian stable ($\beta\in[1,2)$) QMLE for the SDE \eqref{SDE.Z},
where the coefficient $(a,c)$ is correctly specified: Case (iii) is the contribution of this paper.}
\label{table_qls}
\end{table}

\begin{rem}{\rm 
We have been focusing on $\beta\ge 1$. For $\beta\in(0,1)$, direct use of the Euler scheme would spoil the proofs in Section \ref{sec_proofs} because small-time variation of $X$ by the noise term is dominated by that of the trend coefficient $a(x,\al)$.
In this case, direct use of the present stable quasi-likelihood based on the mere Euler scheme would be inadequate.
It would be necessary to take the drift structure into account more precisely, as in the trajectory-fitting estimator studied in \cite{Mas05jjss}.
}\qed\end{rem}

\subsection{Ergodic case under long-time asymptotics}\label{sec_ergodic}

In this section, instead of the bounded-domain asymptotics \eqref{fixed.T} we consider the sampling design
\begin{equation}
T_{n}\to\infty \quad\text{and}\quad \sqrt{n} h^{2-1/\beta}\to 0,
\label{large.T}
\end{equation}
which still implies that $\sqrt{n}h^{1-1/\beta}\to\infty$ when $\beta\in[1,2)$; for example, it suffices to have $T_{n}\to\infty$ and $nh^{2}\to 0$. Theorem \ref{sqmle_ergo_thm} below shows that under the ergodicity of $X$ the asymptotic normality of the SQMLE \eqref{sqmle_def} holds. The logic of construction of the stable quasi-likelihood is completely the same as in Section \ref{sec_SQL.heuristic}. 

We will adopt Assumption \ref{A_J.cf} for the structural assumptions on $J$, and impose Assumption \ref{A_coeff} without any change.

\begin{ass}[Stability]
\begin{enumerate}
\item There exists a unique invariant measure $\pi_{0}$ such that
\begin{equation}
\frac{1}{T}\int_{0}^{T}g(X_{t})dt \cip \int g(x)\pi_{0}(dx),\qquad T\to\infty,
\label{ergo_thm}
\end{equation}
for every measurable function $g$ of at most polynomial growth.
\item $\ds{\sup_{t\in\mbbrp}\E(|X_{t}|^{q})<\infty}$ for every $q>0$.
\end{enumerate}
\label{A_ergo}
\end{ass}

The property \eqref{ergo_thm} follows from the convergence $\|P_{t}(x,\cdot)-\pi_{0}(\cdot)\|_{TV}\to 0$ as $t\to\infty$ for each $x\in\mbbr$, where $P_{t}(x,dy)$ denotes the transition function of $X$ under the true measure and $\|\mu\|_{TV}$ the total variation norm of a signed measure $\mu$.
The next lemma, which directly follows from \cite[Proposition 5.4]{Mas13as}, provides a set of sufficient conditions for Assumption \ref{A_ergo}.

\begin{lem}
Let $X$ be given by \eqref{SDE} and suppose that $\nu(\{z\ne 0;\,|z|\le\ep\})>0$ for every $\ep>0$.
Further, assume the following conditions.
\begin{enumerate}
\item Both $a(\cdot,\al_{0})$ and $c(\cdot,\gam_{0})$ are of class $\mcc^{1}(\mbbr)$ and globally Lipschitz, and $c$ is bounded.
\item $c(x,\gam_{0})\ne 0$ for every $x$.
\item $\E(J_{1})=0$ and either one of the following conditions holds:
\begin{itemize}
\item $\E(|X_{0}|^{q})<\infty$ and $\int_{|z|>1}|z|^{q}\nu(dz)<\infty$ for every $q>0$, and
\begin{equation}
\limsup_{|x|\to\infty}\frac{a(x,\al_{0})}{x}<0.
\nonumber
\end{equation}
\item $\E(e^{q|X_{0}|})<\infty$ and $\int_{|z|>1}e^{q|z|}\nu(dz)<\infty$ for some $q>0$, and
\begin{equation}
\limsup_{|x|\to\infty}{\rm sgn}(x)a(x,\al_{0})<0.
\nonumber
\end{equation}
\end{itemize}
\end{enumerate}
Then Assumption \ref{A_ergo} holds.
\label{lem_ergo}
\end{lem}

We also need a variant of Assumption \ref{A_iden}.

\begin{ass}[Model identifiability]
The functions $x\mapsto\left(a(x,\al), c(x,\gam)\right)$ and $x\mapsto\left(a(x,\al_{0}), c(x,\gam_{0})\right)$ coincide $\pi_{0}$-a.e. if and only if $\theta=\tz$.
\label{A_ergo.iden}
\end{ass}

\begin{thm}
Suppose that Assumptions \ref{A_J.cf} holds with
\begin{equation}
\sqrt{n}(\ep_{\psi}(h) \vee h^{(\del/\beta)\wedge 1})^{\frac{\beta}{\beta+r}} \to 0.
\nonumber
\end{equation}
Suppose also that Assumptions \ref{A_coeff}, \ref{A_ergo}, and \ref{A_ergo.iden} hold. Then, under \eqref{large.T} we have
\begin{equation}
\left(
\sqrt{n}h^{1-1/\beta}(\hat{\al}_{n}-\al_{0}), \, \sqrt{n}(\hat{\gam}_{n}-\gam_{0})
\right)
\cil N_{p}\left(0,\, \diag\left(V_{\al}(\theta_{0};\beta)^{-1},V_{\gam}(\theta_{0};\beta)^{-1}\right)\right),
\nn
\end{equation}
where
\begin{align}
V_{\al}(\theta_{0};\beta)&:= C_{\al}(\beta)\int\frac{\{\p_{\al}a(x,\al_{0})\}^{\otimes 2}}{c(x,\gam_{0})^{2}}\pi_{0}(dx), \nonumber \\
V_{\gam}(\theta_{0};\beta)&:= C_{\gam}(\beta)\int\frac{\{\p_{\gam}c(x,\gam_{0})\}^{\otimes 2}}{c(x,\gam_{0})^{2}}\pi_{0}(dx).
\nonumber
\end{align}
\label{sqmle_ergo_thm}
\end{thm}

The proof of Theorem \ref{sqmle_ergo_thm} will be sketched in Section \ref{sec_ergo.an.proof}.
Obviously, Studentization is possible just the same as in Corollary \ref{sqmle.iv_bda_thm.cor}.
Again we remark that Assumption \ref{A_J.cf} could be replaced with any other one implying the convergences \eqref{llt.imp.1} and \eqref{llt.imp.2}.

\begin{rem}{\rm 
We have $C_{\al}(2)=1$ and $C_{\gam}(2)=2$, hence taking $\beta=2$ in the expressions of $V_{\al}(\tz)$ and $V_{\gam}(\tz)$ formally results in the asymptotic Fisher information matrices for the diffusion case \cite{Kes97} (also \cite{UchYos12}).
}\qed\end{rem}

\section{Numerical experiments}\label{sec_simulations}

For simulations, we use the nonlinear data-generating SDE
\begin{equation}
dX_{t} = \bigg(\al_{1}X_{t}+ \frac{\al_{2}}{1+X_{t}^{2}}\bigg) dt 
+ \exp\big\{\gam_{1}\cos(X_{t}) + \gam_{2}\sin(X_{t})\big\} dJ_{t},\qquad X_{0}=0,
\nn
\end{equation}
with $\theta=(\al_{1},\al_{2},\gam_{1},\gam_{2})$ and $J$ being either:
\begin{itemize}
\item The normal inverse Gaussian {\lp} (Example \ref{ex.gh});
\item The $1.5$-stable {\lp es} (Example \ref{ex.cgn}).
\end{itemize}
The setting is a special case of $a(x,\al)=\al_{1}a_{1}(x)+\al_{2}a_{2}(x)$ and $c(x,\gam)=\exp\{\gam_{1}c_{1}(x)+\gam_{2}c_{2}(x)\}$, for which the asymptotic covariances of the $\sqrt{n}h^{1-1/\beta}(\hat{\al}_{k,n}-\al_{k,0})$ and $\sqrt{n}(\hat{\gam}_{l,n}-\gam_{l,0})$ are given by the inverses of
\begin{equation}
C_{\al}(\beta)\frac{1}{T}\int_{0}^{T}\frac{a_{k}^{2}(X_{t})}{c^{2}(X_{t},\gam_{0})}dt\qquad\text{and}\qquad
C_{\gam}(\beta)\frac{1}{T}\int_{0}^{T}c_{l}^{2}(X_{t})dt,
\nonumber
\end{equation}
respectively.

\subsection{Normal inverse Gaussian driver}\label{sec_sim.nig.J}

Let $J$ be an normal inverse Gaussian (NIG) {\lp} such that
\begin{equation}
\mcl(J_{t})=NIG(\eta,0,t,0),
\nonumber
\end{equation}
where $\eta>0$ may be unknown.
This is a special case of the generalized hyperbolic {\lp} considered in Example \ref{ex.gh} with $\lam=-1/2$.
The numerical results below show that the SQMLE effectively works.

We set $\eta=5$ or $10$; the bigger $\eta$ leads to a lighter tail of $\D_{j}J$, hence a seemingly more ``diffusion-like'' sample-path behavior. Also, we set the terminal time $T=1$ or $5$. For each pair $(\eta,T)$, we proceed as follows.
\begin{itemize}
\item First we apply the Euler scheme for the true model with discretization step size being $\D:=T/(3000\times 50)$.
\item Then we thin generated single path $(X_{k\D})_{k=0}^{3000\times 50}$ to pick up $(X_{jh})_{j=0}^{n}$ with $h=\D \times 50\times 6$, $\D\times 50\times 3$ and $\D\times 50$ for $n=500$, $1000$ and $3000$, respectively.
\end{itemize}
Here, the number ``$50$'' of generation over each sub-periods $(t_{j-1},t_{j}]$ reflects that $X$ virtually continuously evolves as time goes along, though not observable.
We independently repeat the above procedures for $L=1000$ times to get $1000$ independent estimates $\tes=(\aes,\ges)$, based on which boxplots and histograms for Studentized versions are computed (Corollary \ref{sqmle.iv_bda_thm.cor}). We used the function \texttt{optim} in R \cite{R13}, and in each optimization for $l=1,\dots,L$ we generated independent uniform random numbers $\mathrm{Unif}(\al_{k,0}-10,\al_{k,0}+10)$ and $\mathrm{Unif}(\gam_{0,l}-10,\gam_{0,l}+10)$ for initial values for searching $\al_{k}$ and $\gam_{l}$, respectively.

The two cases are conducted:
\begin{itemize}
\item[(i)] We know {\it a priori} that $\al_{2,0}=\gam_{2,0}=0$, and the estimation target is $\tz=(\al_{1,0},\gam_{1,0})=(-1,1.5)$;
\item[(ii)] Estimation target is $\tz=(\al_{1,0},\al_{2,0},\gam_{1,0},\gam_{2,0})=(-1,1,1.5,0.5)$.
\end{itemize}
From the obtained simulation results, we observed the following.
\begin{itemize}
\item Figures \ref{fig:bp1*1_4panels} and \ref{fig:hist1*1}: case of (i).
\begin{itemize}
\item The boxplots show the clear tendency that estimation accuracy for each $T$ gets better for larger $n$.
\item The histograms show overall good standard normal approximations; the straight line in red is the target standard normal density.
It is observed that the estimation performance of $\ges$ gets worse if the nuisance parameter $\eta$ gets larger from $5$ to $10$. In particular, for the cases where $\eta=10$ we can see downward bias of the Studentized $\ges$, although it disappears as $n$ increases.
\end{itemize}
Overall, we see very good finite-sample performance of $\aes$, while that of $\ges$ may be affected to some extent by the value of $(T,\eta)$.
As in the case of estimation of the diffusion coefficient for a diffusion type processes, for better estimation of $\gam$ the value $T$ should not be so large, equivalently $h$ should not be so large.
\item Figures \ref{fig:bp2*2_4panels} and Figures \ref{fig:hist2*2-1}--\ref{fig:hist2*2-2}: case of (ii).
\begin{itemize}
\item General tendencies are the same as in the previous case: for each $T$, estimate accuracy gets better for larger $n$, while the gain of estimation accuracy for larger $n$ is somewhat smaller compared with the previous case.
\item The histograms show that, compared with the previous case, the Studentized estimators are of heavier tails and asymptotic bias associated with  $\ges$ severely remains, especially for $(T,\eta)=(5,10)$ (Figure \ref{fig:hist2*2-2}), unless $n$ is large enough.
\end{itemize}
\end{itemize}

\medskip



\begin{figure}[htbp]
 \begin{minipage}{0.49\hsize}
  \begin{center}
   \includegraphics[scale=0.32]{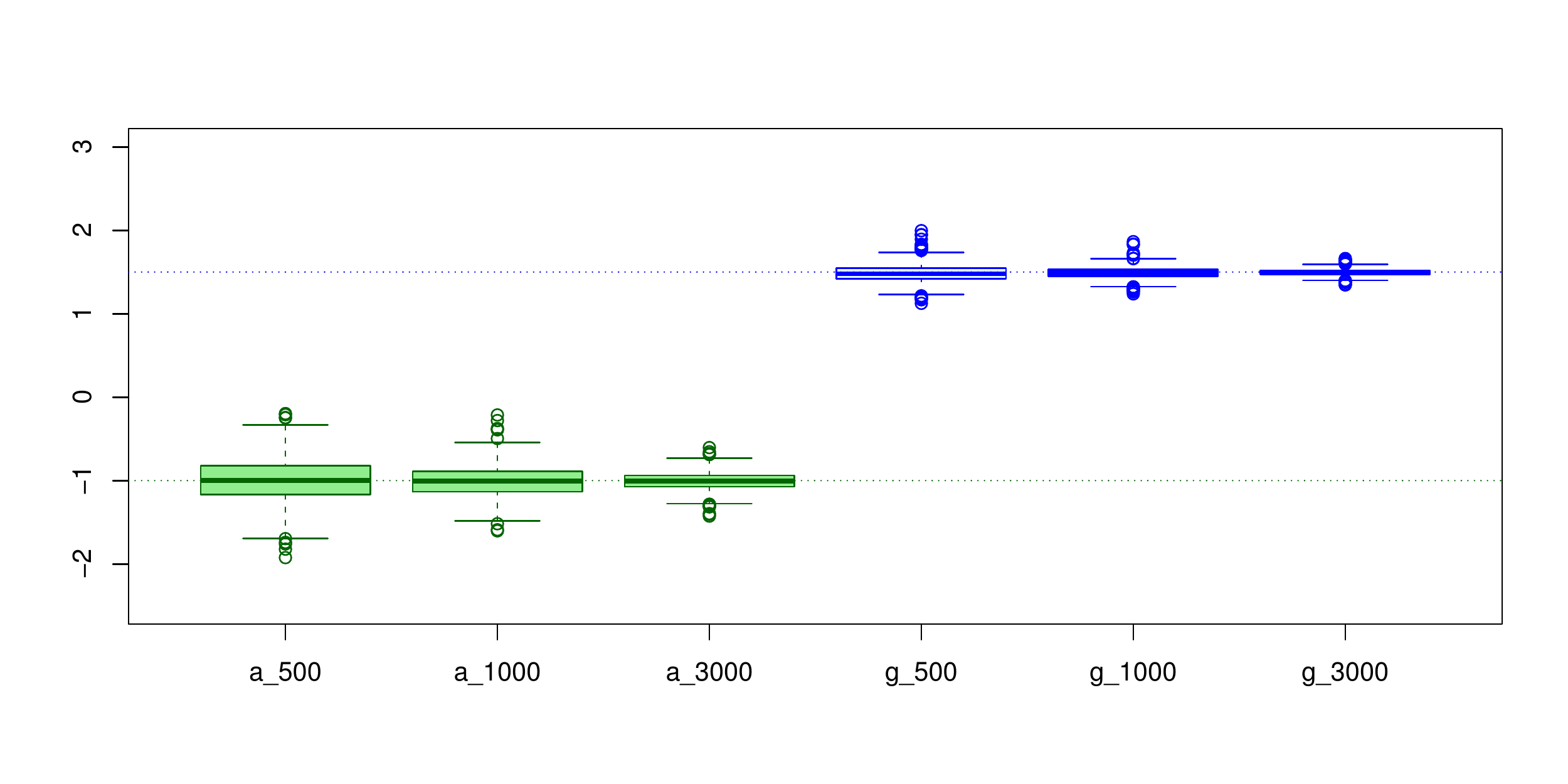}
  \end{center}
 \end{minipage}
 \begin{minipage}{0.49\hsize}
  \begin{center}
   \includegraphics[scale=0.32]{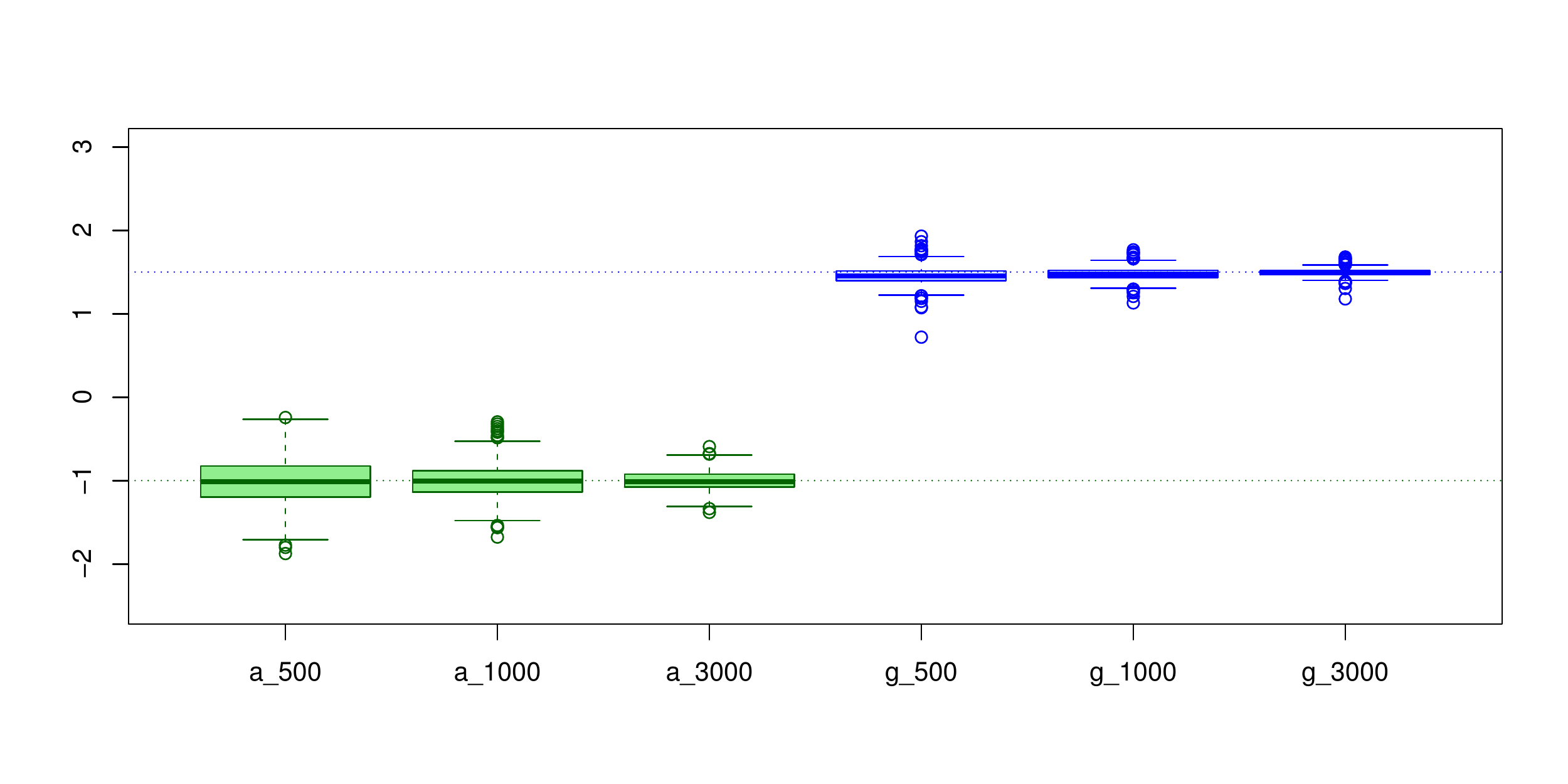}
  \end{center}
 \end{minipage}
 \begin{minipage}{0.49\hsize}
  \begin{center}
   \includegraphics[scale=0.32]{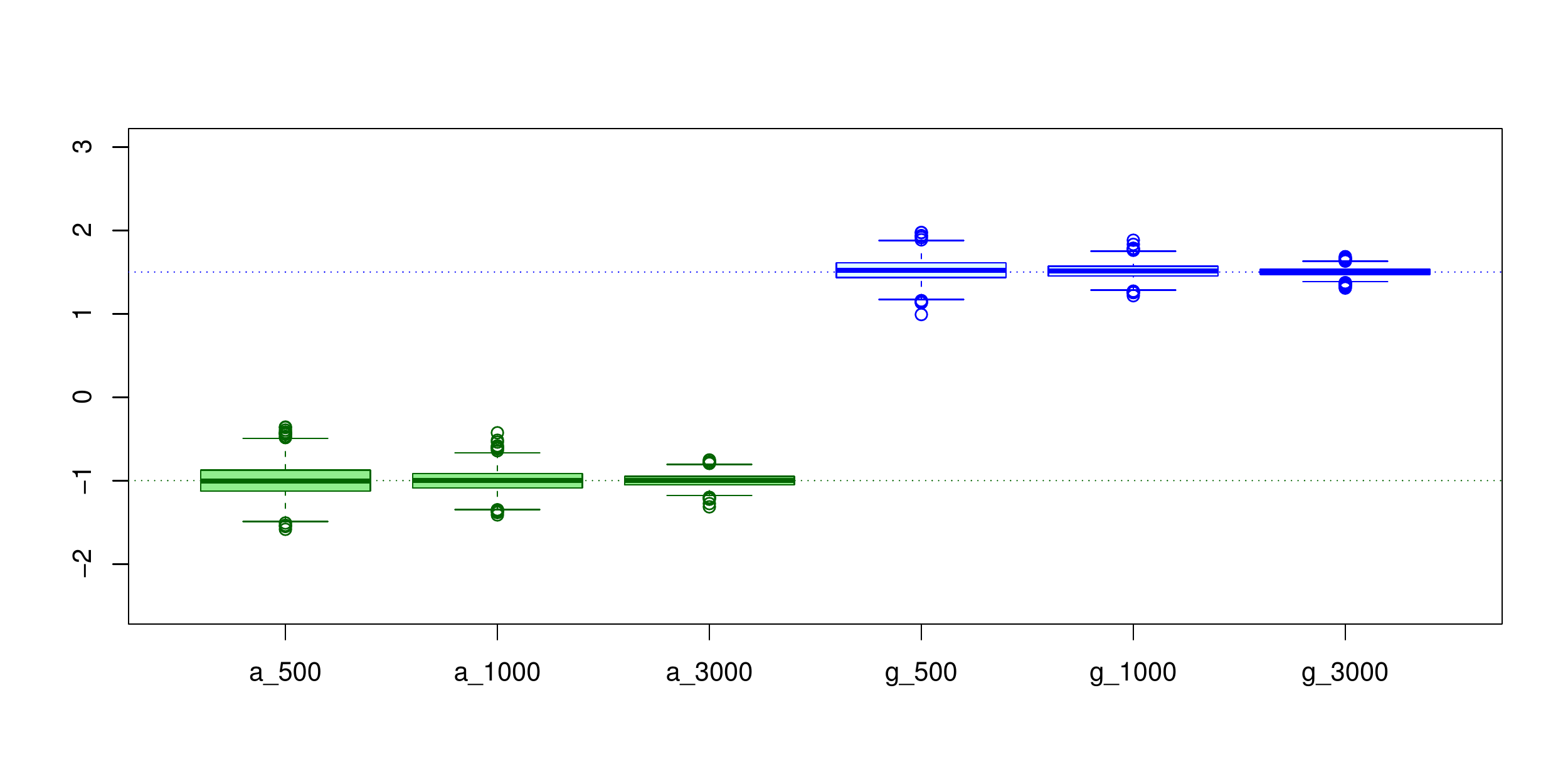}
  \end{center}
 \end{minipage}
 \begin{minipage}{0.49\hsize}
  \begin{center}
   \includegraphics[scale=0.32]{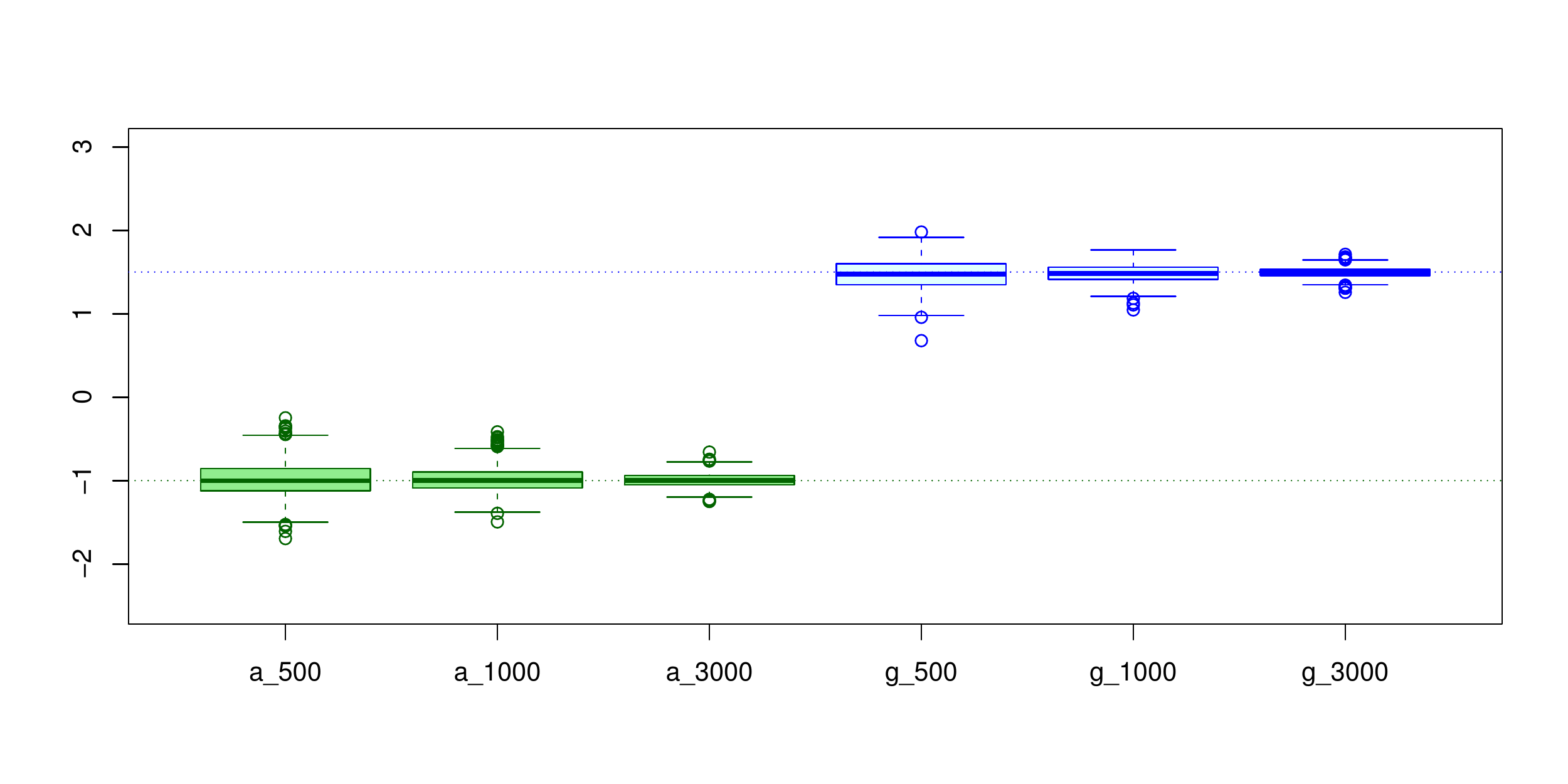}
  \end{center}
 \end{minipage}
 \caption{NIG-$J$ example. 
 Boxplots of $1000$ independent estimates $\aes$ (green) and $\ges$ (blue) for $n=500$, $1000$, $3000$; $(T,\eta)=(1,5)$ (upper left), $(T,\eta)=(1,10)$ (upper right), $(T,\eta)=(5,5)$ (lower left), and $(T,\eta)=(5,10)$ (lower right).}
 \label{fig:bp1*1_4panels}
\end{figure}

\medskip


\begin{figure}[htbp]
 \begin{minipage}{0.49\hsize}
  \begin{center}
   \includegraphics[scale=0.3]{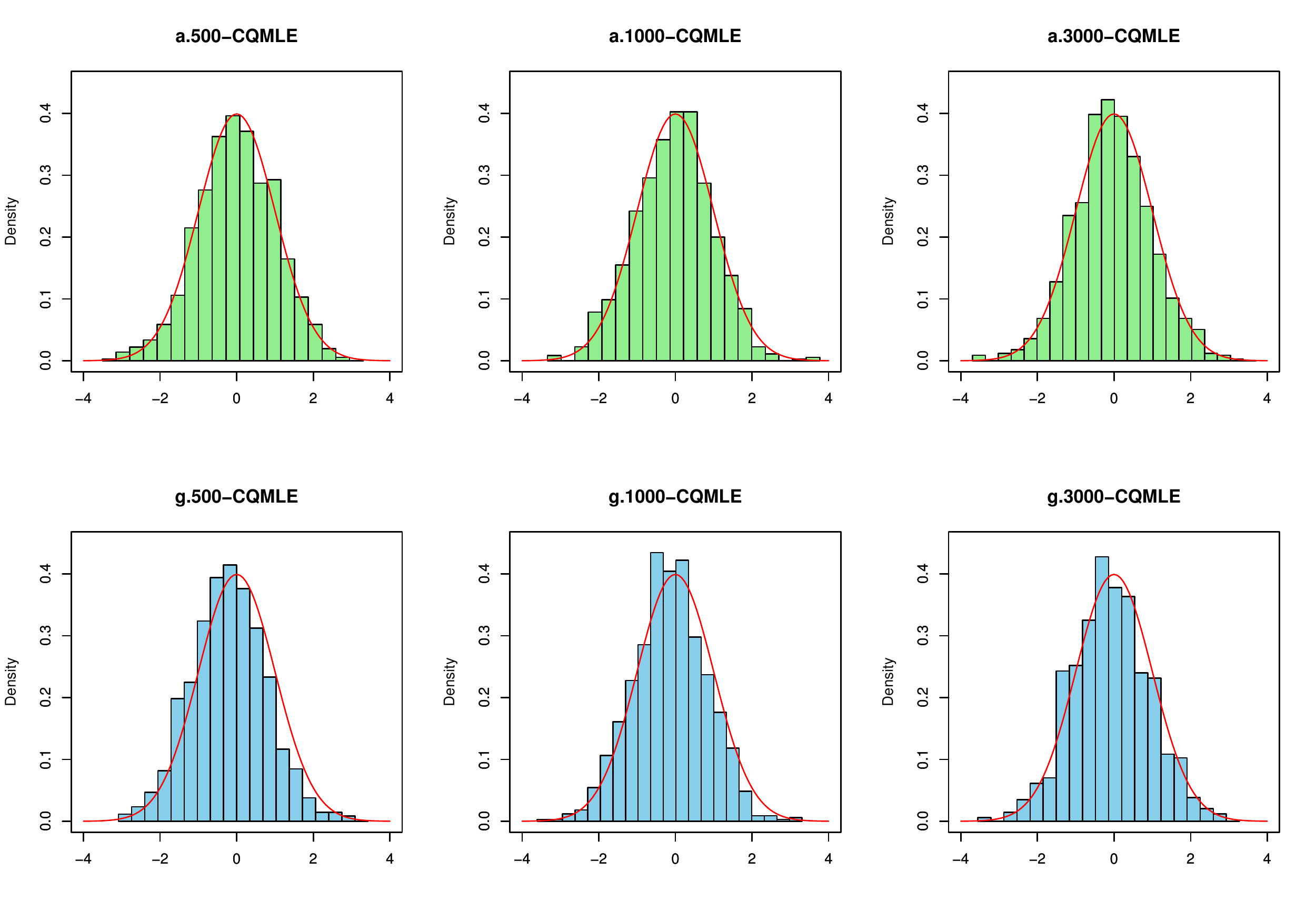}
  \end{center}
 \end{minipage}
 \begin{minipage}{0.49\hsize}
  \begin{center}
   \includegraphics[scale=0.3]{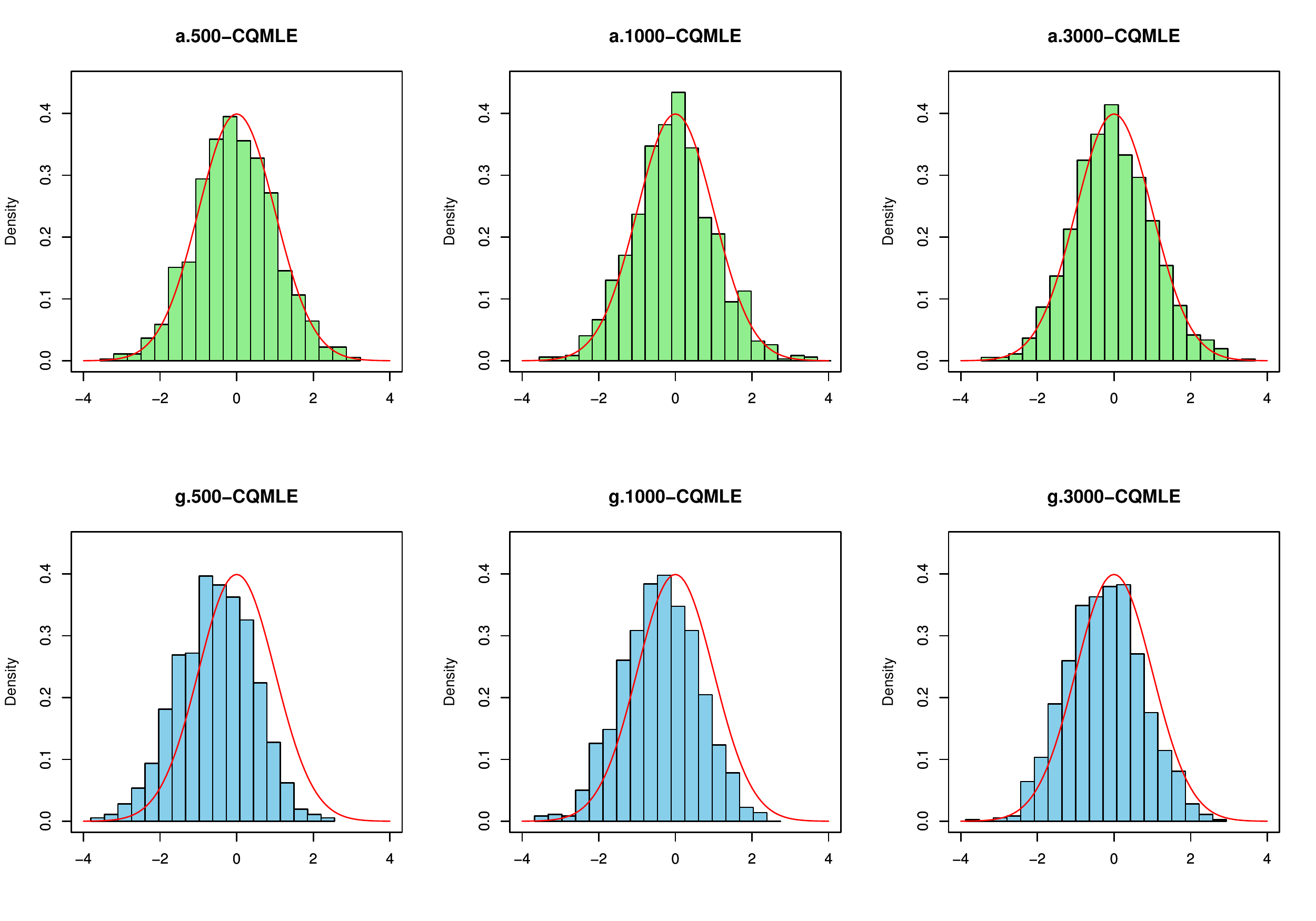}
  \end{center}
 \end{minipage}
	\bigskip
	
\begin{minipage}{0.49\hsize}
  \begin{center}
   \includegraphics[scale=0.3]{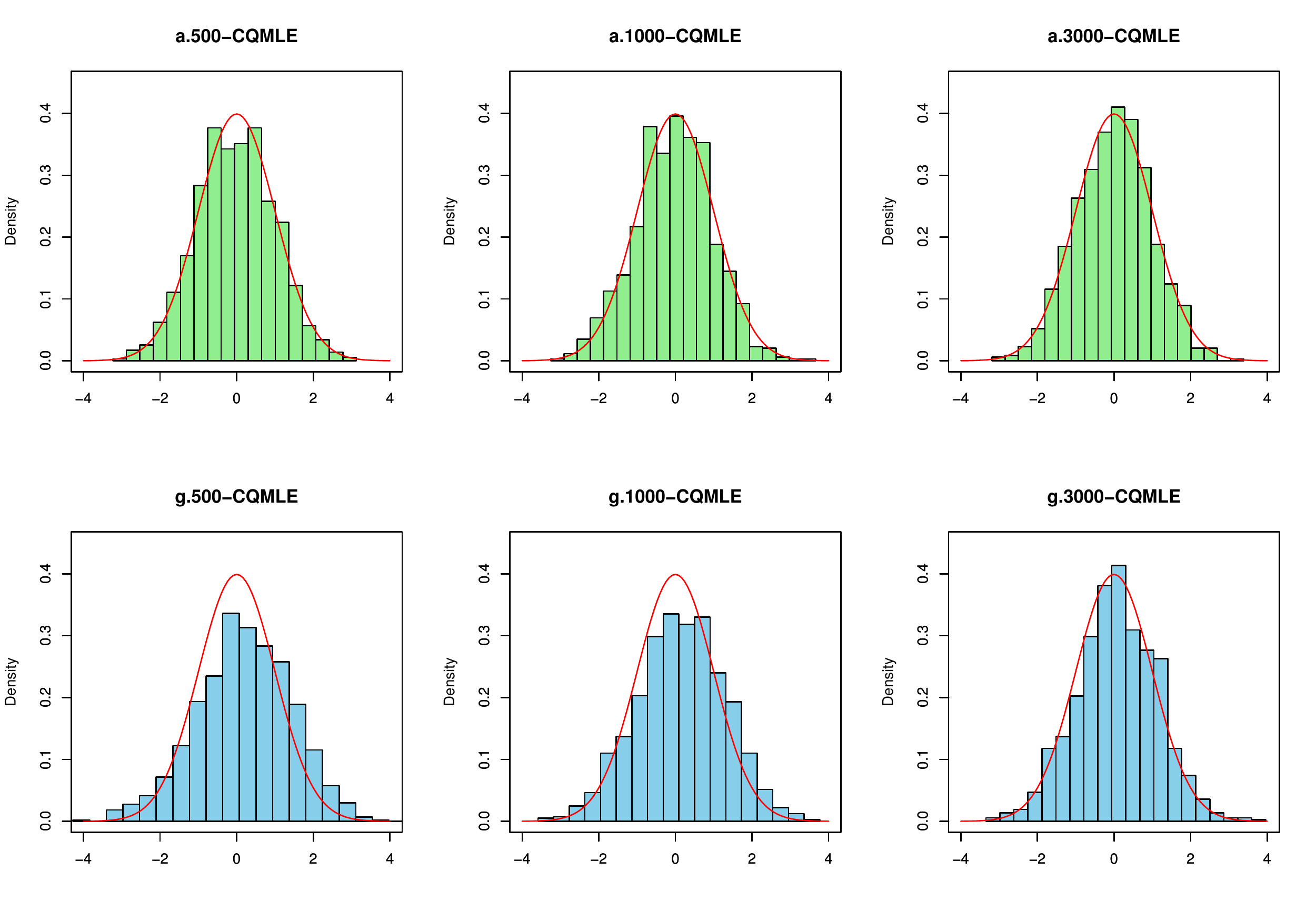}
  \end{center}
 \end{minipage}
 \begin{minipage}{0.49\hsize}
  \begin{center}
   \includegraphics[scale=0.3]{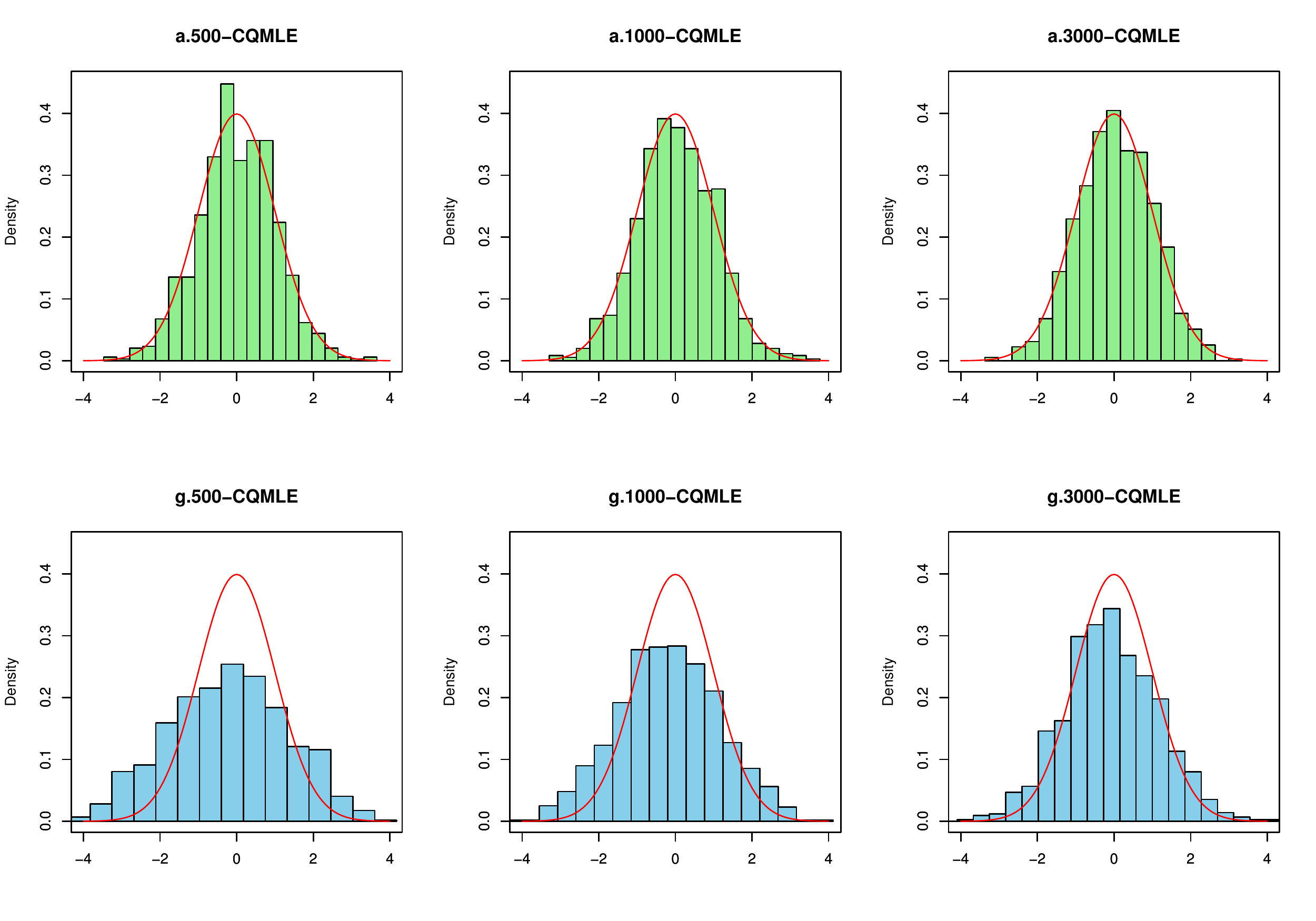}
  \end{center}
 \end{minipage}
\caption{NIG-$J$ example. 
Histograms of $1000$ independent Studentized estimates of $\al$ (green) and $\gam$ (blue) for $n=500$, $1000$, $3000$; $(T,\eta)=(1,5)$ (upper left $2 \times 3$ submatrix), $(T,\eta)=(1,10)$ (upper right $2 \times 3$ submatrix), $(T,\eta)=(5,5)$ (lower left $2 \times 3$ submatrix), and $(T,\eta)=(5,10)$ (lower right $2 \times 3$ submatrix).}
\label{fig:hist1*1}
\end{figure}

\medskip


\begin{figure}[htbp]
\begin{center}
  \includegraphics[scale=0.45]{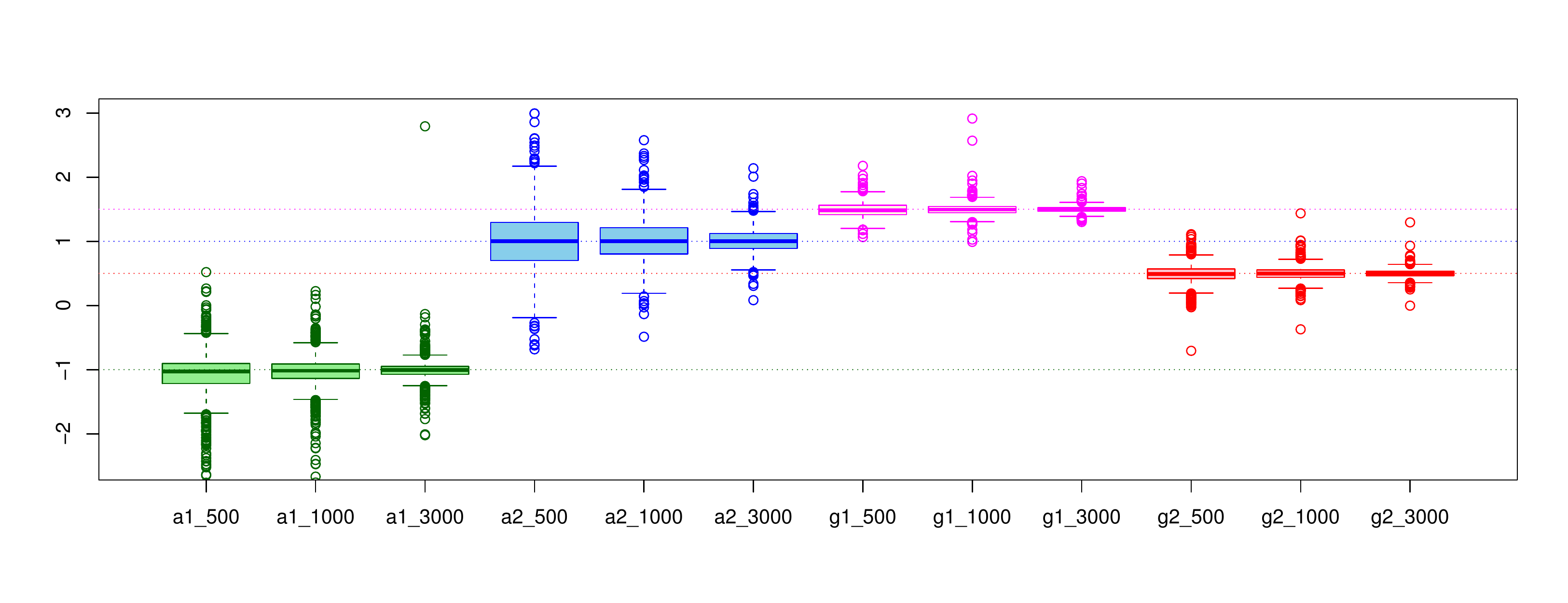}
\end{center}
\begin{center}
  \includegraphics[scale=0.45]{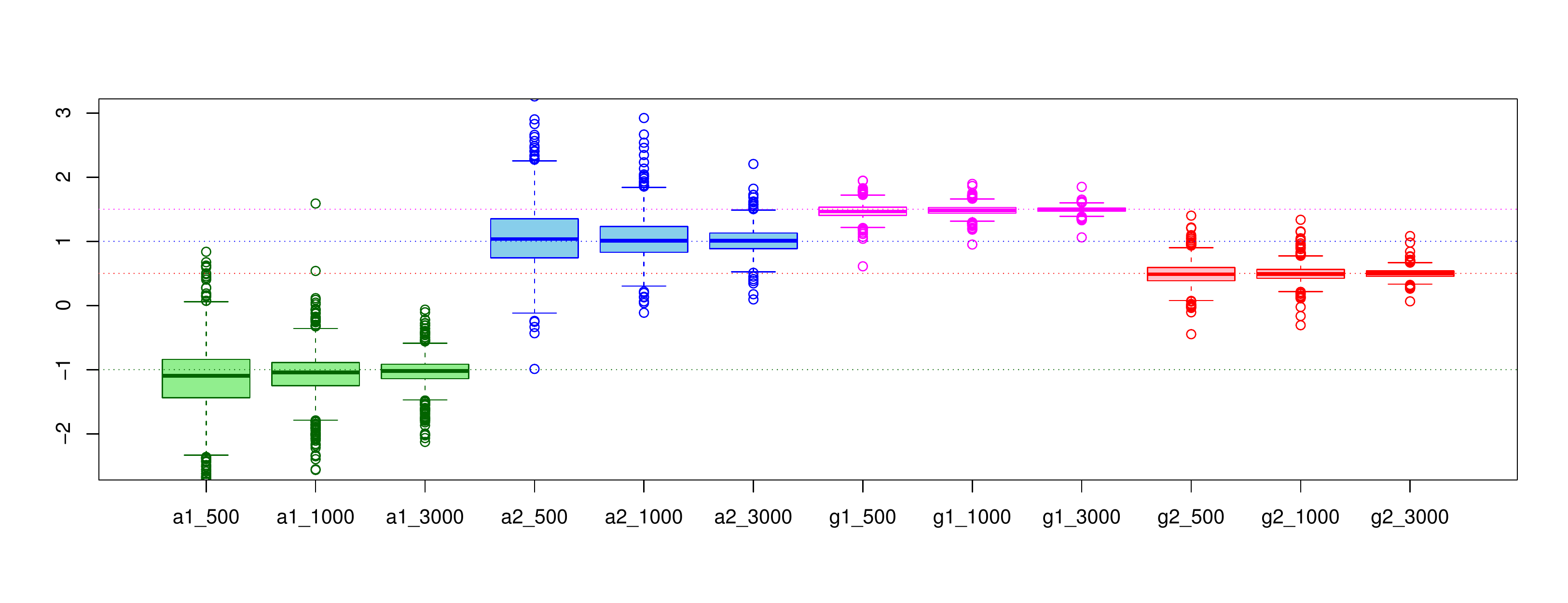}
\end{center}
\begin{center}
  \includegraphics[scale=0.45]{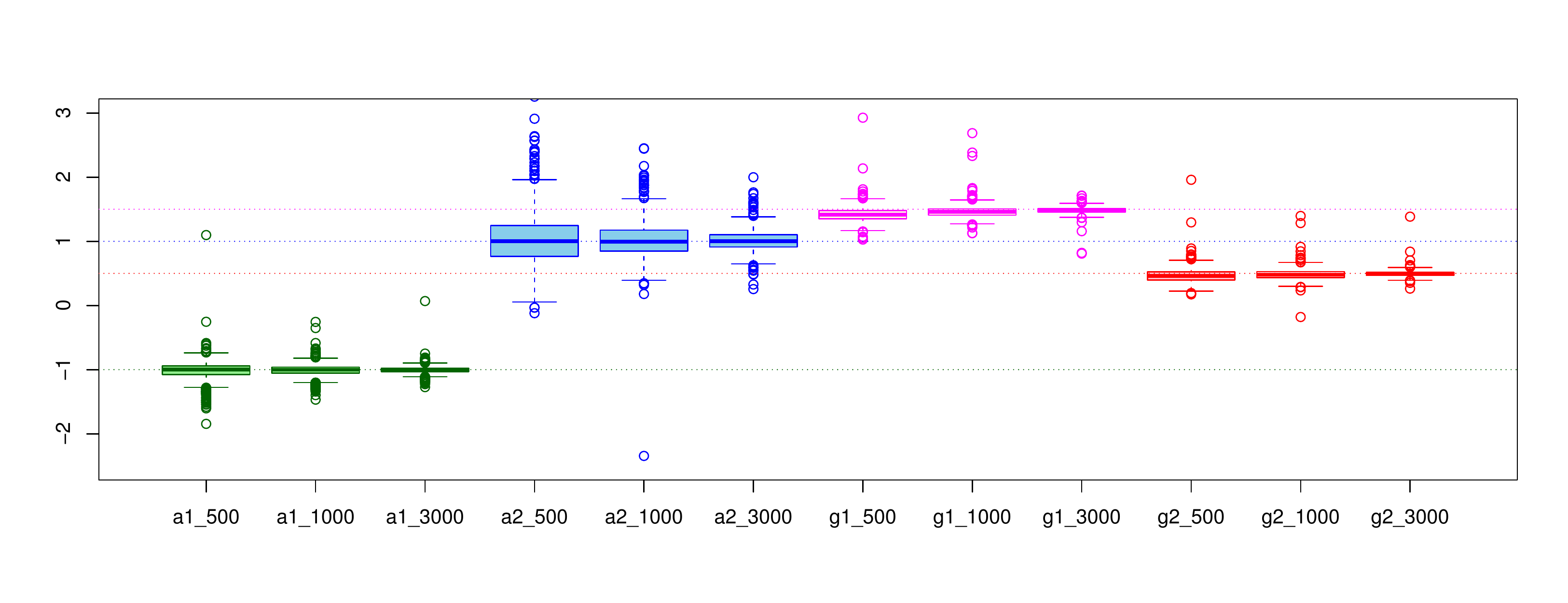}
\end{center}
\begin{center}
  \includegraphics[scale=0.45]{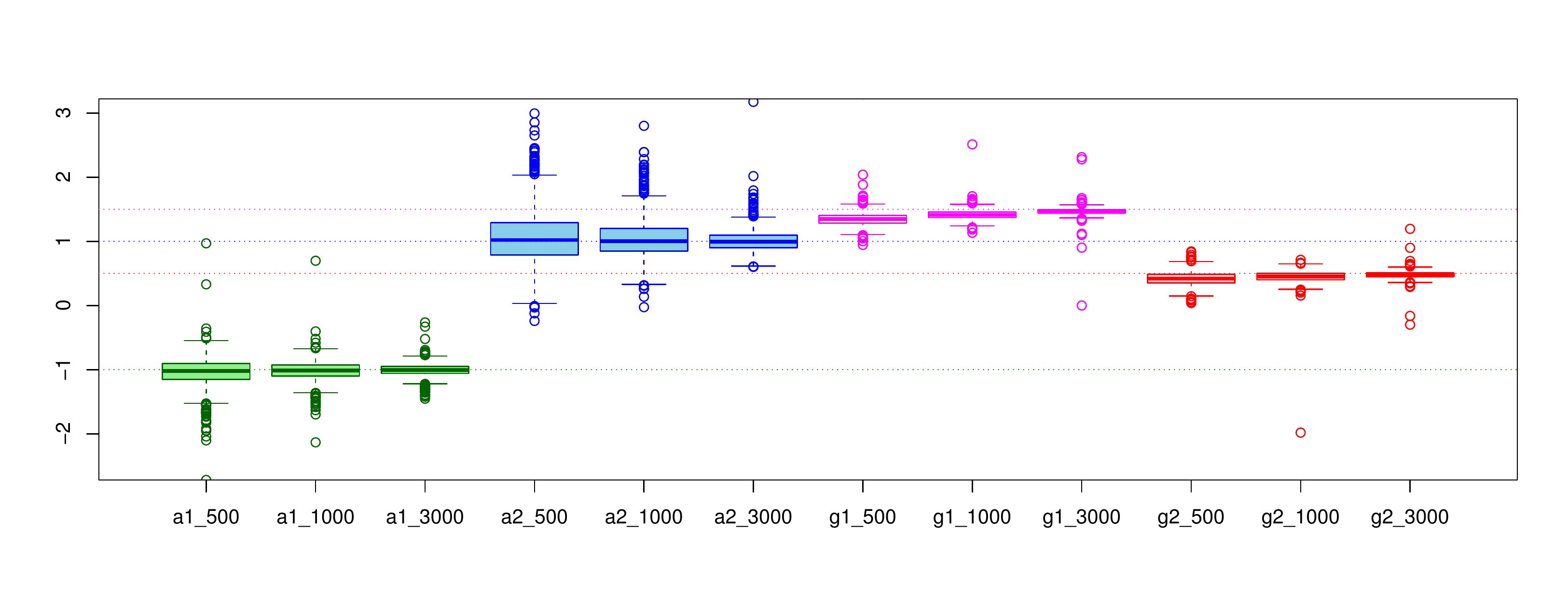}
\end{center}
 \caption{NIG-$J$ example. 
 Boxplots of $1000$ independent estimates $\hat{\al}_{1,n}$ (green), $\hat{\al}_{2,n}$ (blue), $\hat{\gam}_{1,n}$ (pink) and $\hat{\gam}_{2,n}$ (red) for $n=500$, $1000$, $3000$; $(T,\eta)=(1,5)$ (top), $(T,\eta)=(1,10)$ (second from the top), $(T,\eta)=(5,5)$ (second from the bottom), and $(T,\eta)=(5,10)$ (bottom).}
 \label{fig:bp2*2_4panels}
\end{figure}

\medskip


\begin{figure}[htbp]
 \begin{minipage}{0.49\hsize}
  \begin{center}
   \includegraphics[scale=0.3]{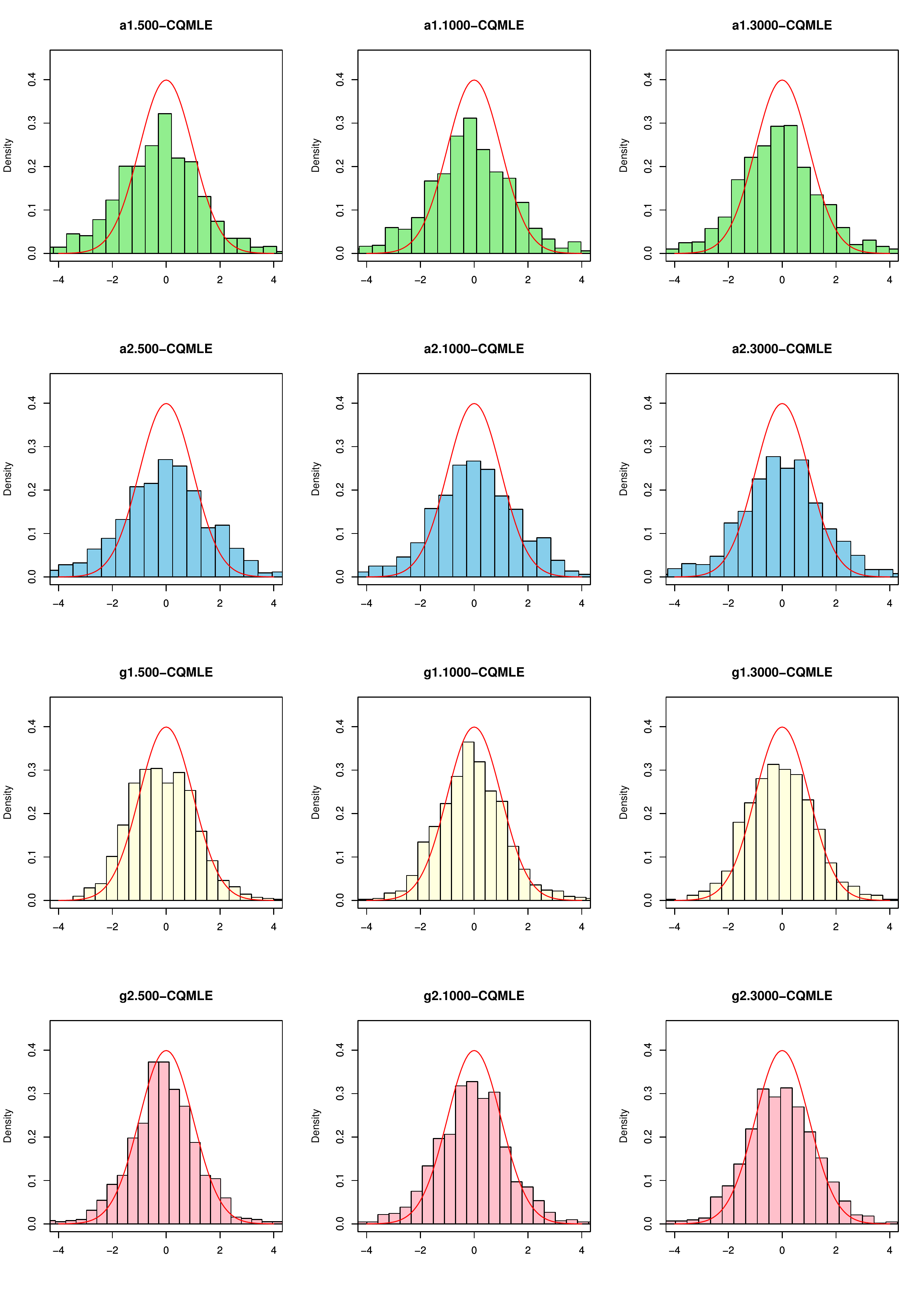}
  \end{center}
 \end{minipage}
 \begin{minipage}{0.49\hsize}
  \begin{center}
   \includegraphics[scale=0.3]{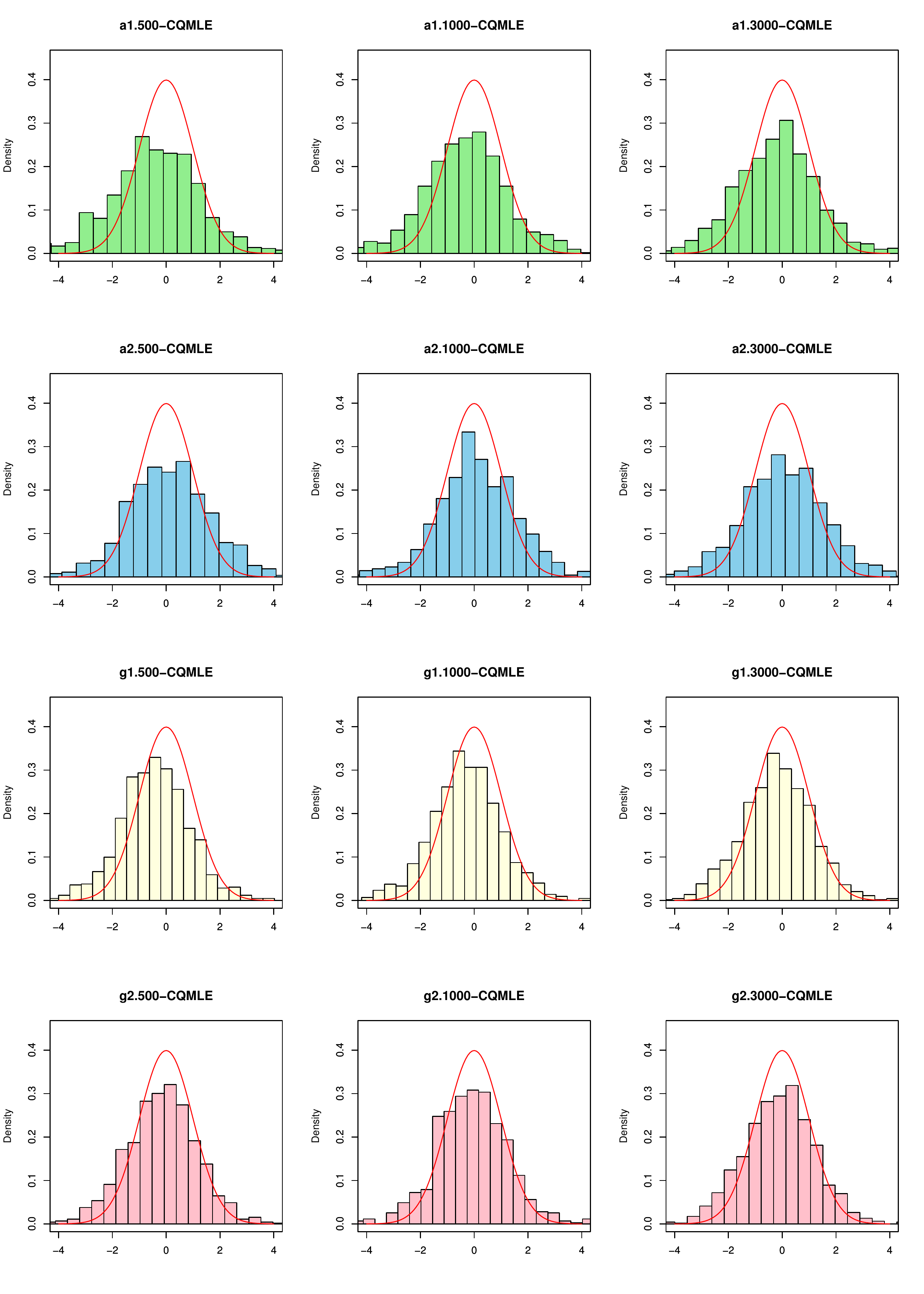}
  \end{center}
 \end{minipage}
\caption{NIG-$J$ example. 
Histograms of $1000$ independent Studentized estimates of $\al_{1}$ (green), $\al_{2}$ (blue), $\gam_{1}$ (cream) and $\gam_{2}$ (red) for $n=500$, $1000$, $3000$; $(T,\eta)=(1,5)$ (left $4 \times 3$ submatrix) and $(T,\eta)=(1,10)$ (right $4 \times 3$ submatrix).}
\label{fig:hist2*2-1}
\end{figure}

\medskip

\begin{figure}[htbp]
 \begin{minipage}{0.49\hsize}
  \begin{center}
   \includegraphics[scale=0.3]{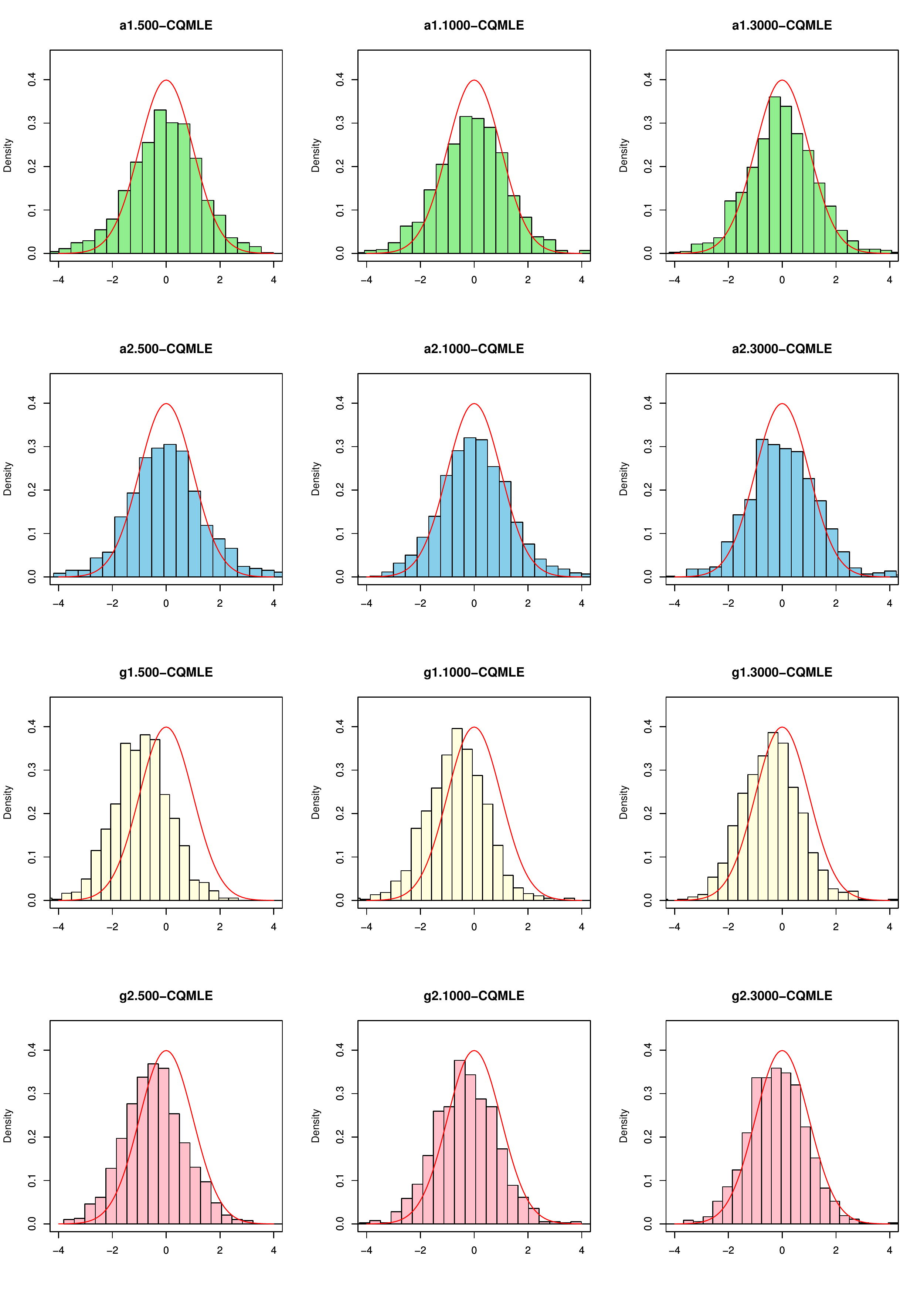}
  \end{center}
 \end{minipage}
 \begin{minipage}{0.49\hsize}
  \begin{center}
   \includegraphics[scale=0.3]{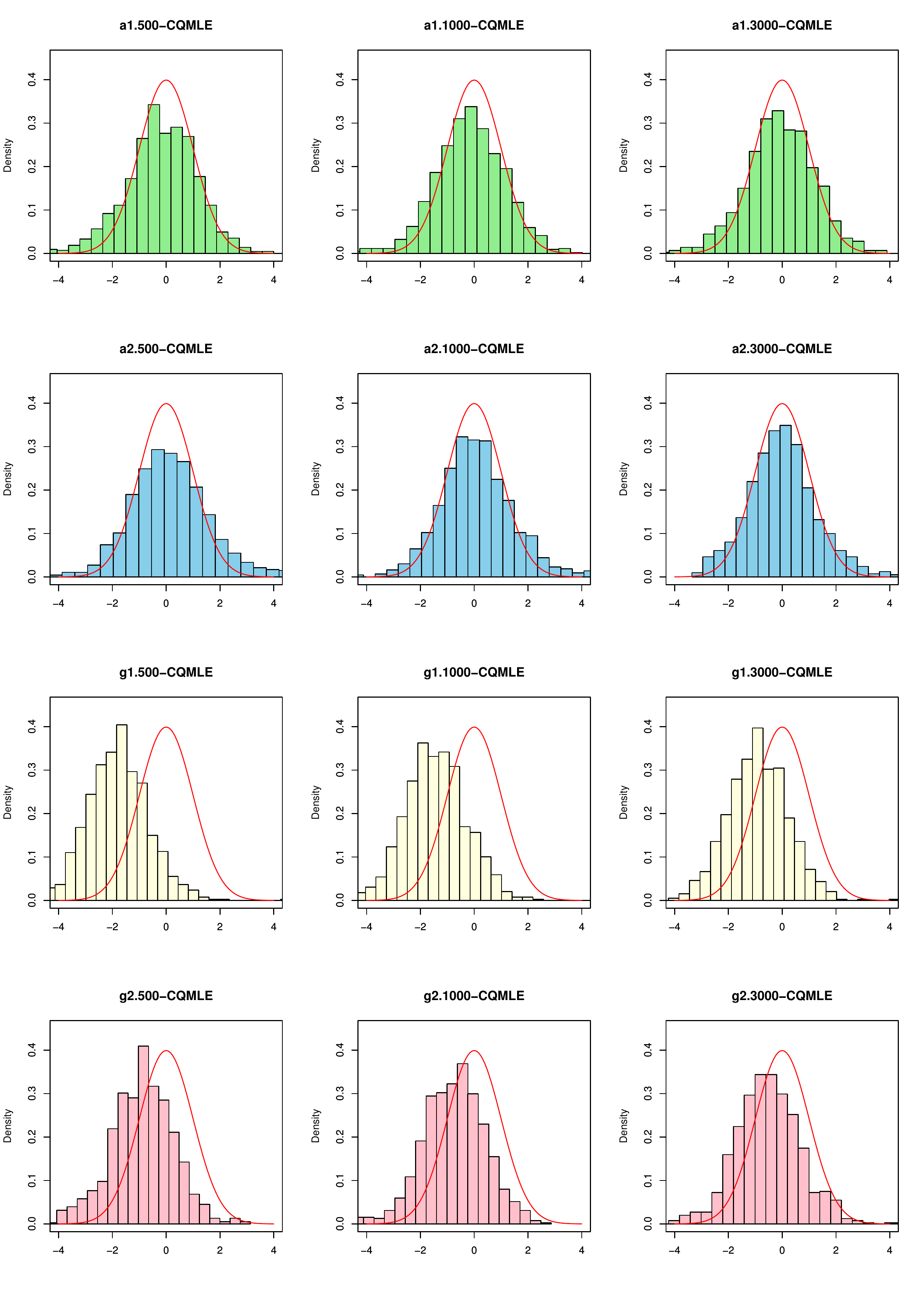}
  \end{center}
 \end{minipage}
\caption{NIG-$J$ example. 
Histograms of $1000$ independent Studentized estimates of $\al_{1}$ (green), $\al_{2}$ (blue), $\gam_{1}$ (cream) and $\gam_{2}$ (red) for $n=500$, $1000$, $3000$; $(T,\eta)=(5,5)$ (left $4 \times 3$ submatrix) and $(T,\eta)=(5,10)$ (right $4 \times 3$ submatrix).}
\label{fig:hist2*2-2}
\end{figure}


\subsection{Genuine $\beta$-stable driver}\label{sec_sim.stable.J}

Next we set $\mcl(J_{1})=S_{\beta}$ with $\beta=1.5$.
Given a realization $(x_{t_{j}})_{j=0}^{n}$ of $(X_{t_{j}})_{j=0}^{n}$ we have to repeatedly evaluate
\begin{equation}
(\al,\gam) \mapsto \sumj \bigg\{-\log[h^{1/\beta}c(x_{t_{j-1}},\gam)]+\log\phi_{\beta}\bigg(\frac{x_{t_{j}}-x_{t_{j-1}}-a(x_{t_{j-1}},\al)h}{h^{1/\beta}c(x_{t_{j-1}},\gam)}\bigg)
\bigg\}.
\nonumber
\end{equation}
The stable density $\phi_{\beta}$ is no longer explicit while we can resort to numerically integration.
Here we used the function \texttt{dstable} in the R package \texttt{stabledist}.
As in the previous example, we give simulation results for $(p_{\al},p_{\gam})=(1,1)$ and $(2,2)$, with using uniformly distributed initial values for \texttt{optim} search. In order to observe effect of the terminal-time value $T$ we conduct the cases of $T=5$ and $T=10$, for $n=100$, $200$, and $500$. For Studentization, we used the values $C_{\al}(1.5)=0.4281$ and $C_{\gam}(1.5)=0.9556$ borrowed from \cite[Table 6]{MatTak06}. 

\begin{itemize}
\item Figures \ref{fig:s_bp1*1_2panels} and \ref{fig:s_hist1*1} show the boxplots and the histograms when $p_{\al}=p_{\gam}=1$ for $T=5$ and $10$. As is expected, we observe much better estimation accuracy compared with the previous NIG-driven case. The figures reveal that the estimation accuracy of $\al$ are overall better for larger $T$, while at the same time a larger $h$ may lead to a more biased $\aes$. Different from the NIG driven case there is no severe bias in estimating $\gam$. Somewhat surprisingly, the accuracy of Studentization especially for the scale parameters may be good enough even for much smaller $n$ compared with the NIG driven case: the standard normality is well achieved even for $n=100$.
\item Figures \ref{fig:s_bp2*2_2panels} and \ref{fig:s_hist2*2} show the results for $p_{\al}=p_{\gam}=2$ with $T=5$ or $10$. The observed tendencies, including those compared with the NIG driven cases, are almost analogous to the case where $p_{\al}=p_{\gam}=1$.
\end{itemize}

In sum, our stable quasi-likelihood works quite well especially when $J$ is standard $1.5$-stable, although so small $T$ should be avoided for good estimation accuracy of $\al$.

\medskip


\begin{figure}[htbp]
 \begin{minipage}{0.49\hsize}
  \begin{center}
   \includegraphics[scale=0.32]{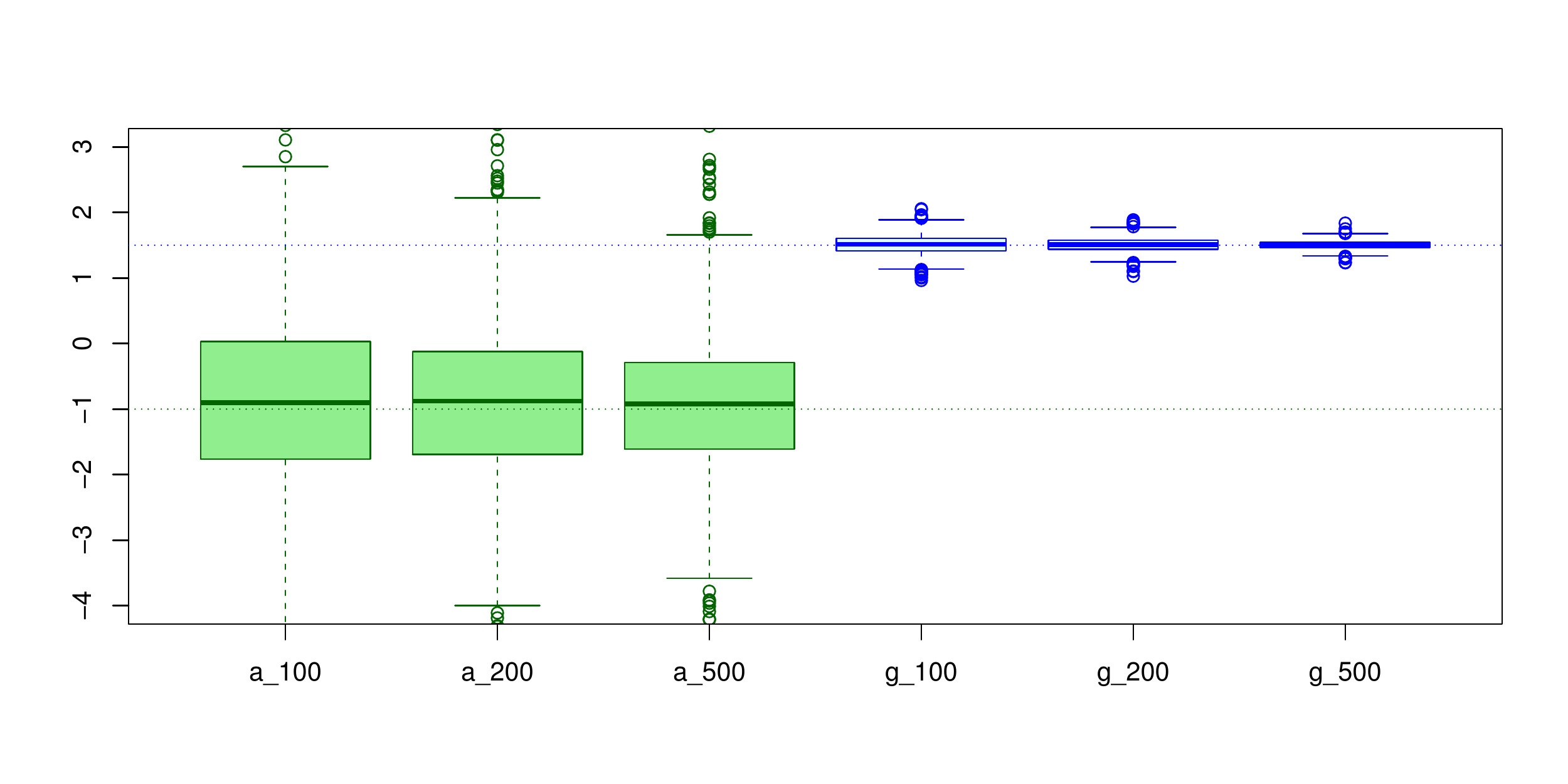}
  \end{center}
 \end{minipage}
 \begin{minipage}{0.49\hsize}
  \begin{center}
   \includegraphics[scale=0.32]{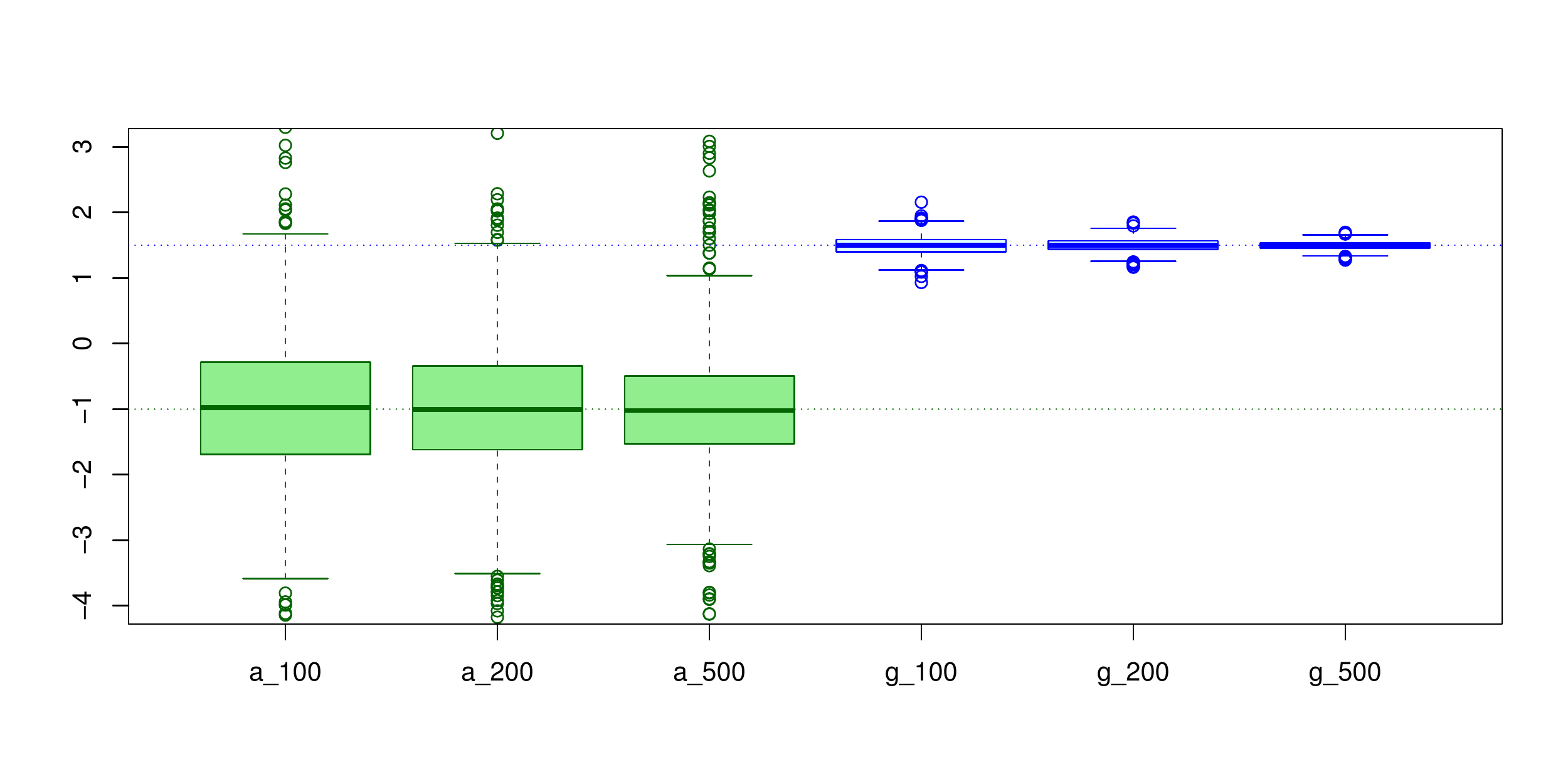}
  \end{center}
 \end{minipage}
 \caption{$S_{1.5}$-$J$ example. 
 Boxplots of $1000$ independent estimates $\aes$ (green) and $\ges$ (blue) for $n=100$, $200$, $500$; $T=5$ (left) and  $T=10$ (right).}
 \label{fig:s_bp1*1_2panels}
\end{figure}

\medskip

\begin{figure}[htbp]
 \begin{minipage}{0.49\hsize}
  \begin{center}
   \includegraphics[scale=0.3]{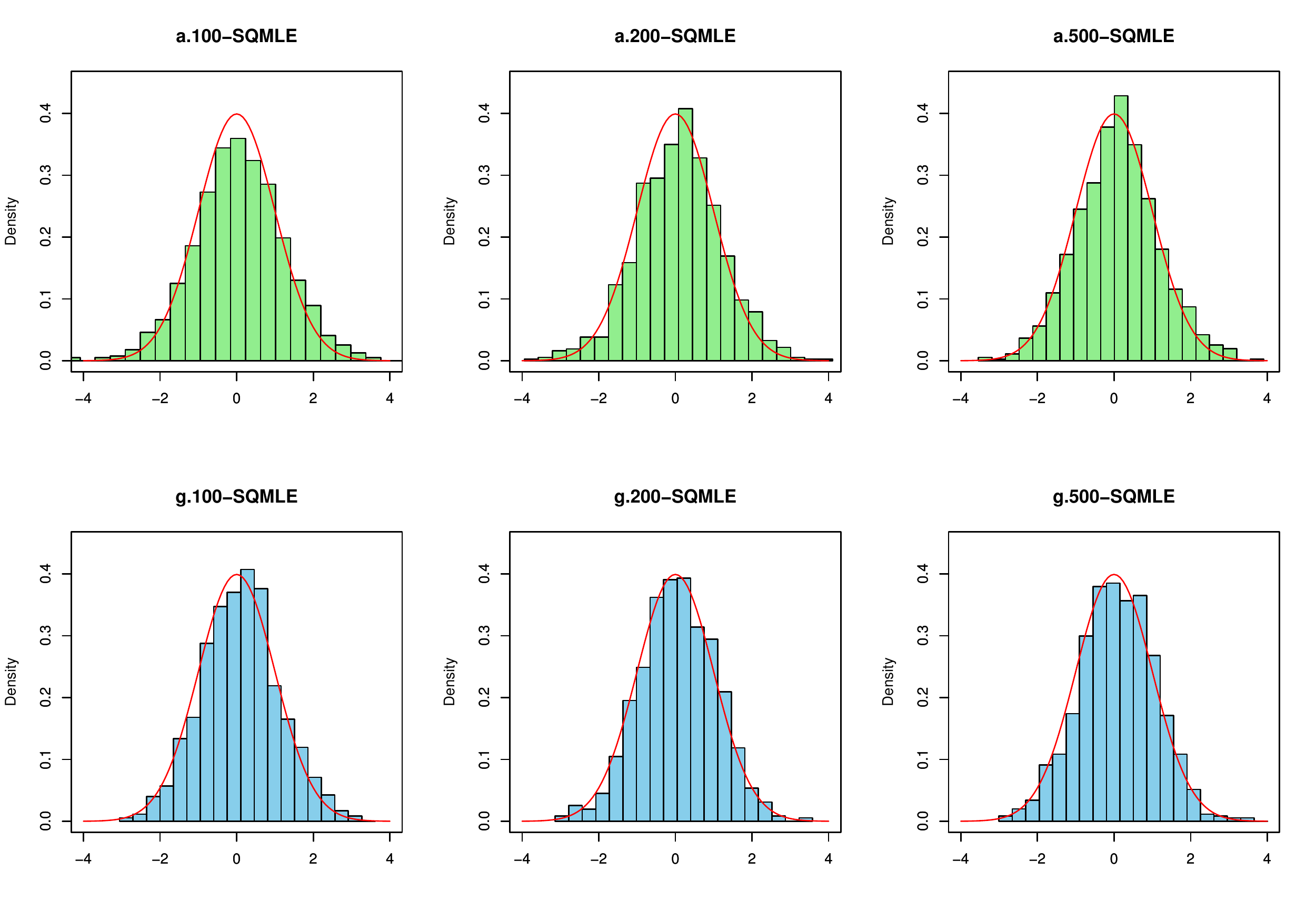}
  \end{center}
 \end{minipage}
 \begin{minipage}{0.49\hsize}
  \begin{center}
   \includegraphics[scale=0.3]{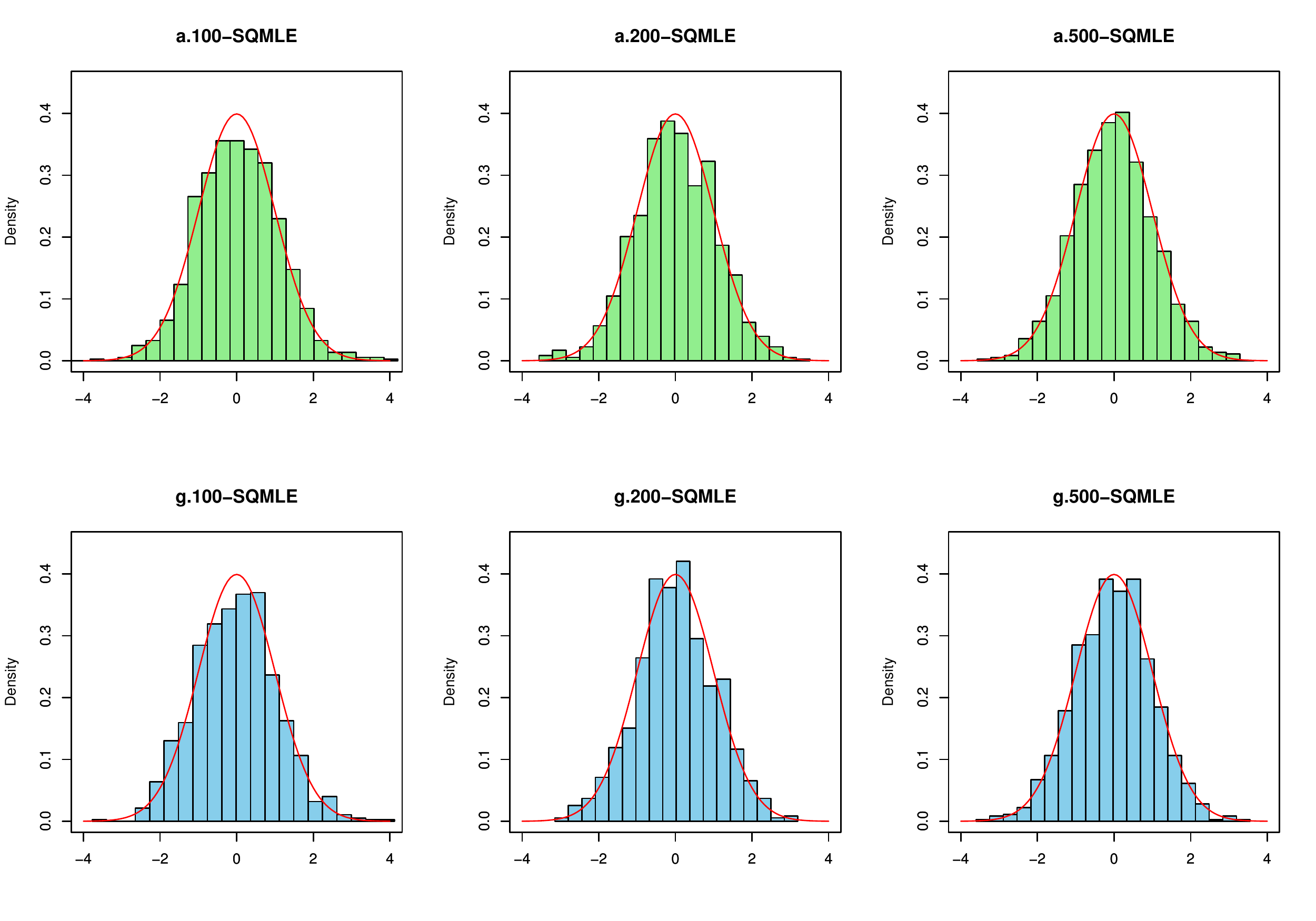}
  \end{center}
 \end{minipage}
\caption{$S_{1.5}$-$J$ example. 
Histograms of $1000$ independent Studentized estimates of $\al$ (green) and $\gam$ (blue) for $n=100$, $200$, $500$; $T=5$ (left $2 \times 3$ submatrix) and $T=10$ (right $2 \times 3$ submatrix).}
\label{fig:s_hist1*1}
\end{figure}

\medskip


\begin{figure}[htbp]
\begin{center}
  \includegraphics[scale=0.45]{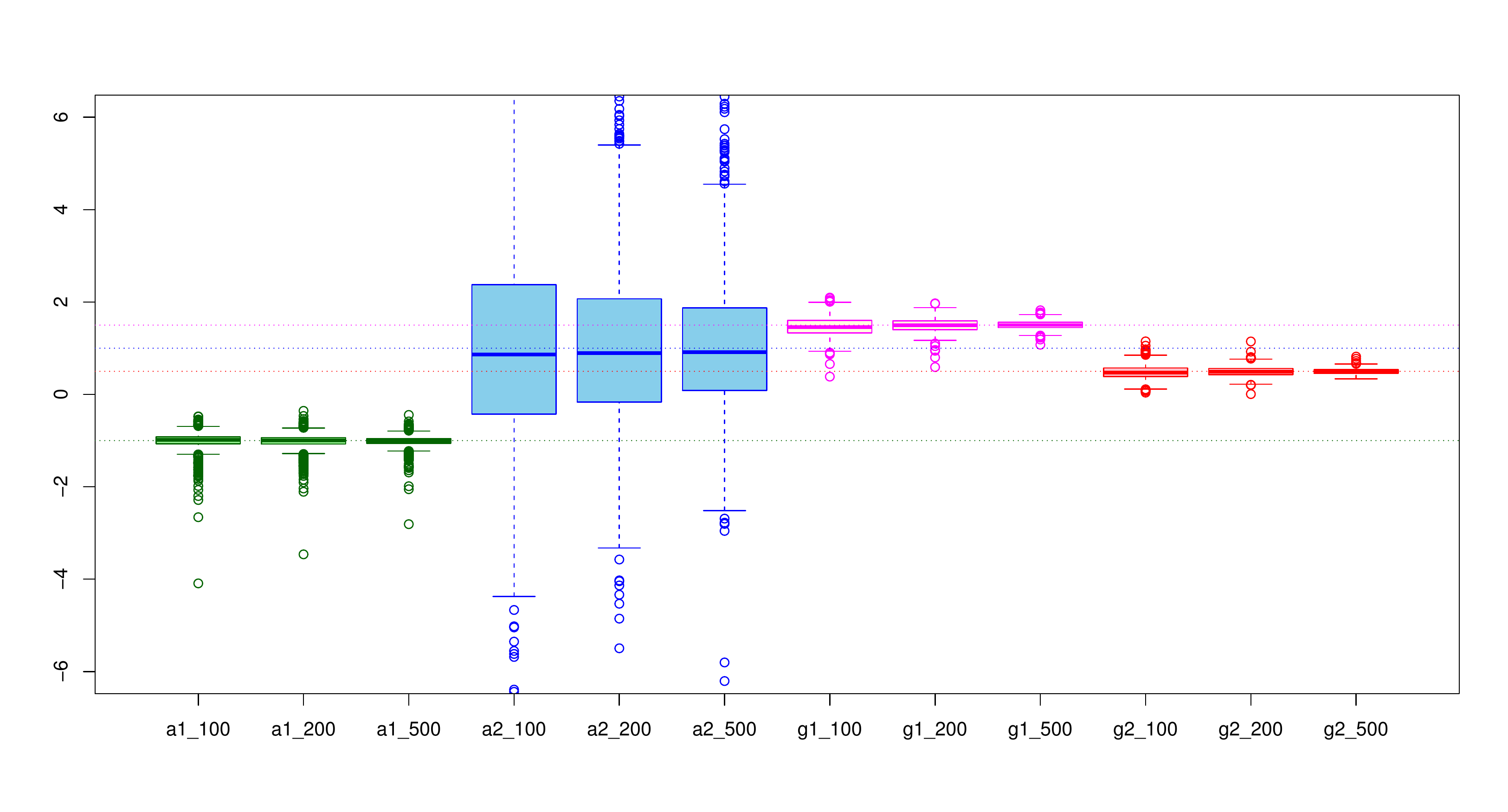}
\end{center}
\begin{center}
  \includegraphics[scale=0.45]{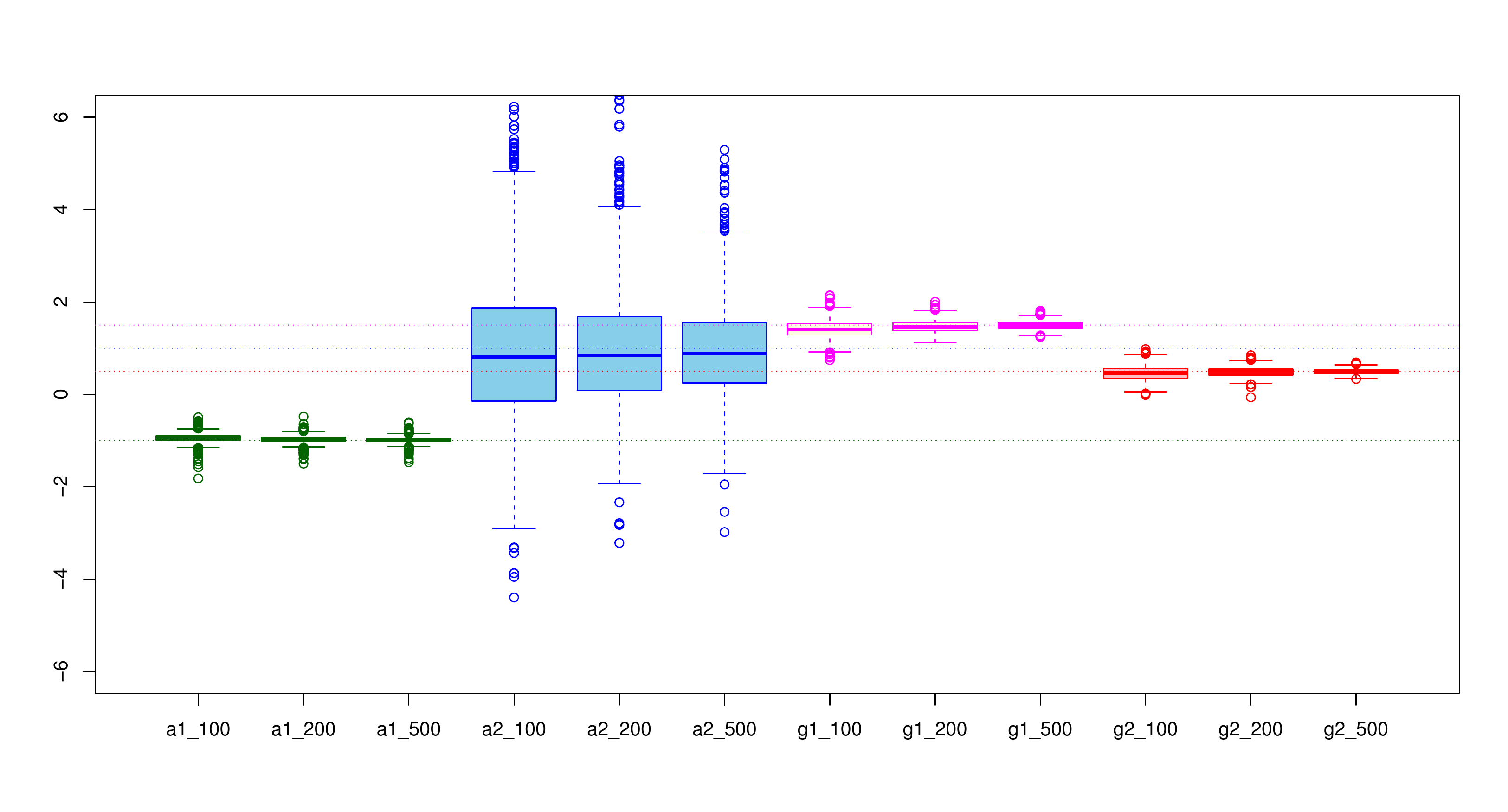}
\end{center}
 \caption{$S_{1.5}$-$J$ example. 
 Boxplots of $1000$ independent estimates $\hat{\al}_{1,n}$ (green), $\hat{\al}_{2,n}$ (blue), $\hat{\gam}_{1,n}$ (pink) and $\hat{\gam}_{2,n}$ (red) for $n=100$, $200$, $500$; $T=5$ (top) and $10$ (bottom).}
 \label{fig:s_bp2*2_2panels}
\end{figure}

\medskip


\begin{figure}[htbp]
 \begin{minipage}{0.49\hsize}
  \begin{center}
   \includegraphics[scale=0.3]{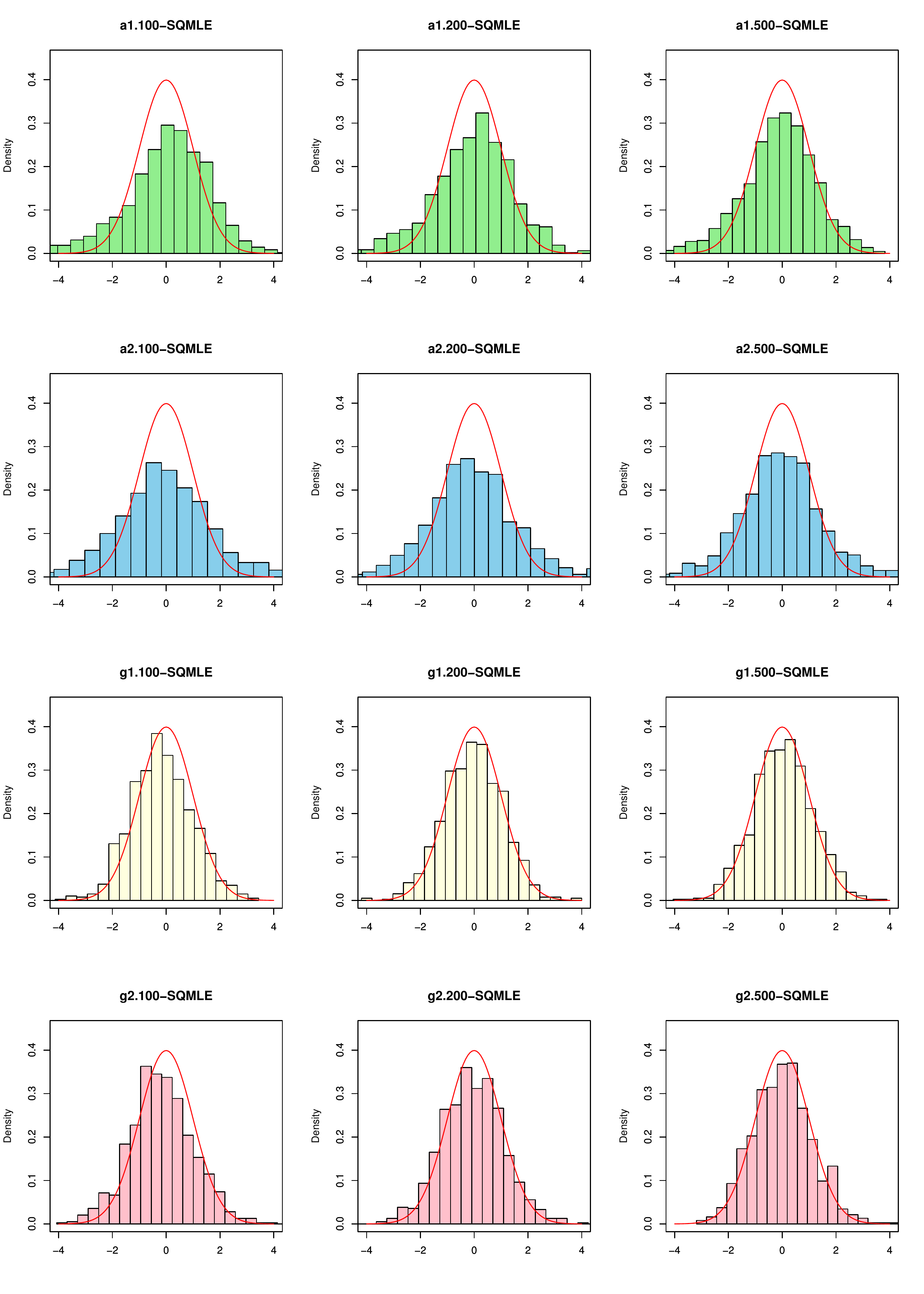}
  \end{center}
 \end{minipage}
 \begin{minipage}{0.49\hsize}
  \begin{center}
   \includegraphics[scale=0.3]{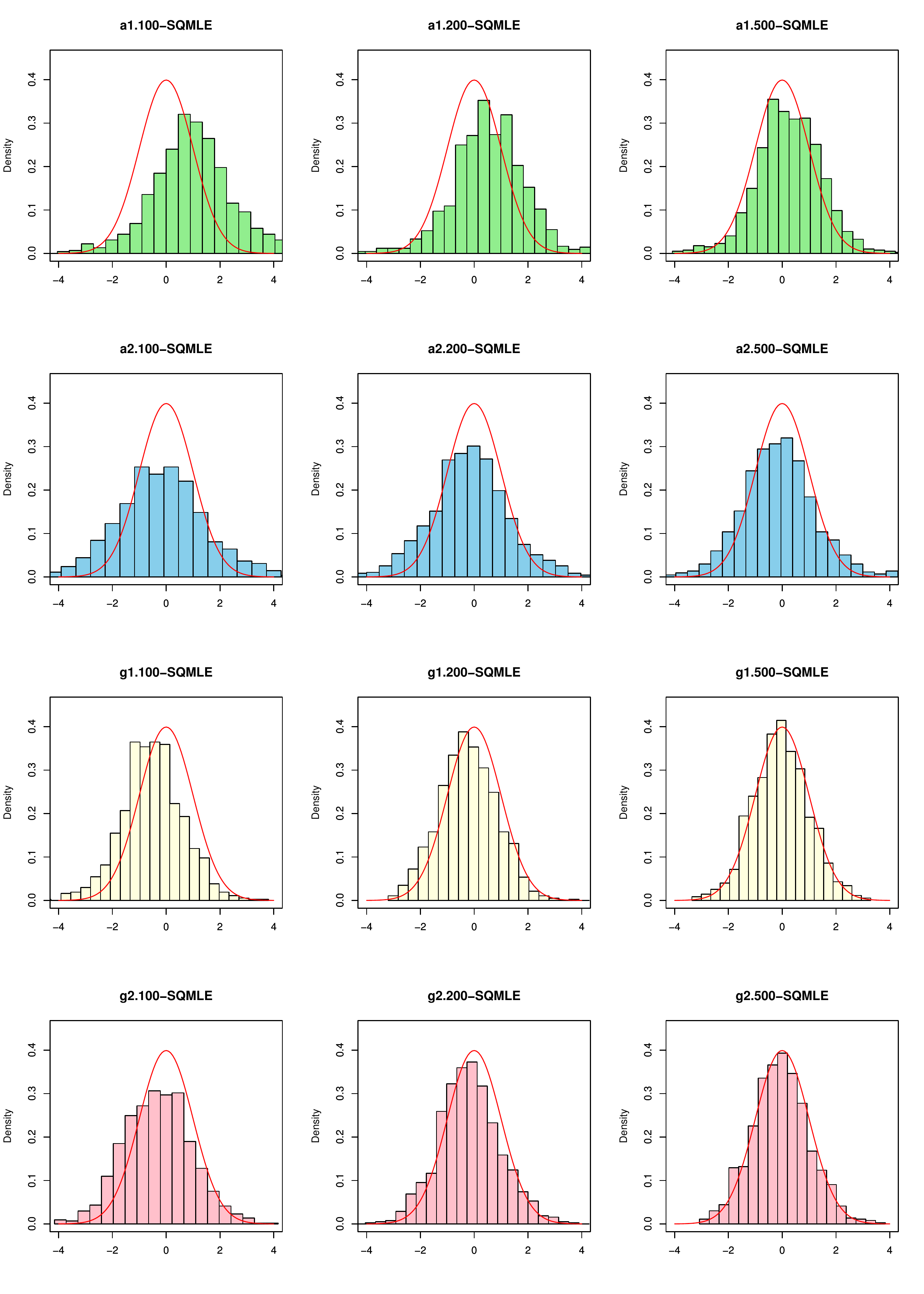}
  \end{center}
 \end{minipage}
\caption{$S_{1.5}$-$J$ example. 
Histograms of $1000$ independent Studentized estimates of $\al_{1}$ (green), $\al_{2}$ (blue), $\gam_{1}$ (cream) and $\gam_{2}$ (red) for $n=100$, $200$, $500$; $T=5$ (left $4 \times 3$ submatrix) and $T=10$ (right $4 \times 3$ submatrix).}
\label{fig:s_hist2*2}
\end{figure}


\section{Proofs of Lemmas \ref{key.lemma} and \ref{key.lemma.cf}}\label{sec_lem.proofs}

This section presents the proofs of the $L^{1}$-local limit theorems given in Section \ref{sec_LSLP}.

\subsection{Proof of Lemma \ref{key.lemma}}

We begin with the following lemma, which in particular completes the proof of the first half of Lemma \ref{key.lemma}(1).
\begin{lem}
Let Assumption \ref{A_J}(1) hold with the function $\rho$ being bounded.
Then, for every $C\ge 0$ and $s<1$,
\begin{equation}
\int_{(0,\infty)} (u^{-s} \vee u^{C})|\vp_{h}(u)-\vp_{0}(u)|du \lesssim h^{a_{\nu}},
\nonumber
\end{equation}
where the constant $a_{\nu}\in(0,1]$ is defined by
\begin{align}
a_{\nu} = \left\{
\begin{array}{cl}
1 & (c_{\rho}=0) \\
(\del/\beta)\wedge 1 & (c_{\rho}>0).
\end{array}\right.
\nonumber
\end{align}
In particular, the distribution $\mcl(h^{-1/\beta}J_{h})$ for $h\in(0,1]$ admits a positive smooth Lebesgue density, which we denote by $f_{h}$, such that
\begin{equation}
\sup_{y}\left| f_{h}(y) - \phi_{\beta}(y) \right| \lesssim h^{a_{\nu}}.
\nonumber
\end{equation}
\label{key.lemma+1}
\end{lem}

\begin{proof}
By the expression \eqref{J.lk} we have
\begin{align}
\vp_{h}(u) &= \exp\bigg( \int(\cos(uz)-1)h^{1+1/\beta}g(h^{1/\beta}z)dz \bigg) \nn\\
&= \vp_{0}(u)\exp\bigg(\int (\cos(uz)-1)\rho(h^{1/\beta}z)g_{0,\beta}(z)dz\bigg)=:\vp_{0}(u)\exp\left\{\chi_{h}(u)\right\}.
\nonumber
\end{align}
Pick a small $\ep'_{\rho}>0$ such that $\sup_{|y|\le \ep'_{\rho}}|\rho(y)|\le 1/2$. We will make use of the following two different bounds for the function $\chi_{h}$: on the one hand, we have
\begin{align}
|\chi_{h}(u)| &\le \int (1-\cos(uz))|\rho(h^{1/\beta}z)|g_{0,\beta}(z)dz \nn\\
&=\int_{|z|\le\ep'_{\rho}h^{-1/\beta}}(1-\cos(uz))g_{0,\beta}(z)|\rho(h^{1/\beta}z)|dz+\int_{|z|>\ep'_{\rho}h^{-1/\beta}}(1-\cos(uz))g_{0,\beta}(z)|\rho(h^{1/\beta}z)|dz \nn\\
&\le \frac{1}{2}\int_{|z|\le\ep'_{\rho}h^{-1/\beta}}(1-\cos(uz))g_{0,\beta}(z)dz + 4\|\rho\|_{\infty}\int_{\ep'_{\rho}h^{-1/\beta}}^{\infty}g_{0,\beta}(z)dz \nn\\
&\le -\frac{1}{2}\int (\cos(uz)-1)g_{0,\beta}(z)dz + Ch = \frac{1}{2}|u|^{\beta}+Ch;
\nn
\end{align}
on the other hand,
\begin{align}
& \int_{|z|\le\ep'_{\rho}h^{-1/\beta}}(1-\cos(uz))g_{0,\beta}(z)|\rho(h^{1/\beta}z)|dz \nn\\
& \le c_{\rho}h^{\del/\beta}\bigg( 2\int_{|z|\le\ep'_{\rho}h^{-1/\beta},\,|z|>1}g_{0,\beta}(z)|z|^{\del}dz 
+ \int_{|z|\le\ep'_{\rho}h^{-1/\beta},\,|z|\le 1} (uz)^{2} g_{0,\beta}(z)|z|^{\del}dz \bigg) \nn\\
&\lesssim c_{\rho}h^{\del/\beta}\bigg( \int_{1}^{\ep'_{\rho}h^{-1/\beta}}z^{-1-\beta+\del}dz + u^{2} \int_{0}^{1}z^{1-\beta+\del}dz\bigg) \nn\\
&\lesssim c_{\rho}\big( h^{(\del/\beta)\wedge 1} + u^{2} h^{\del/\beta}\big),
\nn
\end{align}
where we used the fact $\sup_{y}|\frac{\sin y}{y}|<\infty$ in the second step, so that
\begin{equation}
|\chi_{h}(u)| \lesssim c_{\rho}\big( h^{(\del/\beta)\wedge 1} + u^{2} h^{\del/\beta}\big) + h.
\nn
\end{equation}
It follows from these estimates for $\chi_{h}$ with the mean-value theorem that for every $s<1$ and $C\ge 0$ we have
\begin{align}
\int_{(0,\infty)} (u^{-s} \vee u^{C})|\vp_{h}(u)-\vp_{0}(u)|du
&\lesssim \int_{(0,\infty)} (u^{-s} \vee u^{C}) \vp_{0}(u)\left| \exp\{\chi_{h}(u)\} -1\right| du \nn\\
&\le \int_{(0,\infty)} (u^{-s} \vee u^{C})\vp_{0}(u)\left(\sup_{0\le s\le 1}\exp(s\chi_{h}(u))\right)|\chi_{h}(u)| du \nn\\
&\lesssim \int_{(0,\infty)} (u^{-s} \vee u^{C})e^{-|u|^{^\beta}/2}\left(c_{\rho}\big( h^{(\del/\beta)\wedge 1} + u^{2} h^{\del/\beta}\big) + h\right) du \nn\\
&\lesssim c_{\rho}h^{(\del/\beta)\wedge 1} + h \lesssim h^{a_{\nu}}.
\nn
\end{align}
This prove the first half of the lemma.
Since $\sup_{h\in(0,1]}\vp_{h}(u)\lesssim \exp(-C|u|^{\beta})$ from the above argument, the existence of the positive smooth density $f_{h}$ follows from the same argument as in the proof of \cite[Lemma 4.4(a)]{Mas10ejs}.
The latter half is a direct consequence of the Fourier inversion:
\begin{equation}
\sup_{y}\left| f_{h}(y) - \phi_{\beta}(y) \right|
= \sup_{y}\bigg|\frac{1}{2\pi}\int e^{-iuy}\left(\vp_{h}(u)-\vp_{0}(u)\right)du\bigg| \lesssim \int|\vp_{h}(u)-\vp_{0}(u)|du \lesssim h^{a_{\nu}}.
\nonumber
\end{equation}
This completes the proof.
\end{proof}

\medskip

We now prove Lemma \ref{key.lemma}(1).
Because of the boundedness of $\rho$, the {\ld} of $\mcl(h^{-1/\beta}J_{h})$ is bounded by a constant multiple of $g_{0,\beta}(z)$.
Invoking \cite[Theorem 25.3]{Sat99}, we see that the tail of $f_{h}$ is bounded by that of $\phi_{\beta}$ uniformly in $h\in(0,1]$:
for each $\kappa<\beta$,
\begin{equation}
\sup_{h\in(0,1]}\sup_{M>0}M^{\beta-\kappa}\int_{|y|>M}|y|^{\kappa}f_{h}(y)dy<\infty.
\label{key.lemma-p3}
\end{equation}
Then, for any positive sequence $b_{n}\uparrow\infty$ the quantity $\int |y|^{\kappa}\left|f_{h}(y)-\phi_{\beta}(y)\right|dy$ is bounded by the sum of the two terms
\begin{equation}
\int_{|y|\ge b_{n}}|y|^{\kappa}f_{h}(y)dy + \int_{|y|\ge b_{n}}|y|^{\kappa}\phi_{\beta}(y)dy \lesssim b_{n}^{\kappa-\beta} \to 0
\nonumber
\end{equation}
and
\begin{equation}
\bigg(\sup_{y}\left| f_{h}(y) - \phi_{\beta}(y) \right|\bigg)
\bigg( \int_{|y|\le b_{n}}|y|^{\kappa}dy \bigg) \lesssim b_{n}^{1+\kappa}h^{a_{\nu}},
\nonumber
\end{equation}
where we used Lemma \ref{key.lemma+1} for the latter.
The convergence $b_{n}^{1+\kappa}h^{a_{\nu}} \to 0$ follows on taking any $b_{n}=o(h^{-a_{\nu}/(1+\kappa)})$.

\medskip

Turning to the proof of Lemma \ref{key.lemma}(2), again we pick a positive real sequence $b_{n}\to\infty$.
Then
\begin{equation}
\int|f_{h}(y)-\phi_{\beta}(y)|dy
\lesssim \int_{(b_{n},\infty)}|f_{h}(y)-\phi_{\beta}(y)|dy + \int_{(0,b_{n}]}|f_{h}(y)-\phi_{\beta}(y)|dy =: \del'_{n} + \del''_{n}.
\label{key.lemma-pr.2-0.5}
\end{equation}
By \eqref{key.lemma-p3} with $\kappa=0$ we have
\begin{equation}
\del'_{n}\lesssim b_{n}^{-\beta}.
\label{key.lemma-pr.2-1}
\end{equation}
Recalling that $\psi_{h}(u) := \log\vp_{h}(u)$ and that we are assuming that $g\equiv 0$ on $\{|z|>K\}$, we have $\p_{u}\vp_{h}(u) = \vp_{h}(u)\p_{u}\psi_{h}(u)$ for $u>0$. Using Fourier inversion, integration by parts, and the fact $\sup_{y\in\mbbr}\frac{|\sin y|}{|y|^{r}}<\infty$ for any $r\in[0,1]$, we can bound $\del''_{n}$ as follows:
\begin{align}
\del''_{n}
&\lesssim \int_{(0,b_{n}]} \bigg| \int e^{-iuy} \left(\vp_{h}(u)-\vp_{0}(u)\right) du\bigg| dy \nn\\
&\lesssim \int_{(0,b_{n}]} \bigg| \int_{(0,\infty)} \cos(uy) \left(\vp_{h}(u)-\vp_{0}(u)\right) du\bigg| dy \nn\\
&\lesssim \int_{(0,b_{n}]} \frac{1}{y}\bigg| \int_{(0,\infty)} \sin(uy)\left(\p_{u}\vp_{h}(u)-\p_{u}\vp_{0}(u)\right) du\bigg|dy \nn\\
&\lesssim \int_{(0,b_{n}]} y^{r-1} \int_{(0,\infty)} u^{r}\left|\p_{u}\vp_{h}(u)-\p_{u}\vp_{0}(u)\right| du dy \nn\\
&\lesssim b_{n}^{r} \int_{(0,\infty)} u^{r} \left| \p_{u}\vp_{h}(u)-\p_{u}\vp_{0}(u) \right| du \nn\\
&\lesssim b_{n}^{r} \int_{(0,\infty)} u^{r} |\vp_{h}(u)-\vp_{0}(u)| |\p_{u}\psi_{h}(u)| du + b_{n}^{r} \int_{(0,\infty)} u^{r}\vp_{0}(u) |\p_{u}\psi_{h}(u) + \beta u^{\beta-1}|du.
\label{key.lemma-pr.2-2}
\end{align}
Suppose for a moment that
\begin{equation}
|\p_{u}\psi_{h}(u) + \beta u^{\beta-1}| \lesssim \frac{h}{u}, \qquad u>0.
\label{key.lemma-pr.2-3}
\end{equation}
Then $|\p_{u}\psi_{h}(u)| \lesssim (1+u^{\beta})/u$ and it follows from \eqref{key.lemma-pr.2-2} and the statement (1)(a) that
\begin{align}
\del''_{n} &\lesssim b_{n}^{r} \int_{(0,\infty)} u^{r-1}(1+u^{\beta}) |\vp_{h}(u)-\vp_{0}(u)| du + b_{n}^{r}h \int_{(0,\infty)} u^{r-1}\vp_{0}(u)du \nn\\
&\lesssim b_{n}^{r}h^{a_{\nu}} + b_{n}^{r}h \lesssim b_{n}^{r}h^{1\wedge a_{\nu}} = b_{n}^{r}h
\label{key.lemma-pr.2-4}
\end{align}
if $r\in(0,1]$; under the assumptions of the present Lemma 2.2(2), one can always take $\del>\beta$, hence $a_{\nu}=1$.
By \eqref{key.lemma-pr.2-1} and \eqref{key.lemma-pr.2-4} we obtain
\begin{equation}
\del_{n} \lesssim b_{n}^{-\beta} + b_{n}^{r}h.
\nonumber
\end{equation}
Optimizing the upper bound with respect to $b_{n}$ results in the choice $b_{n}\sim h^{-1/(\beta+r)}$, with which we conclude \eqref{key.lemma-1} since $r\in(0,1]$ was arbitrary.
We note that introducing the parameter $r>0$ is essential in the above estimates.

It remains to prove \eqref{key.lemma-pr.2-3}. 
Since $\rho(z)\equiv -1$ for $|z|>K$, partially differentiating with respect to $u$ under the integral sign we obtain
\begin{align}
\p_{u}\psi_{h}(u) 
& = 2\p_{u}\bigg( \int_{(0,Kh^{-1/\beta}]}(\cos(uy)-1) g_{0,\beta}(y)\{1+\rho(h^{1/\beta}y)\}dy \bigg)
\nn\\
&= -2c_{\beta}\int_{(0,Kh^{-1/\beta}]} \frac{\sin(uy)}{y^{\beta}}\rho(h^{1/\beta}y)dy 
-2c_{\beta}\int_{(0,Kh^{-1/\beta}]} \frac{\sin(uy)}{y^{\beta}}dy \nn\\
&=: R_{h}(u) + A_{h}(u).
\nonumber
\end{align}
It suffices to show that $|R_{h}(u)| \lesssim h/u$ and $|A_{h}(u) + \beta u^{\beta-1}| \lesssim h/u$ for $u>0$.
Write $\xi_{\beta}(y)=y^{-\beta}\rho(y)$. We have $R_{h} \equiv 0$ if $c_{\rho}=0$.
In case where $c_{\rho}>0$, thanks to Assumption \ref{A_J}(2)(b), the change of variables and the integration by parts yield that
\begin{align}
|R_{h}(u)| &\lesssim h^{1-1/\beta}\bigg| \int_{(0,K]}\sin(uh^{-1/\beta}x)\xi_{\beta}(x)dx \bigg| \nn\\
&= h^{1-1/\beta} \bigg|\int_{(0,K]}\p_{x}\bigg(\frac{\cos(uh^{-1/\beta}x)}{uh^{-1/\beta}}\bigg)\xi_{\beta}(x)dx \bigg| \nn\\
&\lesssim \frac{h}{u}\bigg( 1+|\xi_{\beta}(0+)| + \int_{(0,K]}|\p_{x}\xi_{\beta}(x)|dx \bigg) \lesssim \frac{h}{u}.
\nonumber
\end{align}
Turning to $A_{h}(u)$, we need the following specific identity from the Lebesgue integration theory \cite{Iwa15}:
for $r>0$ and $\beta\in(0,2)$, we have
\begin{equation}
\int_{(0,r)} \frac{\sin x}{x^{\beta}}dx - \Gam(1-\beta) \cos\bigg(\frac{\beta\pi}{2}\bigg)
= \frac{1}{\Gam(\beta)}\int_{(0,\infty)} \frac{e^{-ry}y^{\beta-1}(\cos r + y\sin r)}{1+y^{2}}dy.
\label{key.lemma-pr.2-5}
\end{equation}
From the definition \eqref{c.beta_def} and the property of the gamma function, we have the identity $\frac{\beta}{2c_{\beta}}=\Gam(1-\beta)\cos(\frac{\beta\pi}{2})$.
Applying \eqref{key.lemma-pr.2-5} together with the change of variables, we obtain
\begin{align}
|A_{h}(u) + \beta u^{\beta-1} |
&= \bigg| -2c_{\beta}u^{\beta-1}\int_{(0,uKh^{-1/\beta}]} \frac{\sin x}{x^{\beta}}dx + \beta u^{\beta-1} \bigg| \nn\\
&\lesssim u^{\beta-1} \bigg| \int_{(0,uKh^{-1/\beta})} \frac{\sin x}{x^{\beta}}dx - \Gam(1-\beta) \cos\bigg(\frac{\beta\pi}{2}\bigg) \bigg| \nn\\
&\lesssim u^{\beta-1} \bigg( 
\int_{(0,\infty)} \frac{e^{-ry}y^{\beta-1}}{1+y^{2}}dy + \int_{(0,\infty)} \frac{e^{-ry}y^{\beta}}{1+y^{2}}dy\bigg)\bigg|_{r=uKh^{-1/\beta}} \nn\\
&= u^{\beta-1} \bigg( r^{-\beta}\int_{(0,\infty)} \frac{e^{-x}x^{\beta-1}}{1+(x/r)^{2}}dx 
+ r^{-\beta-1}\int_{(0,\infty)} \frac{e^{-x}x^{\beta}}{1+(x/r)^{2}}dx\bigg)\bigg|_{r=uKh^{-1/\beta}} \nn\\
&\le u^{\beta-1} \bigg\{ r^{-\beta} \bigg(\int_{(0,\infty)} e^{-x}x^{\beta-1}dx 
+ \int_{(0,\infty)} \frac{re^{-x}x^{\beta}}{r^{2}+x^{2}}dx\bigg)\bigg\}\bigg|_{r=uKh^{-1/\beta}} \nn\\
&\lesssim \frac{h}{u} \bigg( 1 + \sup_{r>0}\int_{(0,\infty)} \frac{re^{-x}x^{\beta}}{r^{2}+x^{2}}dx\bigg) \lesssim \frac{h}{u}.
\nonumber
\end{align}
Here, in the last step we used that
\begin{align}
\bigg|\int_{(0,\infty)} \frac{re^{-x}x^{\beta}}{r^{2}+x^{2}}dx \bigg|
&= \bigg|\bigg[ \arctan\bigg(\frac{x}{r}\bigg)e^{-x}x^{\beta}\bigg]_{(0,\infty)} - \int_{(0,\infty)} \arctan\bigg(\frac{x}{r}\bigg)(\beta x^{\beta-1}-x^{\beta})e^{-x}dx\bigg| \nn\\
&\lesssim \int_{(0,\infty)} x^{\beta-1}e^{-x}dx + \int_{(0,\infty)} x^{\beta}e^{-x}dx <\infty
\nonumber
\end{align}
uniformly in $r>0$. Thus we have obtained \eqref{key.lemma-pr.2-3}, completing the proof of the claim (2).

\subsection{Proof of Lemma \ref{key.lemma.cf}}

It is enough to notice that combining Lemma \ref{key.lemma}(1), \eqref{key.lemma-pr.2-0.5}, \eqref{key.lemma-pr.2-1} and \eqref{key.lemma-pr.2-2} leads to
\begin{equation}
\int \left|f_{h}(y)-\phi_{\beta}(y)\right|dy \lesssim b_{n}^{-\beta} + (\ep_{\psi}(h) \vee h^{(\del/\beta)\wedge 1})b_{n}^{r},
\nonumber
\end{equation}
and that the upper bound is optimized (with respect to $b_{n}$) to be $(\ep_{\psi}(h) \vee h^{(\del/\beta)\wedge 1})^{\frac{\beta}{\beta+r}}$.

\section{Proofs of the main results}\label{sec_proofs}

This section is devoted to proving Theorems \ref{sqmle.iv_bda_thm} and \ref{sqmle.iv_bda_thm.cf}, Corollary \ref{sqmle.iv_bda_thm.cor}, and Theorem \ref{sqmle_ergo_thm}.

\subsection{Localization: elimination of large jumps}\label{sec_localization}


Prior to the proofs, we need to introduce a localization of the underlying probability space by eliminating possible large jumps of $J$.
Specifically, by means of \cite[Section 4.4.1]{JacPro12}, in order to prove Theorems \ref{sqmle.iv_bda_thm} and \ref{sqmle.iv_bda_thm.cf} and Corollary \ref{sqmle.iv_bda_thm.cor} we may and do suppose that
\begin{equation}
\exists K>0,\quad \pr\left(\forall t\in[0,T],~|\D J_{t}|\le K\right)=1,
\label{localization_1}
\end{equation}
(The arguments in \cite[Section 4.4.1]{JacPro12} partly concerns the stable convergence in law, which we will briefly mention in Section \ref{sec_proof.amn}).
The point here is that, since our main results are concerned with the weak properties over the fixed period $[0,T]$, we may conveniently focus on a subset $\Omega_{K,T}(\in\mcf)\subset\Omega$ on which jumps of $J$ are bounded by a constant $K$: $\sup_{\omega\in\Omega,\, t\le T}|\D J_{t}(\omega)|\le K$, the probability $\pr(\Omega_{K,T})$ being arbitrarily close to $1$ for $K$ large enough;
the simple yet very powerful localization device is standard in the context of limit theory for statistics based on high-frequency data \cite{Jac12}, and has been considered for quite general semimartingale models.
Note that the symmetry assumption of $\nu$ makes the parametric form of the drift coefficient unaffected by elimination of large jumps of $J$.

\medskip

For later use, we mention and recall some important consequences of either Assumption \ref{A_J} with \eqref{localization_1}, or Assumption \ref{A_J.cf}.

\begin{itemize}
\item Following the argument \cite[Section 2.1.5]{JacPro12} together with Gronwall's inequality under the global Lipschitz condition of $(a(\cdot,\al_{0}), c(\cdot,\gam_{0}))$, we see that
\begin{equation}
\E\bigg(\sup_{t\le T}|X_{t}|^{q}\bigg) \le C, \qquad
\sup_{t\in[s,s+h]\cap [0,T]}\E\left(|X_{t}-X_{s}|^{q}|\mcf_{s}\right) \lesssim h (1+|X_{s}|^{C})
\label{loc.X.ineq}
\end{equation}
for any $q\ge 2$ and $s\in[0,T]$; in particular,
\begin{equation}
\sup_{t\in[s,s+h]\cap [0,T]}\E\left(|X_{t}-X_{s}|^{q}\right)=O(h).
\nonumber
\end{equation}

\item There exists a constant $C_{0}>0$ such that $\int_{|z|>y}\nu(dz)\lesssim y^{-\beta}$ for $y\in(0,C_{0}]$, with which \cite[Theorem 2(a) and (c)]{LusPag08} gives
\begin{equation}
\E\left(\sup_{t\le h}|J_{t}|^{\kappa}\right)\lesssim h^{\kappa/\beta}
\label{LusPag.ineq.result}
\end{equation}
for each $\kappa\in(0,\beta)$.

\item The convergences \eqref{llt.imp.1} and \eqref{llt.imp.2} hold when $\rho(z)\equiv -1$ for $|z|>K$ (hence $\rho$ is bounded):
\begin{align}
\sqrt{n}\int|f_{h}(y)-\phi_{\beta}(y)|dy &\to 0, \nn\\
\int|y|^{\kappa}|f_{h}(y)-\phi_{\beta}(y)|dy &\to 0, \qquad \kappa\in[0,\beta).
\nn
\end{align}
\end{itemize}

\subsection{Preliminary asymptotics}\label{sec_pre.asymp}

Let us recall the notation $\ep_{j}(\theta)=\{h^{1/\beta}c_{j-1}(\gam)\}^{-1}(\D_{j}X-ha_{j-1}(\al))$.
Throughout this section, we look at asymptotic behavior of the auxiliary random function
\begin{equation}
U_{n}(\theta) := \sumj \pi_{j-1}(\theta) \eta(\ep_{j}(\theta)),
\nonumber
\end{equation}
where $\pi: \mbbr \times\overline{\Theta} \to \mbbr^{k}\otimes\mbbr^{m}$ and $\eta:\mbbr\to\mbbr^{m}$ are measurable functions. This form of $U_{n}(\theta)$ will appear in common in the proofs of the consistency and asymptotic (mixed) normality of the SQMLE. The results in this section will be repeatedly used in the subsequent sections.

Let $\E^{j-1}(\cdot)$ be a shorthand for $\E(\cdot|\mcf_{t_{j-1}})$ and write $U_{n}(\theta)=U_{1,n}(\theta)+U_{2,n}(\theta)$, where
\begin{align}
U_{1,n}(\theta) &:=\sumj \pi_{j-1}(\theta) \Big(\eta(\ep_{j}(\theta)) - \E^{j-1}\{\eta(\ep_{j}(\theta))\} \Big), \nn\\
U_{2,n}(\theta) &:=\sumj \pi_{j-1}(\theta) \E^{j-1}\{\eta(\ep_{j}(\theta))\}.
\nonumber
\end{align}
Given doubly indexed random functions $F_{nj}(\theta)$ on $\overline{\Theta}$, a positive sequence $(a_{n})$, and a constant $q>0$, we will write
\begin{equation}
\text{$F_{nj}(\theta)=O^{\ast}_{L^{q}}(a_{n})$ \quad if \quad $\ds{\sup_{n}\sup_{j\le n}\E\bigg(\sup_{\theta}|a_{n}^{-1}F_{nj}(\theta)|^{q}\bigg)<\infty}$.}
\nonumber
\end{equation}

\subsubsection{Uniform estimate of the martingale part $U_{1,n}$}

\begin{lem}
Suppose that:
\begin{itemize}
\item[{\rm (i)}] $\pi\in\mcc^{1}(\mbbr\times\Theta)$ and $\sup_{\theta}\left\{|\pi(x,\theta)|+|\p_{\theta}\pi(x,\theta)|\right\}\lesssim 1+|x|^{C}$;
\item[{\rm (ii)}] $\eta\in\mcc^{1}(\mbbr)$ and $|\eta(y)| + |y||\p\eta(y)| \lesssim 1+\log(1+|y|)$.
\end{itemize}
Then, for every $q>0$ we have $U_{1,n}(\theta)=O^{\ast}_{L^{q}}(\sqrt{n})$, hence in particular 
\begin{equation}
\sup_{\theta}\bigg|\frac{1}{nh^{1-1/\beta}}U_{1,n}(\theta)\bigg| = O_{p}\left((\sqrt{n}h^{1-1/\beta})^{-1}\right)=o_{p}(1).
\label{lem_aux.se1+1}
\end{equation}
\label{lem_aux.se1}
\end{lem}

\begin{proof}
Since we are assuming that the parameter space $\Theta$ is a bounded convex domain, the Sobolev inequality \cite[p.415]{Ada73} is in force: for each $q>p\vee 2$,
\begin{equation}
\E\left(\sup_{\theta}|n^{-1/2}U_{1,n}(\theta)|^{q}\right) \lesssim 
\sup_{\theta}\E\left(|n^{-1/2}U_{1,n}(\theta)|^{q}\right)+\sup_{\theta}\E\left(|n^{-1/2}\p_{\theta}U_{1,n}(\theta)|^{q}\right).
\nonumber
\end{equation}
To complete the proof, it therefore suffices to show that both $\{n^{-1/2}U_{1,n}(\theta)\}$ and $\{n^{-1/2}\p_{\theta}U_{1,n}(\theta)\}$ are $L^{q}$-bounded for each $\theta$ and $q>p\vee 2$. Fix any $q>p\vee 2$ and $\theta$ in the rest of this proof.

Put $\chi_{j}(\theta)=\pi_{j-1}(\theta) \left(\eta(\ep_{j}(\theta)) - \E^{j-1}\{\eta(\ep_{j}(\theta))\} \right)$, so that $U_{1,n}(\theta)=\sumj \chi_{j}(\theta)$. Under the present regularity conditions we may pass the differentiation with respect to $\theta$ through the operator $\E^{j-1}$:
\begin{align}
\p_{\theta}\chi_{j}(\theta)
&=\p_{\theta}\pi_{j-1}(\theta) \Big( \eta(\ep_{j}(\theta)) - \E^{j-1}\{\eta(\ep_{j}(\theta))\} \Big) \nn\\
&{}\qquad+ \pi_{j-1}(\theta) \Big( \p\eta(\ep_{j}(\theta))\p_{\theta}\ep_{j}(\theta) - \E^{j-1}\{\p\eta(\ep_{j}(\theta))\p_{\theta}\ep_{j}(\theta)\} \Big).
\label{lem_aux.se1_4}
\end{align}
For each $n$, the sequences $\{\chi_{j}(\theta)\}_{j}$ and $\{\p_{\theta}\chi_{j}(\theta)\}_{j}$ form martingale difference arrays with respect to $(\mcf_{t_{j}})$, hence Burkholder's inequality gives $\E\{|n^{-1/2}\p_{\theta}^{k}U_{n}(\theta)|^{q}\}\lesssim n^{-1}\sumj \E\{|\p_{\theta}^{k}\chi_{j}(\theta)|^{q}\}$ for $k=0,1$. The required $L^{q}$-boundedness of $\{n^{-1/2}\p_{\theta}^{k}U_{1,n}(\theta)\}$ follows on showing that $\sup_{j\le n}\E(|\p_{\theta}^{k}\chi_{j}(\theta)|^{q})\lesssim 1$.

Observe that for $\beta\ge 1$ and $r\in(0,\beta)$,
\begin{align}
|\ep_{j}(\theta)|^{r}
&= \left| h^{-1/\beta}c_{j-1}^{-1}(\gam)\{\D_{j}X-ha_{j-1}(\al)\}\right|^{r} \nn\\
&\lesssim (1+|X_{t_{j-1}}|^{C})\left\{ |h^{-1/\beta}\D_{j}X|^{r} + h^{r(1-1/\beta)}(1+|X_{t_{j-1}}|^{C}) \right\} \nn\\
&\lesssim (1+|X_{t_{j-1}}|^{C})\left( |h^{-1/\beta}\D_{j}X|^{r} + 1 \right).
\label{lem_aux.se1_2}
\end{align}
Applying the estimate \eqref{LusPag.ineq.result} together with the linear growth property of $a(\cdot,\al_{0})$, the Lipschitz property of $c(\cdot,\gam_{0})$, the estimate \eqref{loc.X.ineq}, and Burkholder's inequality for the stochastic integral with respect to $J$, we derive the chain of inequalities:
\begin{align}
\E^{j-1}\left(|h^{-1/\beta}\D_{j}X|^{r}\right) 
&\lesssim h^{r(1-1/\beta)}\bigg(\frac{1}{h}\int_{j}\E^{j-1}\{|a(X_{s},\al_{0})|^{2}\}ds\bigg)^{r/2} \nn\\
&{}\qquad + h^{-r/\beta}\E^{j-1}\bigg(\bigg|\int_{j}(c(X_{s},\gam_{0})-c_{j-1}(\gam_{0}))dJ_{s}\bigg|^{r}\bigg) \nn\\
&{}\qquad + (1+|X_{t_{j-1}}|^{C})\E(|h^{-1/\beta}J_{h}|^{r}) \nn\\
&\lesssim (1+h^{r(1-1/\beta)})(1+|X_{t_{j-1}}|^{C}) + h^{-r/\beta}\bigg(\int_{j}\E^{j-1}(|X_{s}-X_{t_{j-1}}|^{2}) ds\bigg)^{r/2} \nn\\
&\lesssim (1+h^{r(1-1/\beta)})(1+|X_{t_{j-1}}|^{C}) + h^{-r/\beta}\left\{ h^{2}(1+|X_{t_{j-1}}|^{C}) \right\}^{r/2} \nn\\
&\lesssim 1+|X_{t_{j-1}}|^{C}.
\label{lem_aux.se1_3}
\end{align}
Using \eqref{lem_aux.se1_2} and \eqref{lem_aux.se1_3}, we arrive at the estimate
\begin{equation}
\E\left\{ (1+|X_{t_{j-1}}|^{C})|\ep_{j}(\theta)|^{r} \right\} \lesssim 1+\sup_{t\le T}\E(|X_{t}|^{C}) \lesssim 1
\nonumber
\end{equation}
valid for $r\in(0,\beta)$. By means of the condition on $\eta$,
\begin{align}
\E(|\chi_{j}(\theta)|^{q})
&\lesssim \E\big[(1+|X_{t_{j-1}}|^{C}) \E^{j-1}\{|\eta(\ep_{j}(\theta))|^{q}\}\big] \nn\\
&\lesssim \E\big[(1+|X_{t_{j-1}}|^{C}) \big(1+\E^{j-1}[\{\log(1+|\ep_{j}(\theta)|)\}^{q}]\big)\big] \nn\\
&\lesssim \E\big[(1+|X_{t_{j-1}}|^{C}) \big(1+\E^{j-1}\{|\ep_{j}(\theta)|^{r}\}\big)\big] \nn\\
&\lesssim 1+\sup_{t\le T}\E(|X_{t}|^{C}),
\nonumber
\end{align}
concluding that $\sup_{j\le n}\E(|\chi_{j}(\theta)|^{q}) \lesssim 1$.

Next we note that
\begin{equation}
\p_{\al}\ep_{j}(\theta) = -h^{1-1/\beta}\frac{\p_{\al}a_{j-1}(\al)}{c_{j-1}(\gam)},\qquad
\p_{\gam}\ep_{j}(\theta) = -\frac{\p_{\gam}c_{j-1}(\gam)}{c_{j-1}(\gam)}\ep_{j}(\theta).
\nonumber
\end{equation}
By \eqref{lem_aux.se1_4}, the components of $\p_{\theta}\chi_{j}(\theta)$ consists of the terms
\begin{align}
& \pi_{j-1}^{(1)}(\theta) \Big( \eta(\ep_{j}(\theta)) - \E^{j-1}\{\eta(\ep_{j}(\theta))\} \Big), \nn\\
& \pi_{j-1}^{(2)}(\theta) \Big( \p\eta(\ep_{j}(\theta)) - \E^{j-1}\{\p\eta(\ep_{j}(\theta))\} \Big), \nn\\
& \pi_{j-1}^{(3)}(\theta) \Big( \ep_{j}(\theta)\p\eta(\ep_{j}(\theta)) - \E^{j-1}\{\ep_{j}(\theta)\p\eta(\ep_{j}(\theta))\} \Big)
\nonumber
\end{align}
for some $\pi^{(i)}(x,\theta)$, $i=1,2,3$, all satisfying the conditions imposed on $\pi(x,\theta)$. Again taking the conditions on $\eta$ into account, we can proceed as in the previous paragraph to obtain $\sup_{j\le n}\E(|\p_{\theta}\chi_{j}(\theta)|^{q}) \lesssim 1$. The proof is complete.
\end{proof}

\subsubsection{Uniform estimate of the predictable (compensator) part $U_{2,n}$}

Introduce the notation:
\begin{align}
& \del'_{j}(\gam) = \frac{c_{j-1}(\gam_{0})}{c_{j-1}(\gam)}h^{-1/\beta}\D_{j}J, \qquad \mfb(x,\theta) =c^{-1}(x,\gam)\{a(x,\al_{0})-a(x,\al)\}, \nn\\
& a^{\D}_{j-1}(s) = a(X_{s},\al_{0})-a_{j-1}(\al_{0}), \qquad c^{\D}_{j-1}(s) = c(X_{s},\gam_{0})-c_{j-1}(\gam_{0}), \nn\\
& r_{j}(\gam)=\frac{h^{-1/\beta}}{c_{j-1}(\gam)}\int_{j} a^{\D}_{j-1}(s) ds 
+ \frac{h^{-1/\beta}}{c_{j-1}(\gam)}\int_{j} c^{\D}_{j-1}(s-) dJ_{s}.
\nonumber
\end{align}
Then
\begin{align}
\ep_{j}(\theta)=\del'_{j}(\gam) + h^{1-1/\beta}\mfb_{j-1}(\theta) + r_{j}(\gam).
\nn\label{epj.form}
\end{align}
Expanding $\eta$ we have
\begin{equation}
U_{2,n}(\theta)=U_{2,n}^{0}(\theta)+U_{2,n}'(\theta)+U_{2,n}''(\theta),
\label{U2n_deco}
\end{equation}
where, with $\overline{r}_{j}(\theta;\eta):=\int_{0}^{1}\p\eta(\del'_{j}(\gam) + h^{1-1/\beta}\mfb_{j-1}(\theta) + sr_{j}(\gam))ds$ and $\pi'(x,\theta):=\pi(x,\theta)c^{-1}(x,\gam)$,
\begin{align}
U_{2,n}^{0}(\theta) &:= \sumj \pi_{j-1}(\theta) \E^{j-1}\left\{ \eta\left( \del'_{j}(\gam)+h^{1-1/\beta}\mfb_{j-1}(\theta) \right) \right\},
\nn\\
U_{2,n}'(\theta) &:= h^{-1/\beta}\sumj \pi_{j-1}'(\theta) \E^{j-1}\bigg( \overline{r}_{j}(\theta;\eta)\,\int_{j}a^{\D}_{j-1}(s)ds \bigg),
\nn\\
U_{2,n}''(\theta) &:= h^{-1/\beta}\sumj \pi_{j-1}'(\theta) \E^{j-1}\bigg( \overline{r}_{j}(\theta;\eta)\,\int_{j}c^{\D}_{j-1}(s-)dJ_{s} \bigg).
\nonumber
\end{align}
A uniform law of large numbers for $(nh^{1-1/\beta})^{-1}U_{2,n}(\theta)$ will be one of the key ingredients in the proofs. Lemma \ref{lem_aux.se_r1} below reveals that the terms $U_{2,n}'(\theta)$ and $U_{2,n}''(\theta)$ have no contribution in the limit; we will deal with the remaining term $U_{2,n}^{0}(\theta)$ in Section \ref{sec_ulln}.

\medskip

Let us recall It\^{o}'s formula, which is valid for any $\mcc^{\beta}$-function\footnote{In case of $\beta\in(1,2)$, this means that $\psi$ is $\mcc^{1}$ and the derivative $\p\psi$ is locally H\"older continuous with index $\beta-[\beta]$.} $\psi$ (see \cite[Theorems 3.2.1b) and 3.2.2a)]{JacPro12}): for $t>s$,
\begin{align}
\psi(X_{t}) &= \psi(X_{s}) + \int_{s}^{t}\p\psi(X_{u-})dX_{u} \nn\\
&{}\qquad
+\int_{s}^{t}\int \left\{\psi(X_{u-}+c(X_{u-},\gam_{0})z) -\psi(X_{u-}) -\p\psi(X_{u-})c(X_{u-},\gam_{0})z\right\}\mu(du,dz).
\nn
\end{align}
Let $\mca$ denote the formal infinitesimal generator of $X$:
\begin{align}
\mca\psi(x) = \p\psi(x)a(x,\al_{0}) + \int\left\{\psi(x+c(x,\gam_{0})z) -\psi(x) -\p\psi(x)c(x,\gam_{0})z\right\}\nu(dz),
\nonumber
\end{align}
the second term in the right-hand side being assumed well-defined. Then
\begin{equation}
\psi(X_{t}) = \psi(X_{s}) + \int_{s}^{t}\mca\psi(X_{u})du + \int_{s}^{t}\int \left\{\psi(X_{u-}+c(X_{u-},\gam_{0})z) -\psi(X_{u-}) \right\}\tilde{\mu}(du,dz).
\label{ito_C.beta_A}
\end{equation}
Obviously, we have $|\mca\psi(x)|\lesssim 1+|x|^{C}$ for $\psi$ such that the derivatives $\p^{k}\psi$ for $k\in\{0,1,2\}$ 
exist and have polynomial majorants.

\begin{lem}
Suppose that:
\begin{itemize}
\item[{\rm (i)}] $\pi\in\mcc^{1}(\mbbr\times\Theta)$ and $\sup_{\theta}\{|\pi(x,\theta)|+|\p_{\theta}\pi(x,\theta)|\}\lesssim 1+|x|^{C}$;
\item[{\rm (ii)}] $\eta\in\mcc^{1}(\mbbr)$ with bounded first derivative.
\end{itemize}
Then we have $U_{2,n}'(\theta)=O_{L^{q}}^{\ast}(nh^{2-1/\beta})$ and $U_{2,n}''(\theta)=O_{L^{q}}^{\ast}(nh^{2-1/\beta})$ for every $q>0$. In particular, we have $\sup_{\theta}|n^{-1/2}U_{2,n}'(\theta)|=o_{p}(1)$ and $\sup_{\theta}|n^{-1/2}U_{2,n}''(\theta)|=o_{p}(1)$.
\label{lem_aux.se_r1}
\end{lem}

\begin{proof}
In this proof, $q$ denotes any positive real greater than or equal to $2$. We begin with $U'_{2,n}(\theta)$. 
Applying \eqref{ito_C.beta_A} with $\psi(x)=a(x,\al_{0})$ and then taking the conditional expectation, we get
\begin{align}
\bigg|\E^{j-1}\bigg(\int_{j}a^{\D}_{j-1}(s)ds\bigg)\bigg|
&= \bigg|\int_{j}\E^{j-1}\{a^{\D}_{j-1}(s)\}ds\bigg| \nn\\
&\le \int_{j}\int_{t_{j-1}}^{s}\E^{j-1}\left\{|\mca a(X_{u},\al_{0})|\right\}duds \nn\\
&\lesssim \int_{j}\int_{t_{j-1}}^{s}\left\{ 1+\E^{j-1}(|X_{u}|^{C})\right\}duds \nn\\
&\lesssim \int_{j}\int_{t_{j-1}}^{s}(1+|X_{t_{j-1}}|^{C})duds = O^{\ast}_{L^{q}}(h^{2}).
\label{U'_so.1}
\end{align}
Write $m_{j}(\theta;\eta) = \overline{r}_{j}(\theta;\eta) - \E^{j-1}\{\overline{r}_{j}(\theta;\eta)\}$ and $\tilde{a}^{\D}_{j-1}(s) = a^{\D}_{j-1}(s)-\E^{j-1}\{a^{\D}_{j-1}(s)\}$. 
Using \eqref{U'_so.1} and noting that $\overline{r}_{j}(\theta;\eta)$ is essentially bounded, we get
\begin{align}
U_{2,n}'(\theta) &= h^{-1/\beta}\sumj \pi_{j-1}'(\theta) \E^{j-1}\bigg( m_{j}(\theta;\eta)\int_{j}\tilde{a}^{\D}_{j-1}(s)ds \bigg) +O^{\ast}_{L^{q}}(nh^{2-1/\beta}) \nn\\
&= h^{-1/\beta}\sumj \pi_{j-1}'(\theta) \int_{j}\E^{j-1}\left\{ m_{j}(\theta;\eta)\tilde{a}^{\D}_{j-1}(s) \right\}ds +O^{\ast}_{L^{q}}(nh^{2-1/\beta}).
\nn
\end{align}
By Jensen's inequality, the claim $U'_{2,n}(\theta)=O_{L^{q}}^{\ast}(nh^{2-1/\beta})$ follows if we show
\begin{equation}
\sup_{n}\sup_{j\le n}\sup_{s\in[t_{j-1}, t_{j}]}\E\bigg(\sup_{\theta}\bigg|
\frac{1}{h}\E^{j-1}\left\{ m_{j}(\theta;\eta)\tilde{a}^{\D}_{j-1}(s) \right\}\bigg|^{q}
\bigg) \lesssim 1.
\label{U'_so.3}
\end{equation}
By \eqref{ito_C.beta_A} we may express $\tilde{a}^{\D}_{j-1}(s)$ as
\begin{equation}
\tilde{a}^{\D}_{j-1}(s)=\int_{t_{j-1}}^{s}f(X_{t_{j-1}},X_{u})du + \int_{t_{j-1}}^{s}\int g(X_{u-},z) \tilde{\mu}(du,dz),
\label{U'_so.6}
\end{equation}
where $\E^{j-1}\{f(X_{t_{j-1}},X_{u})\}=0$ with $f(x,x')$ being at most of polynomial growth in $(x,x')$, and where
\begin{equation}
g(x,z):=a(x+zc(x,\gam_{0}),\al_{0})-a(x,\al_{0}).
\nonumber
\end{equation}
Hence, for \eqref{U'_so.3} it suffices to prove
\begin{equation}
\sup_{n}\sup_{j\le n}\sup_{s\in[t_{j-1}, t_{j}]}\E\bigg(\sup_{\theta}\bigg|
\frac{1}{h}\E^{j-1}\bigg( m_{j}(\theta;\eta) \int_{t_{j-1}}^{s}\int g(X_{u-},z) \tilde{\mu}(du,dz)\bigg|^{q}
\bigg) <\infty.
\label{U'_so.4}
\end{equation}
Let $H_{j,t}(\theta;\eta):=\E\left\{m_{j}(\theta;\eta)|\,\mcf_{t}\right\}$ for $t\in[t_{j-1},t_{j}]$; then, $H_{j,t_{j}}(\theta;\eta)=m_{j}(\theta;\eta)$. Recall we are supposing \eqref{Ft_def}: $\mcf_{t}=\sig(X_{0})\vee\sig(J_{s};s\le t)$.
By its construction, $\{H_{j,t}(\theta;\eta),\mcf_{t_{j-1}}\vee\sig(J_{t});\,t\in[t_{j-1},t_{j}]\}$ is an essentially bounded martingale. According to the martingale representation theorem \cite[Theorem III.4.34]{JacShi03}, the process $H_{j,t}(\theta)$ can be represented as a stochastic integral of the form
\begin{equation}
H_{j,t}(\theta;\eta) = \int_{t_{j-1}}^{t}\int\xi_{j}(s,z;\theta)\tilde{\mu}(ds,dz),\qquad t\in[t_{j-1},t_{j}],
\label{U'_so.5}
\end{equation}
with a bounded 
predictable process $s\mapsto\xi_{j}(s,z;\theta)$ such that
\begin{equation}
\sup_{n}\sup_{j\le n}\sup_{s\in[t_{j-1}, t_{j}]}\sup_{\theta}\E^{j-1}\bigg(\int \xi_{j}^{2}(s,z;\theta)\nu(dz)\bigg) <\infty.
\nonumber
\end{equation}
Now, we look at the quantity inside the absolute value sign $|\cdots|$ in the left-hand side of \eqref{U'_so.4}.
By conditioning with respect to $\mcf_{s}$ inside the sign ``$\E^{j-1}$'', substituting the expression \eqref{U'_so.5} with $t=s$, and then applying the integration-by-parts formula for martingales, it follows that the quantity equals
\begin{equation}
\frac{1}{h}\E^{j-1}\bigg( \int_{t_{j-1}}^{s}\int \xi_{j}(u,z;\theta)g(X_{u-},z) \nu(dz)du \bigg).
\nonumber
\end{equation}
By the regularity conditions on $a(x,\al_{0})$ and $c(x,\gam_{0})$ we have $|g(x,z)| \lesssim |z|(1+|x|)$.
It follows from this bound together with \eqref{loc.X.ineq} and Jensen and Cauchy-Schwarz inequalities that
\begin{align}
& \E\bigg\{ \sup_{\theta}\bigg|\frac{1}{h}\E^{j-1}\bigg( \int_{t_{j-1}}^{s}\int \xi_{j}(u,z;\theta)g(X_{u-},z) \nu(dz)du \bigg)\bigg|^{q} \bigg\} 
\nn\\
&\lesssim \frac{1}{h}\int_{t_{j-1}}^{s}\E\bigg\{ \sup_{\theta} \E^{j-1}\bigg(\bigg| \int \xi_{j}(u,z;\theta)g(X_{u-},z) \nu(dz) \bigg|\bigg)^{q} \bigg\} du \nn\\
&\lesssim \frac{1}{h}\int_{t_{j-1}}^{s}\E\bigg[ \bigg\{\sup_{\theta} \E^{j-1}\bigg(\int \xi_{j}^{2}(u,z;\theta) \nu(dz)\bigg)\bigg\}^{q/2}
\bigg\{\E^{j-1}\bigg(\int g^{2}(X_{u-},z) \nu(dz)\bigg)\bigg\}^{q/2}\bigg] du \nn\\
& \lesssim \frac{1}{h}\int_{t_{j-1}}^{s}\E\left\{ \E^{j-1}\left( 1+ |X_{u}|^{C} \right) \right\} du \nn\\
&\lesssim \frac{1}{h}\int_{t_{j-1}}^{s}\E\left( 1+ |X_{t_{j-1}}|^{C} \right)du \lesssim 1.
\nonumber
\end{align}
This proves \eqref{U'_so.4}, concluding that $U'_{2,n}(\theta)=O_{L^{q}}^{\ast}(nh^{2-1/\beta})$.

\medskip

Next we consider $U_{2,n}''(\theta)$. Using the martingale representation for $m_{j}(\theta;\eta)$ as before, we have
\begin{align}
U_{2,n}''(\theta) &= h^{-1/\beta}\sumj \pi_{j-1}'(\theta) 
\E^{j-1}\bigg(m_{j}(\theta;\eta)\,\int_{j}c^{\D}_{j-1}(s-)dJ_{s}\bigg)
\nn\\
&{}\qquad + h^{-1/\beta}\sumj \pi_{j-1}'(\theta) \E^{j-1}\{\overline{r}_{j}(\theta;\eta)\} \E^{j-1}\bigg(\int_{j}c^{\D}_{j-1}(s-)dJ_{s} \bigg)
\nn\\
&=h^{-1/\beta}\sumj \pi_{j-1}'(\theta) \E^{j-1}\bigg( \D_{j}H_{j}(\theta;\eta)\,\int_{j}c^{\D}_{j-1}(s-)dJ_{s}\bigg) \nn\\
&=h^{-1/\beta}\sumj \pi_{j-1}'(\theta) \E^{j-1}\bigg( \int_{j}\int\xi_{j}(s,z;\theta)\tilde{\mu}(ds,dz)\,\int_{j}\int c^{\D}_{j-1}(s-)z \tilde{\mu}(ds,dz)\bigg)
\nn\\
&=h^{-1/\beta}\sumj \pi_{j-1}'(\theta) \E^{j-1}\bigg( \int_{j}\int \xi_{j}(s,z;\theta)z c^{\D}_{j-1}(s)\nu(dz)ds \bigg).
\nn
\end{align}
As in the case of $a^{\D}_{j-1}$, we have $|\E^{j-1}\{c^{\D}_{j-1}(s)\}| \le \int_{t_{j-1}}^{s}\E^{j-1}\{|\mca c(X_{u},\gam_{0})|\}du=O_{L^{q}}^{\ast}(h)$. Hence
\begin{align}
U_{2,n}''(\theta) 
&=h^{-1/\beta}\sumj \pi_{j-1}'(\theta) \int_{j}\E^{j-1}\left( \tilde{\Xi}_{j,s}(\theta)\tilde{c}^{\D}_{j-1}(s)\right) ds + O_{L^{q}}^{\ast}(nh^{2-1/\beta}),
\label{U''_so.2}
\end{align}
where $\tilde{\Xi}_{j,s}(\theta):=\int \xi_{j}(s,z;\theta)z\nu(dz) - \E^{j-1}\{\int \xi_{j}(s,z;\theta)z\nu(dz)\}$ for $s\in[t_{j-1},t_{j}]$ and $\tilde{c}^{\D}_{j-1}(s) := c^{\D}_{j-1}(s)-\E^{j-1}\{c^{\D}_{j-1}(s)\}$.
We have $\sup_{s\in[t_{j-1}, t_{j}]}\sup_{\theta}\E^{j-1}\{|\tilde{\Xi}_{j,s}(\theta)|^{2}\} \lesssim 1$ and $\tilde{c}^{\D}_{j-1}(s)$ admits a similar representation to \eqref{U'_so.6}.
Now we once more apply the martingale representation theorem: for each $j$ and $s\in[t_{j-1},t_{j}]$, the processes $M^{\prime j}_{u}(\theta):=\E^{j-1}\{\tilde{\Xi}_{j,s}(\theta)|\mcf_{u}\}$ and $M^{\prime\prime j}_{u}:=\E^{j-1}\{\tilde{c}^{\D}_{j-1}(s)|\mcf_{u}\}$ for $u\in[t_{j-1},s]$ are martingales with respect to the filtration $\{\mcf_{t_{j-1}}\vee\sig(J_{u}):\, u\in[t_{j-1},s]\}$, hence there correspond predictable processes $m^{\prime j}_{u}(z;\theta)$ and $m^{\prime\prime j}_{u}(z)$ such that $M^{\prime j}_{s}(\theta) = \int_{t_{j-1}}^{s}\int m^{\prime j}_{u}(z;\theta)\tilde{\mu}(du,dz)$ and $M^{\prime\prime j}_{s} = \int_{t_{j-1}}^{s}\int m^{\prime\prime j}_{u}(z)\tilde{\mu}(du,dz)$,
and that $\sup_{\theta}\E^{j-1}\{\int (m^{\prime j}_{u}(z;\theta))^{2}\nu(dz)\} \vee \E^{j-1}\{\int (m^{\prime\prime j}_{u}(z))^{2}\nu(dz)\} \lesssim 1$.
Thus, using the integration by parts formula as before we can rewrite \eqref{U''_so.2} as
\begin{align}
U_{2,n}''(\theta) &= nh^{2-1/\beta} \cdot \frac{1}{n}\sumj \pi_{j-1}'(\theta)\frac{1}{h^{2}}\int_{j}\int_{t_{j-1}}^{s}
\E^{j-1}\bigg(\int m^{\prime j}_{u}(z;\theta)m^{\prime\prime j}_{u}(z)\nu(dz)\bigg)duds \nn\\
&{}\qquad + O_{L^{q}}^{\ast}(nh^{2-1/\beta}).
\nonumber
\end{align}
We can apply Cauchy-Schwarz inequality to conclude that the first term in the right-hand side is $O_{L^{q}}^{\ast}(nh^{2-1/\beta})$, hence so is $U_{2,n}''(\theta)$.

Since $\sqrt{n}h^{2-1/\beta}\lesssim h^{3/2-1/\beta}\to 0$ for $\beta\in[1,2)$, the last part of the lemma is trivial.
The proof is thus complete.
\end{proof}

\subsubsection{Uniform law of large numbers}\label{sec_ulln}

Building on Lemmas \ref{lem_aux.se1} and \ref{lem_aux.se_r1}, we now turn to the uniform law of large numbers for $U_{n}(\theta)=U_{1,n}(\theta)+U_{2,n}(\theta)$. First we note the following auxiliary result.

\begin{lem}
For any measurable function $f:\mbbr\times\overline{\Theta}\to\mbbr$ such that
\begin{equation}
\sup_{\theta}\left\{ |f(x,\theta)| + |\p_{x}f(x,\theta)| \right\} \lesssim 1+|x|^{C},
\nonumber
\end{equation}
we have $(h=T/n)$
\begin{equation}
\sup_{\theta}\sup_{t\le T}\bigg|\frac{1}{n}\sum_{j=1}^{[t/h]} f(X_{t_{j-1}},\theta)
-\frac{1}{T}\int_{0}^{t}f(X_{s},\theta)ds\bigg|\cip 0.
\nonumber
\end{equation}
\label{lem.aux.LLN}
\end{lem}

\begin{proof}
The target quantity can be bounded by
\begin{align}
& \sup_{t\le T}\frac{1}{n}\sum_{j=1}^{[t/h]}\frac{1}{h}\int_{j}\sup_{\theta}|f(X_{s},\theta)-f_{j-1}(\theta)|ds 
+ \frac{h}{T}\sup_{\theta}\sup_{t\le T}|f(X_{t},\theta)| \nn\\
&\lesssim \frac{1}{n}\sum_{j=1}^{n}\frac{1}{h}\int_{j}(1+|X_{t_{j-1}}|+|X_{s}|)^{C}|X_{s}-X_{t_{j-1}}|ds
+ \frac{h}{T}\bigg( 1+\sup_{t\le T}|X_{t}|^{C}\bigg).
\nonumber
\end{align}
By \eqref{loc.X.ineq} the expectation of the upper bound tends to zero, hence the claim.
\end{proof}

\begin{prop}
Assume that the conditions in Lemma \ref{lem_aux.se1} hold and that $\sup_{\theta}|\p_{x}\pi(x,\theta)|\lesssim 1+|x|^{C}$.
\begin{enumerate}
\item If $\beta=1$, we have
\begin{equation}
\sup_{\theta}\bigg|
\frac{1}{n}U_{n}(\theta) - \frac{1}{T}\int_{0}^{T}
\pi(X_{t},\theta)\int\eta\bigg(\frac{c(X_{t},\gam_{0})}{c(X_{t},\gam)}z + \mfb(X_{t},\theta)\bigg)\phi_{1}(z)dzdt
\bigg|=o_{p}(1).
\nonumber
\end{equation}
\item If $\beta\in(1,2)$, we have
\begin{equation}
\sup_{\theta}\bigg|
\frac{1}{n}U_{n}(\theta) - 
\frac{1}{T}\int_{0}^{T}\pi(X_{t},\theta)\eta\bigg(\frac{c(X_{t},\gam_{0})}{c(X_{t},\gam)}z\bigg)\phi_{\beta}(dz)dzdt
\bigg|=o_{p}(1).
\nonumber
\end{equation}
If further $\eta$ is odd, then
\begin{equation}
\sup_{\theta}\bigg|\frac{1}{nh^{1-1/\beta}}U_{n}(\theta)\bigg| = O_{p}(1).
\nonumber
\end{equation}
\end{enumerate}
\label{prop_mainULLN}
\end{prop}

\begin{proof}
By Lemmas \ref{lem_aux.se1} and \ref{lem_aux.se_r1} it suffices to only look at $U_{2,n}^{0}(\theta)$ (recall \eqref{U2n_deco}); the assumptions in Lemma \ref{lem_aux.se_r1} are implied by those in Lemma \ref{lem_aux.se1}. Let
\begin{equation}
\overline{U}_{2,n}^{0}(\theta) := \frac{1}{nh^{1-1/\beta}}U_{2,n}^{0}(\theta)
=\frac{1}{n}\sumj \pi_{j-1}(\theta)\frac{1}{h^{1-1/\beta}}
\E^{j-1}\left\{ \eta\left( \del'_{j}(\gam)+h^{1-1/\beta}\mfb_{j-1}(\theta) \right) \right\}.
\nn
\end{equation}

\medskip

(1) For $\beta=1$, we can write $\overline{U}_{2,n}^{0}(\theta)$ as the sum of $n^{-1}\sumj f^{1}_{j-1}(\theta)$ and $n^{-1}\sumj f^{2}_{j-1}(\theta)$, where
\begin{align}
f^{1}(x,\theta) &:= \pi(x,\theta)\int\eta\bigg(\frac{c(x,\gam_{0})}{c(x,\gam)}z + \mfb(x,\theta)\bigg)\phi_{1}(z)dz,
\nn\\
f^{2}(x,\theta) &:= \pi(x,\theta)\int\eta\bigg(\frac{c(x,\gam_{0})}{c(x,\gam)}z + \mfb(x,\theta)\bigg)\{f_{h}(z)-\phi_{1}(z)\}dz.
\nonumber
\end{align}
Pick a $\kappa\in(0,\beta)=(0,1)$. Since $|\eta(y)| \lesssim 1 + |y|^{\kappa}$,
\begin{equation}
\sup_{\theta}\bigg| \eta\bigg(\frac{c(x,\gam_{0})}{c(x,\gam)}z + \mfb(x,\theta)\bigg) \bigg| \lesssim (1 + |x|^{C}) (1 + |z|^{\kappa}).
\nn
\end{equation}
Hence we have the bounds:
\begin{equation}
\sup_{\theta}|f^{1}(x,\theta)| \lesssim (1+|x|^{C})\int (1+|z|^{\kappa})\phi_{1}(y)dy \lesssim 1+|x|^{C}
\label{prop_mainULLN_1.25}
\end{equation}
and $|f^{2}(x,\theta)| \lesssim (1+|x|^{C})\int (1+|z|^{\kappa})|f_{h}(y)-\phi_{1}(y)|dy = (1+|x|^{C})o(1)$; in particular,
\begin{equation}
\sup_{\theta}\bigg|\frac{1}{n}\sumj f^{2}_{j-1}(\theta)\bigg|=o_{p}(1).
\label{prop_mainULLN_1.5}
\end{equation}
Under the conditions on $\eta$, simple manipulations lead to
\begin{align}
& \sup_{\theta}\bigg| \p\eta\bigg(\frac{c(x,\gam_{0})}{c(x,\gam)}z + \mfb(x,\theta)\bigg) \cdot \bigg\{ \p_{x}\bigg(\frac{c(x,\gam_{0})}{c(x,\gam)}\bigg) z + \p_{x}\mfb(x,\theta)\bigg\} \bigg| \nn\\
&\lesssim (1 + |x|^{C})\bigg\{ \|\p\eta\|_{\infty} + 
\sup_{\theta}\bigg| \bigg(\frac{c(x,\gam_{0})}{c(x,\gam)}z + \mfb(x,\theta)\bigg) \p\eta\bigg(\frac{c(x,\gam_{0})}{c(x,\gam)}z + \mfb(x,\theta)\bigg) \bigg| \bigg\} \nn\\
&\lesssim (1+|x|^{C})\bigg( 1+\int (1+|z|^{\kappa}) \phi_{1}(y)dy\bigg) \lesssim 1+|x|^{C}.
\nn
\end{align}
Consequently,
\begin{equation}
\sup_{\theta}\left|\p_{x}f^{1}(x,\theta)\right| \lesssim 1+|x|^{C}.
\label{prop_mainULLN_1.75}
\end{equation}
The claim follows on applying Lemma \ref{lem.aux.LLN} with \eqref{prop_mainULLN_1.25}, \eqref{prop_mainULLN_1.5}, and \eqref{prop_mainULLN_1.75}.

\medskip

(2) For $\beta\in(1,2)$, we have
\begin{align}
& h^{1-1/\beta}\overline{U}_{2,n}^{0}(\theta) \nn\\
&=\frac{1}{n}\sumj \pi_{j-1}(\theta)\E^{j-1}\left\{ \eta\left( \del'_{j}(\gam)+h^{1-1/\beta}\mfb_{j-1}(\theta) \right) \right\} 
\nn \\
&=\frac{1}{n}\sumj \pi_{j-1}(\theta)
\int \eta\bigg( \frac{c_{j-1}(\gam_{0})}{c_{j-1}(\gam)}z\bigg)f_{h}(z)dz
\nn\\
&{}\qquad + h^{1-1/\beta}\frac{1}{n}\sumj \pi_{j-1}(\theta)
\mfb_{j-1}(\theta)\int_{0}^{1}\E^{j-1}\left\{\p\eta\left(\del'_{j}(\gam)+sh^{1-1/\beta}\mfb_{j-1}(\theta)\right)\right\}ds.
\label{prop_mainULLN_2}
\end{align}
As with the case $\beta=1$, the first term in the rightmost side of \eqref{prop_mainULLN_2} turns out to be equal to $\frac{1}{T}\int_{0}^{T}\pi(X_{t},\theta)\eta(\frac{c(X_{t},\gam_{0})}{c(X_{t},\gam)}z)\phi_{\beta}(dz)dzdt + o_{p}(1)$ uniformly in $\theta$. Moreover, by the boundedness of $\p\eta$ and the estimate $|\pi_{j-1}(\theta)\mfb_{j-1}(\theta)| \lesssim 1+|X_{t_{j-1}}|^{C}$, the second term is $O_{p}(h^{1-1/\beta})=o_{p}(1)$ uniformly in $\theta$. Hence Lemma \ref{lem.aux.LLN} ends the proof of the first half.

Under the conditions in Lemma \ref{lem_aux.se1}, it follows from \eqref{U2n_deco} and Lemma \ref{lem_aux.se_r1} that
\begin{equation}
\sup_{\theta}\bigg|\frac{1}{nh^{1-1/\beta}}U_{n}(\theta) - \overline{U}_{2,n}^{0}(\theta)\bigg| = O_{p}(h).
\nn
\end{equation}
If $\eta$ is odd, the symmetry of the density $f_{h}$ implies that the first term in the rightmost side of \eqref{prop_mainULLN_2} a.s. equals $0$ for each $\gam$.
Then $\sup_{\theta}|\overline{U}_{2,n}^{0}(\theta)|=O_{p}(1)$, hence the latter claim.
\end{proof}

\medskip

We will also need the next corollary.

\begin{cor}
Assume that the conditions in Lemma \ref{lem_aux.se1} hold, let $\beta\in(1,2)$, and let $\eta:\mbbr\to\mbbr$ be an odd function. Then, for every $q>0$ we have
\begin{equation}
\frac{1}{nh^{1-1/\beta}} \sumj \pi_{j-1}(\theta)\E^{j-1}\left\{\eta(\ep_{j}\left(\al_{0},\gam)\right)\right\} = O^{\ast}_{L^{q}}(h),
\nonumber
\end{equation}
and also
\begin{equation}
\frac{1}{nh^{1-1/\beta}} \sumj \pi_{j-1}(\theta)\eta(\ep_{j}\left(\al_{0},\gam)\right) = O^{\ast}_{L^{q}}\left(
(\sqrt{n}h^{1-1/\beta})^{-1}
\right).
\nonumber
\end{equation}
\label{prop_mainULLN.cor}
\end{cor}

\begin{proof}
We have $U^{0}_{2,n}(\theta)\equiv 0$ from the first identity in \eqref{prop_mainULLN_2} and the fact $\mfb_{j-1}(\al_{0},\gam)\equiv 0$. This combined with \eqref{U2n_deco} and Lemmas \ref{lem_aux.se1} and \ref{lem_aux.se_r1} ends the proof.
\end{proof}

\subsection{Proof of Theorem \ref{sqmle.iv_bda_thm}: consistency}\label{sec_proof.consistency}

For convenience we state the following lemma.

\begin{lem}[Consistency under possible multi-scaling]
Let $K_1\subset\mbbr^{p_1}$ and $K_2\subset\mbbr^{p_2}$ be compact sets, and let $H_{n}: K_1 \times K_2\to\mbbr$ be a random function of the form
\begin{equation}
H_{n}(u_{1},u_{2}) = k_{1,n}H_{1,n}(u_{1}) + k_{2,n}H_{2,n}(u_{1},u_{2})
\nonumber
\end{equation}
for some positive non-random sequences $(k_{1,n})$ and $(k_{2,n})$ and some continuous random functions $H_{1,n}: K_1 \to \mbbr$ and $H_{2,n}: K_1\times K_2 \to \mbbr$. Let $(u_{1,0}, u_{2,0})\in K_1^{\circ} \times K_2^{\circ}$ be a non-random vector. Assume the following conditions:
\begin{itemize}
\item $k_{2,n}=o(k_{1,n})$;
\item $\sup_{u_{1}}|H_{1,n}(u_{1})-H_{1,0}(u_{1})|\cip 0$ and $\sup_{(u_{1},u_{2})}|H_{2,n}(u_{1},u_{2})-H_{2,0}(u_{1},u_{2})|\cip 0$ for some continuous random functions $H_{1,0}$ and $H_{2,0}$;
\item $\{u_{1,0}\}=\argmax H_{1,0}$ and $\{u_{2,0}\}=\argmax H_{2,0}(u_{1,0},\cdot)$ a.s.
\end{itemize}
Then, for any $(\hat{u}_{1,n},\hat{u}_{2,n}) \in K_1\times K_2$ such that $H_{n}(\hat{u}_{1,n},\hat{u}_{2,n}) \ge \sup H_{n} - o_{p}(k_{2,n})$, we have $(\hat{u}_{1,n},\hat{u}_{2,n})\cip(u_{1,0},u_{2,0})$.
\label{lem.2step.consistency}
\end{lem}

%

Lemma \ref{lem.2step.consistency} easily follows on applying the argmax theorem (for example \cite{vdV98}) twice for the random functions
\begin{align}
u_{1} &\mapsto k_{1,n}^{-1}H_{n}(u_{1},\hat{u}_{2,n})=H_{1,n}(u_{1}) + k_{2,n}k_{1,n}^{-1}H_{2,n}(u_{1},\hat{u}_{2,n}), \nn\\
u_{2} &\mapsto k_{2,n}^{-1}\{H_{n}(\hat{u}_{1,n},u_{2}) - H_{n}(\hat{u}_{1,n},u_{2,0})\}=H_{2,n}(\hat{u}_{1,n},u_{2})-H_{2,n}(\hat{u}_{1,n},u_{2,0}),
\nonumber
\end{align}
in this order.

\medskip

Returning to our model, we make a few remarks on Assumption \ref{A_iden}.
Recall the notation $\mfb(x,\theta) =c^{-1}(x,\gam)\{a(x,\al_{0})-a(x,\al)\}$ and $g_{\beta}(y)=\p_{y}\log\phi_{\beta}(y)$. For $\beta>1$, we define the random functions $\mbby_{\beta,1}(\cdot)=\mbby_{\beta,1}(\cdot;\gam_{0}): \Theta_{\gam}\to\mbbr$ and $\mbby_{\beta,2}(\cdot)=\mbby_{\beta,2}(\cdot;\tz): \Theta\to\mbbr$ by
\begin{align}
\mbby_{\beta,1}(\gam) &= \frac{1}{T}\int_{0}^{T}\int\bigg[ 
\log\bigg\{\frac{c(X_{t},\gam_{0})}{c(X_{t},\gam)}\phi_{\beta}\bigg(\frac{c(X_{t},\gam_{0})}{c(X_{t},\gam)}z\bigg)\bigg\}
-\log\phi_{\beta}(z)\bigg]\phi_{\beta}(z)dzdt,
\label{def_Y0b.a} \\
\mbby_{\beta,2}(\theta) &= \frac{1}{2T}\int_{0}^{T}\mfb^{2}(X_{t},\theta)
\int\p g_{\beta}\bigg(\frac{c(X_{t},\gam_{0})}{c(X_{t},\gam)}z\bigg)\phi_{\beta}(z)dzdt.
\nn
\end{align}
We also define $\mbby_{1}'(\cdot)=\mbby_{1}'(\cdot;\tz): \Theta\to\mbbr$ by
\begin{align}
\mbby_{1}'(\theta) &= \frac{1}{T}\int_{0}^{T}\int\bigg[ 
\log\bigg\{\frac{c(X_{t},\gam_{0})}{c(X_{t},\gam)}
\phi_{1}\bigg(\frac{c(X_{t},\gam_{0})}{c(X_{t},\gam)}z + \mfb(X_{t},\theta)
\bigg)\bigg\} -\log\phi_{1}(z)\bigg]\phi_{1}(z)dzdt.
\nn
\end{align}
These three functions are continuous in $\theta$. Since the function $z\mapsto c_{1}\phi_{\beta}(c_{1}z+c_{2})$ defines a probability density for every constants $c_{1}>0$ and $c_{2}\in\mbbr$, Jensen's inequality (applied $\omega$-wise) imply that the $dt$-integrand in \eqref{def_Y0b.a} is non-positive. The equality $\mbby_{\beta,1}(\gam)=0$ holds only when the $dt$-integrand is zero for $t\in[0,T]$ a.s., hence $\{\gam_{0}\}=\argmax\mbby_{\beta,1}$ a.s. Similarly, $\{\tz\}=\argmax\mbby_{1}'$ a.s.
Moreover,
\begin{align}
\mbby_{\beta,2}(\al,\gam_{0})
&=\frac{1}{2T}\int_{0}^{T}\mfb^{2}(X_{t},(\al,\gam_{0}))\int\p g_{\beta}(z)\phi_{\beta}(z)dzdt \nn\\
&= - \frac{1}{2}\int\frac{\{\p\phi_{\beta}(z)\}^{2}}{\phi_{\beta}(z)}dz\cdot\frac{1}{T}\int_{0}^{T}c^{-2}(X_{t},\gam_{0})\{a(X_{t},\al_{0})-a(X_{t},\al)\}^{2}dt \le 0,
\nonumber
\end{align}
where the maximum $0$ is attained if and only if $\al=\al_{0}$.

\subsubsection{Case of $\beta=1$}

Let
\begin{equation}
\mbby_{1,n}'(\theta):=\frac{1}{n}\big(\sqllf_{n}(\theta)-\sqllf_{n}(\theta_{0})\big)
=\frac{1}{n}\sumj \bigg( \log\frac{c_{j-1}(\gam_{0})}{c_{j-1}(\gam)} +\log\phi_{1}(\ep_{j}(\theta)) - \log\phi_{1}(\ep_{j}(\tz))\bigg).
\nn
\end{equation}
Since $\{\tz\}=\argmax\mbby_{1}'$ a.s., by means of Lemma \ref{lem.2step.consistency} the consistency of $\tes\,(\in\argmax\mbby_{1,n}')$ is ensured by the uniform convergence $\sup_{\theta}|\mbby_{1,n}'(\theta)-\mbby_{1}'(\theta)| \cip 0$. This follows from Lemma \ref{lem.aux.LLN} and Proposition \ref{prop_mainULLN}(1) with $\pi(x,\theta)\equiv 1$ and $\eta=\log\phi_{1}$.

\subsubsection{Case of $\beta\in(1,2)$}

We have
\begin{equation}
\sqllf_n(\theta)-\sqllf_n(\tz)=k_{n}\mbby_{\beta,1,n}(\gam) + l_{n}\mbby_{\beta,2,n}(\al,\gam),
\nonumber
\end{equation}
where $k_{n}:=n$, $l_{n}:=nh^{2(1-1/\beta)}$, and
\begin{align}
\mbby_{\beta,1,n}(\gam) &:=\frac{1}{n}\{\sqllf_n(\al_0,\gam)-\sqllf_n(\al_0,\gam_0)\}, \nn\\
\mbby_{\beta,2,n}(\al,\gam) &:=\frac{1}{nh^{2(1-1/\beta)}}\{\sqllf_n(\al,\gam)-\sqllf_n(\al_0,\gam)\}.
\nonumber
\end{align}
By Lemma \ref{lem.2step.consistency}, it suffices to prove the uniform convergences:
\begin{align}
& \sup_{\gam}\left| \mbby_{\beta,1,n}(\gam) -\mbby_{\beta,1}(\gam)\right|\cip 0,
\label{iv.c.1*} \\
& \sup_{\theta}\left|\mbby_{\beta,2,n}(\theta)-\mbby_{\beta,2}(\theta)\right|\cip 0.
\label{iv.c.2*}
\end{align}
The proof of \eqref{iv.c.1*} is much the same as in the case of $\beta=1$, hence we only prove \eqref{iv.c.2*}. Observe that
\begin{align}
\mbby_{\beta,2,n}(\theta)
&= \frac{1}{nh^{2(1-1/\beta)}}\sumj\bigg( \log\phi_{\beta}(\ep_{j}(\theta)) - \log\phi_{\beta}(\ep_{j}(\al_{0},\gam)) \bigg) \nn\\
&= \frac{1}{nh^{1-1/\beta}}\sumj \mfb_{j-1}(\theta) g_{\beta}(\ep_{j}(\al_{0},\gam)) 
+\frac{1}{2n}\sumj \mfb_{j-1}^{2}(\theta) \p g_{\beta}(\ep_{j}(\al_{0},\gam)) \nn\\
&{}\qquad +\frac{1}{2n}\sumj \mfb_{j-1}^{2}(\theta)
\left\{ \p g_{\beta}(\tilde{\ep}_{j}(\theta)) - \p g_{\beta}(\ep_{j}(\al_{0},\gam)) \right\} \nn\\
&=: \mbby_{\beta,2,n}'(\theta) + \mbby_{\beta,2,n}^{0}(\theta) + \mbby_{\beta,2,n}''(\theta),
\nn
\end{align}
where $\tilde{\ep}_{j}(\theta)$ is a random point on the segment connecting $\ep_{j}(\theta)$ and $\ep_{j}(\al_{0},\gam)$.
Since $g_{\beta}$ is odd, by means of \eqref{lem_aux.se1+1} and Corollary \ref{prop_mainULLN.cor} we have $\sup_{\theta}|\mbby_{\beta,2,n}'(\theta)|=o_{p}(1)$.
We also get $\sup_{\theta}|\mbby_{\beta,2,n}''(\theta)|=o_{p}(1)$, by noting that
$\sup_{\theta}|\tilde{\ep}_{j}(\theta)-\ep_{j}(\al_{0},\gam)| \le \sup_{\theta}|\ep_{j}(\theta) - \ep_{j}(\al_{0},\gam)|\lesssim (1+|X_{t_{j-1}}|^{C})h^{1-1/\beta}=(1+|X_{t_{j-1}}|^{C})\cdot o(1)$.
It remains to look at $\mbby_{\beta,2,n}^{0}$. The function $g_{\beta}$ is bounded and smooth, and satisfies that
\begin{equation}
\sup_{y}|y|^{k+1}\left|\p^{k}g_{\beta}(y)\right|<\infty
\label{case.b>1_eq1}
\end{equation}
for each non-negative integer $k$. The convergence $\sup_{\theta}|\mbby_{\beta,2,n}^{0}(\theta) - \mbby_{\beta,2}(\theta)|=o_{p}(1)$ now follows on applying Proposition \ref{prop_mainULLN}(2) for $\pi(x,\theta)=\frac{1}{2}\mfb^{2}(x,\theta)$ and $\eta=\p g_{\beta}$ with the trivial modification that inside the function $\eta$ we have ``$\ep_{j}(\al_{0},\gam)$'' instead of ``$\ep_{j}(\theta)$''.

\subsection{Proof of Theorem \ref{sqmle.iv_bda_thm}: asymptotic mixed normality}\label{sec_proof.amn}

We introduce the rate matrix
\begin{equation}
D_{n}=\diag(D_{n,1},\dots,D_{n,p}):=\diag\left( \sqrt{n}h^{1-1/\beta}I_{p_{\al}},\,\sqrt{n}I_{p_{\gam}} \right) \in\mbbr^{p}\otimes\mbbr^{p}.,
\nonumber
\end{equation}
and then denote the normalized SQMLE by
\begin{equation}
\hat{u}_{n}=\left( \sqrt{n}h^{1-1/\beta}(\aes -\al_{0}),\,\sqrt{n}(\ges-\gam_{0})\right) := D_{n}(\tes-\tz).
\nonumber
\end{equation}
The consistency allows us to focus on the event $\{\tes\in\Theta\}$, on which we have $\p_{\theta}\sqllf_{n}(\tes)=0$ so that the two-term Taylor expansion gives
\begin{equation}
\left( -D_{n}^{-1}\p_{\theta}^{2}\sqllf_{n}(\tz)D_{n}^{-1} + \hat{r}_{n} \right)\hat{u}_{n}=D_{n}^{-1}\p_{\theta}\sqllf_{n}(\tz),
\label{amn.proof-5}
\end{equation}
where $\hat{r}_{n}=\{\hat{r}_{n}^{kl}\}_{k,l}$ is a bilinear form such that 
\begin{equation}
|\hat{r}_{n}| \lesssim \sum_{k,l,m=1}^{p}\bigg(D_{n,k}^{-1}D_{n,l}^{-1}\sup_{\theta}
\left|\p_{\theta_{k}}\p_{\theta_{l}}\p_{\theta_{m}}\sqllf_{n}(\theta)\right|\bigg) \big|\hat{\theta}_{n,m}-\theta_{0,m}\big|.
\nn
\end{equation}
Here we wrote $\theta=(\theta_{i})_{i=1}^{p}$, and similarly for $\tz$ and $\tes$.
Let
\begin{equation}
\D_{n,T} :=D_{n}^{-1}\p_{\theta}\sqllf_{n}(\tz), \qquad \Gam_{n,T} := -D_{n}^{-1}\p_{\theta}^{2}\sqllf_{n}(\tz)D_{n}^{-1}.
\nonumber
\end{equation}
If we have
\begin{align}
& \left( \D_{n,T},\, \Gam_{n,T}\right) \cil \left(\D_{T},\, \Gam_{T}(\tz;\beta)\right)
\quad\text{where}\quad \D_{T}\sim MN_{p}\left(0,\, \Gam_{T}(\tz;\beta) \right),
\label{amn.proof-DGjoint} \\
& D_{n,k}^{-1}D_{n,l}^{-1}\sup_{\theta}
\left|\p_{\theta_{k}}\p_{\theta_{l}}\p_{\theta_{m}}\sqllf_{n}(\theta)\right| = O_{p}(1),\qquad k,l,m\in\{1,\dots,p\},
\label{amn.proof-4}
\end{align}
then $\hat{r}_{n}=o_{p}(1)$ and
\begin{align}
\hat{u}_{n} &= \bigg( \Gam_{T}(\tz;\beta) + o_{p}(1) \bigg)^{-1}\D_{n,T} \nn\\
&= \Gam_{T}^{-1}(\tz;\beta) \D_{n,T} + o_{p}(1) \nn\\
&\cil \Gam_{T}^{-1}(\tz;\beta) \D_{T} \sim MN_{p}\left( 0,\, \Gam_{T}^{-1}(\tz;\beta) \right),
\nn
\end{align}
completing the proof.
Since $\Gam_{T}(\tz;\beta)$ may be random, the appropriate mode of convergence to deduce \eqref{amn.proof-DGjoint} is the stable convergence in law: 
recall we say that $\D_{n,T}$ convergences stably in law to $\D_{T}$ if $(\D_{n,T},G_{n}) \cil (\D_{T},G)$ for every $\mcf$-measurable random variables $G_{n}$ and $G$ such that $G_{n}\cip G$;
we refer to \cite{GenJac93}, \cite{Jac97}, \cite{Jac12}, \cite{JacPro12}, \cite[Chapters VIII.5c and IX.7]{JacShi03} for detailed accounts of the stable convergence in law which can handle statistics for high-frequency data.
It therefore suffices to prove \eqref{amn.proof-4} and
\begin{align}
\D_{n,T} & \scl \D_{T} \sim MN_{p}\big( 0,\, \Gam_{T}(\tz;\beta) \big),
\label{amn.proof-2} \\
\Gam_{n,T} & \cip \Gam_{T}(\tz;\beta).
\label{amn.proof-3}
\end{align}

\subsubsection{Proof of \eqref{amn.proof-4}}

We may and do suppose that $p_{\al}=p_{\gam}=1$. Write $R(x,\theta)$ for generic matrix-valued function on $\mbbr\times\Theta$ such that $\sup_{\theta}|R(x,\theta)| \lesssim 1 + |x|^{C}$.
By straightforward computations,
\begin{align}
\frac{1}{nh^{2(1-1/\beta)}}\p_{\al}^{3}\sqllf_{n}(\theta) 
&=\frac{1}{nh^{1-1/\beta}}\sumj R_{j-1}(\theta)g_{\beta}(\ep_{j}(\theta)) \nn\\
&{}\qquad +\frac{1}{n}\sumj \bigg( R_{j-1}(\theta)\p g_{\beta}(\ep_{j}(\theta)) + h^{1-1/\beta}R_{j-1}(\theta)\p^{2}g_{\beta}(\ep_{j}(\theta)) \bigg),
\nn\\
\frac{1}{nh^{2(1-1/\beta)}}\p_{\al}^{2}\p_{\gam}\sqllf_{n}(\theta) 
&=\frac{1}{nh^{1-1/\beta}}\sumj \bigg( R_{j-1}(\theta)g_{\beta}(\ep_{j}(\theta)) + R_{j-1}(\theta)\ep_{j}(\theta)\p g_{\beta}(\ep_{j}(\theta)) \bigg) \nn\\
&{}\qquad +\frac{1}{n}\sumj \bigg( R_{j-1}(\theta)\ep_{j}(\theta)\p^{2}g_{\beta}(\ep_{j}(\theta)) + R_{j-1}(\theta)\p g_{\beta}(\ep_{j}(\theta)) \bigg), \nn\\
\frac{1}{n}\p_{\gam}^{3}\sqllf_{n}(\theta) 
&= \frac{1}{n}\sumj \bigg( R_{j-1}(\theta) + R_{j-1}(\theta)\ep_{j}(\theta)g_{\beta}(\ep_{j}(\theta)) \nn\\
&{}\qquad + R_{j-1}(\theta)\ep_{j}^{2}(\theta)\p g_{\beta}(\ep_{j}(\theta)) + R_{j-1}(\theta)\ep_{j}^{3}(\theta)\p^{2}g_{\beta}(\ep_{j}(\theta)) \bigg),
\nn\\
\frac{1}{nh^{1-1/\beta}}\p_{\al}\p_{\gam}^{2}\sqllf_{n}(\theta) 
&=\frac{1}{n}\sumj \bigg( R_{j-1}(\theta)g_{\beta}(\ep_{j}(\theta))+ R_{j-1}(\theta)\ep_{j}(\theta)\p g_{\beta}(\ep_{j}(\theta)) \nn\\
&{}\qquad + R_{j-1}(\theta)\ep_{j}^{2}(\theta)\p^{2}g_{\beta}(\ep_{j}(\theta)) \bigg).
\nonumber
\end{align}
By \eqref{case.b>1_eq1}, all the terms having the factor ``$1/n$'' in front of the summation sign in the above right-hand sides are $O_{p}(1)$ uniformly in $\theta$. Since the functions $y\mapsto g_{\beta}(y)$ and $y\mapsto y\p g_{\beta}(y)$ are odd, it follows from Proposition \ref{prop_mainULLN}(2) that both
\begin{align}
& \frac{1}{nh^{1-1/\beta}}\sumj R_{j-1}(\theta)g_{\beta}(\ep_{j}(\theta)) = O_{p}(1), \nn\\
& \frac{1}{nh^{1-1/\beta}}\sumj R_{j-1}(\theta)\ep_{j}(\theta)\p g_{\beta}(\ep_{j}(\theta)) = O_{p}(1),
\nonumber
\end{align}
hold uniformly in $\theta$. These observations are enough to conclude \eqref{amn.proof-4}.

\subsubsection{Proof of \eqref{amn.proof-2}}\label{sec_score.scle.proof}

Let $\ep_{j}:=\ep_{j}(\tz)$ and observe that
\begin{align}
\D_{n,T} &= \bigg( \frac{1}{\sqrt{n}h^{1-1/\beta}}\p_{\al}\sqllf_{n}(\tz),\, 
\frac{1}{\sqrt{n}}\p_{\gam}\sqllf_{n}(\tz) \bigg) \nn\\
&= \bigg( -\frac{1}{\sqrt{n}}\sumj\frac{\p_{\al}a_{j-1}(\al_{0})}{c_{j-1}(\gam_{0})}g_{\beta}(\ep_{j}),\, 
-\frac{1}{\sqrt{n}}\sumj
\frac{\p_{\gam}c_{j-1}(\gam_{0})}{c_{j-1}(\gam_{0})}\left\{1+\ep_{j}g_{\beta}(\ep_{j})\right\}
\bigg).
\nonumber
\end{align}
To apply Jacod's stable central limit theorem, we introduce the partial sum process in $\mbbd([0,T];\mbbr^{p})$,
where $\mbbd([0,T];\mbbr^{p})$ denote the space of {\cadlag} processes over $[0,T]$ taking values in $\mbbr^{p}$:
\begin{equation}
\D_{n,t} := \bigg( -\frac{1}{\sqrt{n}}\sum_{j=1}^{[t/h]}
\frac{\p_{\al}a_{j-1}(\al_{0})}{c_{j-1}(\gam_{0})}g_{\beta}(\ep_{j}),\, 
-\frac{1}{\sqrt{n}}\sum_{j=1}^{[t/h]}
\frac{\p_{\gam}c_{j-1}(\gam_{0})}{c_{j-1}(\gam_{0})}\left\{1+\ep_{j}g_{\beta}(\ep_{j})\right\}
\bigg),\quad t\in[0,T].
\nonumber
\end{equation}
Let
\begin{align}
\pi_{j-1}=\pi_{j-1}(\tz) &:= \diag\bigg( -\frac{\p_{\al}a_{j-1}(\al_{0})}{c_{j-1}(\gam_{0})},
\,  -\frac{\p_{\gam}c_{j-1}(\gam_{0})}{c_{j-1}(\gam_{0})} \bigg)
\in\mbbr^{p}\otimes\mbbr^{2}, \label{6.4.2_pi.def}\\
\eta(y) &:= \left(g_{\beta}(y),\, 1+yg_{\beta}(y) \right) = \left(g_{\beta}(y),\, k_{\beta}(y) \right)\in\mbbr^{2}
\quad\text{(bounded)},
\label{6.4.2_eta.def}
\end{align}
so that $\D_{n,t}=n^{-1/2}\sumj \pi_{j-1}\eta(\ep_{j})$. Write $\Gam_{t}(\tz;\beta)$ for $\Gam_{T}(\tz;\beta)$ with the integral signs ``$\int_{0}^{T}$'' in their definitions replaced by ``$\int_{0}^{t}$''. Then, by means of \cite[Theorem 3-2]{Jac97} (or \cite[Theorem IX.7.28]{JacShi03}), the stable convergence \eqref{amn.proof-2} is implied by the following conditions: for each $t\in[0,T]$ and for any bounded $(\mcf_{t})$-martingale $M$,
\begin{align}
& \sum_{j=1}^{[t/h]}\E^{j-1}\bigg(\bigg|\frac{1}{\sqrt{n}}\pi_{j-1}\eta(\ep_{j})\bigg|^{4}\bigg)
\cip 0, \label{sclt.proof-1} \\
& \frac{1}{n}\sum_{j=1}^{[t/h]}
\pi_{j-1} \cdot \E^{j-1}\Bigl\{ \Bigl(\eta(\ep_{j}) - \E^{j-1}\{\eta(\ep_{j})\}\Bigr)^{\otimes 2}\Bigr\} \pi_{j-1} \cip \Gam_{t}(\tz;\beta),
\label{sclt.proof-2} \\
& \sup_{t\in[0,T]}\bigg|\frac{1}{\sqrt{n}}\sum_{j=1}^{[t/h]}\pi_{j-1}\E^{j-1}\{\eta(\ep_{j})\} \bigg| \cip 0,
\label{sclt.proof-3} \\
& \sum_{j=1}^{[t/h]}\E^{j-1}\bigg(\frac{1}{\sqrt{n}}\pi_{j-1}\eta(\ep_{j})\D_{j}M\bigg)\cip 0.
\label{sclt.proof-4}
\end{align}

\medskip

The Lyapunov condition \eqref{sclt.proof-1} trivially holds since $\eta$ is bounded and $|\pi_{j-1}| \lesssim 1+|X_{t_{j-1}}|^{C}$. For \eqref{sclt.proof-2}, arguing as in the proof of Lemma \ref{lem_aux.se_r1} with $\int\eta(z)\phi_{\beta}(z)dz=0$ and noting that
\begin{align}
& \bigg|\int\eta(z)\{f_{h}(z)-\phi_{\beta}(z)\}dz\bigg| \le \|\eta\|_{\infty}\int |f_{h}(z)-\phi_{\beta}(z)|dz=o(n^{-1/2}), \nn\\
& \int g_{\beta}(y)k_{\beta}(y)\phi_{\beta}(y)dy=\int \{g_{\beta}(y)+yg_{\beta}^{2}(y)\}\phi_{\beta}(y)dy=0,
\nonumber
\end{align}
we obtain for each $q>0$
\begin{align}
\E^{j-1}\{\eta(\ep_{j})\} &= \int\eta(z)f_{h}(z)dz + O^{\ast}_{L^{q}}(h^{2-1/\beta}) \label{sclt.proof-5} \\
&=\int\eta(z)\phi_{\beta}(z)dz + O^{\ast}_{L^{q}}(n^{-1/2}) = O^{\ast}_{L^{q}}(n^{-1/2}), \nn\\
\E^{j-1}\left\{\eta^{\otimes 2}(\ep_{j}) \right\} &= \int\eta^{\otimes 2}(z)f_{h}(z)dz + O^{\ast}_{L^{q}}(h^{2-1/\beta}) \nn\\
&=\int\eta^{\otimes 2}(z)\phi_{\beta}(z)dz + O^{\ast}_{L^{q}}(n^{-1/2})=
\begin{pmatrix}
C_{\al}(\beta) & 0 \\
0 & C_{\gam}(\beta)
\end{pmatrix}
+ O^{\ast}_{L^{q}}(n^{-1/2}).
\nonumber
\end{align}
Then the left-hand side of \eqref{sclt.proof-2} equals 
\begin{align}
& \frac{1}{n}\sum_{j=1}^{[t/h]}\pi_{j-1} \cdot \left(\int\eta^{\otimes 2}(z)\phi_{\beta}(z)dz\right) \pi_{j-1} + O_{p}(n^{-1/2}) \nn\\
&=\frac{1}{n}\sum_{j=1}^{[t/h]}\pi_{j-1} \cdot 
\begin{pmatrix}
C_{\al}(\beta) & 0 \\
0 & C_{\gam}(\beta)
\end{pmatrix}
\pi_{j-1} + O_{p}(n^{-1/2}).
\nonumber
\end{align}
By Lemma \ref{lem.aux.LLN} the first term in the right-hand side converges in probability to $\Gam_{t}(\tz;\beta)$, hence \eqref{sclt.proof-2} is verified.

\medskip

The convergence \eqref{sclt.proof-3} follows on applying \eqref{sclt.proof-5} and Lemma \ref{lem.aux.LLN}:
\begin{align}
\frac{1}{\sqrt{n}}\sum_{j=1}^{[t/h]}\pi_{j-1}\E^{j-1}\{\eta(\ep_{j})\}
&=\frac{1}{n}\sum_{j=1}^{[t/h]}\pi_{j-1}\bigg(\sqrt{n}\int\eta(z)f_{h}(z)dz\bigg) + O_{p}(\sqrt{n}h^{2-1/\beta})
\nn\\
&=o_{p}(1) + O_{p}(h^{3/2-1/\beta}) = o_{p}(1),
\nonumber
\end{align}
all the order symbols above being uniformly valid in $t\in[0,T]$.

\medskip

Finally we turn to \eqref{sclt.proof-4}. By means of the decomposition theorem for local martingales (see \cite[Theorem I.4.18]{JacShi03}), we may write $M=M^{c}+M^{d}$ for the continuous part $M^{c}$ and the associated purely discontinuous part $M^{d}$. Our underlying probability space supports no Wiener process, so that in view of the martingale representation theorem \cite[Theorem III.4.34]{JacShi03} for $M$, we may set $M^{c}=0$; recall \eqref{Ft_def}. To show \eqref{sclt.proof-4} we will follow an analogous way to \cite{TodTau12} with successive use of general theory of martingales convergence.

It suffices to prove the claim when both $\pi$ and $\eta$ are real-valued. The jumps of $M$ over $[0,T]$ are bounded, and we have $M^{n}_{t}:=\sum_{j=1}^{[t/h]}\D_{j}M \asc M_{t}=M^{d}_{t}$ in $\mbbd([0,T];\mbbr)$. Let
\begin{equation}
N^{n}_{t} := \sum_{j=1}^{[t/h]} \frac{1}{\sqrt{n}}\pi_{j-1} \tilde{\eta}(\ep_{j}),
\nonumber
\end{equation}
with $\tilde{\eta}(\ep_{j}):=\eta(\ep_{j})-\E^{j-1}\{\eta(\ep_{j})\}$.
For each $n$, $N^{n}$ is a local martingale with respect to $(\mcf_{t})$, and \eqref{sclt.proof-4} equals that $\la M^{n},N^{n}\ra_{t} \to 0$ for each $t\le T$. The angle-bracket process
\begin{equation}
\la N^{n}\ra_{t} = \frac{1}{n}\sum_{j=1}^{[t/h]}\pi^{2}_{j-1}\E^{j-1}\{ \tilde{\eta}^{2}(\ep_{j})\}
\nonumber
\end{equation}
is $C$-tight, that is, it is tight in $\mbbd([0,T];\mbbr)$ and any weak limit process has a.s. continuous sample paths; this can be deduced as in the proof of \eqref{sclt.proof-2}. Hence, by \cite[Theorem VI.4.13]{JacShi03} the sequence $(N^{n})$ is tight in $\mbbd([0,T];\mbbr)$. Further, for every $\ep>0$, as in the case of \eqref{sclt.proof-1} we have
\begin{equation}
\pr\bigg(\sup_{t\le T}|\D N^{n}_{t}|>\ep\bigg) = \pr\bigg(\max_{j\le n}|\D_{j}N^{n}|>\ep\bigg)
\le \sumj \pr\left(|\D_{j}N^{n}|>\ep\right) \lesssim \sumj \E\left(|\D_{j}N^{n}|^{4}\right) \to 0.
\nonumber
\end{equation}
We conclude from \cite[Theorem VI.3.26]{JacShi03} that $(N^{n})$ is $C$-tight.

Fix any $\{n'\}\subset\mbbn$. By \cite[Theorem VI.3.33]{JacShi03} the process $H^{n} := (M^{n},N^{n})$ is tight in $\mbbd([0,T];\mbbr)$. Hence, by Prokhorov's theorem we can pick a subsequence $\{n''\}\subset\{n'\}$ for which there exists a process $H=(M^{d},N)$ with $N$ being continuous, such that $H^{n''}\cil H$ along $\{n''\}$ in $\mbbd([0,T];\mbbr)$. By \eqref{loc.X.ineq} we have
\begin{equation}
\sup_{n}\E\bigg(\max_{j\le n}|\D_{j}N^{n}|\bigg)
\lesssim \sup_{n}\frac{1}{\sqrt{n}}\E\bigg( 1+\sup_{t\le T}|X_{t}|^{C}\bigg) < \infty,
\nonumber
\end{equation}
hence it follows from \cite[Corollary VI.6.30]{JacShi03} that the sequence $(H^{n''})$ is predictably uniformly tight.
In particular, $(H^{n''}, [H^{n''}]) \cil (H, [H])$ with the off-diagonal component of the limit quadratic-variation process being $0$ a.s.: $[M,N] =\la M^{c},N^{c}\ra + \sum_{s\le \cdot}(\D M_{s})(\D N_{s})=0$ a.s. identically (see \cite[Theorem I.4.52]{JacShi03}). Therefore, given any $\{n'\}\subset\mbbn$ we can find a further subsequence $\{n''\}\subset\{n'\}$ for which $[M^{n''},N^{n''}]\cil 0$. This concludes that
\begin{equation}
[M^{n},N^{n}]_{t}=\sum_{j=1}^{[t/h]}\frac{1}{\sqrt{n}}\pi_{j-1} \tilde{\eta}(\ep_{j}) \D_{j} M \cip 0
\label{sclt.proof-6}
\end{equation}
in $\mbbd([0,T];\mbbr)$.

Since $[M^{n},N^{n}] - \la M^{n},N^{n}\ra$ is a martingale, we may write
\begin{equation}
G^{n}_{t}:=[M^{n},N^{n}]_{t} - \la M^{n},N^{n}\ra_{t}=\int_{0}^{t}\int \chi^{n}(s,z)\tilde{\mu}(ds,dz)
\nonumber
\end{equation}
for some predictable process $\chi^{n}(s,z)$. Using the isometry property and the martingale property of the stochastic integral, we have $\E\{(G^{n}_{t})^{2}\}=\E\{\int_{0}^{t}\int \chi^{n}(s,z)^{2}ds\nu(dz)\}=\E\{\int_{0}^{t}\int \chi^{n}(s,z)^{2}\mu(ds,dz)\}
=\E\{\sum_{0<s\le t}(\D G^{n}_{s})^{2}\}=\E\{\sum_{0<s\le t}(\D M^{n}_{s}\D N^{n}_{s})^{2}\}$.
Since
\begin{equation}
(\D N^{n}_{s})^{2}\le\max_{j\le n}(\D_{j}N^{n})^{2}\lesssim \frac{1}{n}\bigg(1+\sup_{t\le T}|X_{t}|^{C}\bigg)
\nonumber
\end{equation}
and $\sum_{0<s\le T}(\D M^{n}_{s})^{2}$ is essentially bounded (for $M$ is bounded),
we obtain
\begin{equation}
\sup_{t\le T}\E\{(G^{n}_{t})^{2}\} \lesssim \frac{1}{n}\E\bigg(1+\sup_{t\le T}|X_{t}|^{C}\bigg) \to 0,
\nonumber
\end{equation}
which combined with \eqref{sclt.proof-6} yields \eqref{sclt.proof-4}: $\la M^{n},N^{n}\ra_{t} \cip 0$.
The proof of \eqref{amn.proof-2} is complete.

\medskip

\begin{rem}{\rm 
The setting \eqref{Ft_def} of the underlying filtration is not essential.
Even when the underlying probability space carries a Wiener process, we may still follow the martingale-representation argument as in \cite{TodTau12}.
\label{rem_sclt.proof}
}\qed\end{rem}

\subsubsection{Proof of \eqref{amn.proof-3}}\label{sec_lln.scle.proof}

The components of $\Gam_{n,T}$ consist of
\begin{align}
-\frac{1}{nh^{2(1-1/\beta)}}\p_{\al}^{2}\sqllf_{n}(\tz)&=\frac{1}{nh^{1-1/\beta}}\sumj
\frac{\p_{\al}^{2}a_{j-1}(\al_{0})}{c_{j-1}(\gam_{0})}g_{\beta}(\ep_{j})
-\frac{1}{n}\sumj
\frac{\{\p_{\al}a_{j-1}(\al_{0})\}^{\otimes 2}}{c_{j-1}^{2}(\gam_{0})}\p g_{\beta}(\ep_{j}),
\label{p.aa.H} \\
-\frac{1}{n}\p_{\gam}^{2}\sqllf_{n}(\tz)&=
-\frac{1}{n}\sumj \frac{\p_{\gam}^{2}c_{j-1}(\gam_{0})}{c_{j-1}(\gam_{0})}
\left\{1+\ep_{j}g_{\beta}(\ep_{j})\right\}
\label{p.gg.H}\\
&{}\qquad -\frac{1}{n}\sumj \frac{\{\p_{\gam}c_{j-1}(\gam_{0})\}^{\otimes 2}}{c_{j-1}^{2}(\gam_{0})}
\left\{1+2\ep_{j}g_{\beta}(\ep_{j})+\ep_{j}^{2}\p g_{\beta}(\ep_{j})\right\},
\nonumber \\
-\frac{1}{nh^{1-1/\beta}}\p_{\al}\p_{\gam}\sqllf_{n}(\tz) &= -\frac{1}{n}\sumj
\frac{\{\p_{\al}a_{j-1}(\al_{0})\}\otimes\{\p_{\gam}c_{j-1}(\gam_{0})\}}{c_{j-1}^{2}(\gam_{0})}
\left\{g_{\beta}(\ep_{j})+\ep_{j}\p g_{\beta}(\ep_{j})\right\}.
\nn
\end{align}
By \eqref{lem_aux.se1+1}, with Corollary \ref{prop_mainULLN.cor} when $\beta\in(1,2)$, the first term in the right-hand side of \eqref{p.aa.H} is $o_{p}(1)$. Since $-\int \p g_{\beta}(z)\phi_{\beta}(z)dz=\int g^{2}_{\beta}(z)\phi_{\beta}(z)dz=C_{\al}(\beta)$, by Proposition \ref{prop_mainULLN} we derive
\begin{align}
-\frac{1}{nh^{2(1-1/\beta)}}\p_{\al}^{2}\sqllf_{n}(\tz)
&= C_{\al}(\beta)\Sig_{T,\al}(\tz) + o_{p}(1).
\nonumber
\end{align}
By Proposition \ref{prop_mainULLN} and $\int k_{\beta}(z)\phi_{\beta}(z)dz=0$, the first term in the right-hand side of \eqref{p.gg.H} is $o_{p}(1)$. As for the second term, noting that the function $l_{\beta}(z):=1+2zg_{\beta}(z)+z^{2}\p g_{\beta}(z)$ satisfies $\int l_{\beta}(z)\phi_{\beta}(z)dz = -\int k_{\beta}^{2}(z)\phi_{\beta}(z)dz = -C_{\gam}(\beta)$, we obtain
\begin{align}
-\frac{1}{n}\p_{\gam}^{2}\sqllf_{n}(\tz) &= 
-\frac{1}{n}\sumj \frac{\{\p_{\gam}c_{j-1}(\gam_{0})\}^{\otimes 2}}{c_{j-1}^{2}(\gam_{0})}
\int l_{\beta}(z)\phi_{\beta}(z)dz + o_{p}(1) \nn\\
&= C_{\gam}(\beta)\Sig_{T,\gam}(\gam_{0}) + o_{p}(1).
\nonumber
\end{align}
Finally, since $\int\{ g_{\beta}(z) + z\p g_{\beta}(z)\}\phi_{\beta}(z)dz=0$, Proposition \ref{prop_mainULLN} concludes that
\begin{equation}
-\frac{1}{nh^{1-1/\beta}}\p_{\al}\p_{\gam}\sqllf_{n}(\tz) =o_{p}(1),
\nonumber
\end{equation}
completing the proof of \eqref{amn.proof-3}.

\subsection{Proof of Theorem \ref{sqmle.iv_bda_thm.cf}}

Under the condition $\int_{|z|>1}|z|^{q}\nu(dz)<\infty$ for every $q>0$, the moment estimates \eqref{loc.X.ineq} and \eqref{LusPag.ineq.result} are in force without truncating the support of $\nu$ (see Section \ref{sec_localization}).
Further, it follows from Lemmas \ref{key.lemma}(1) and \ref{key.lemma.cf} that we have both \eqref{llt.imp.1} and \eqref{llt.imp.2}.

\subsection{Proof of Corollary \ref{sqmle.iv_bda_thm.cor}}\label{sec_cor.proof}

The random mapping $\theta\mapsto (\Sig_{T,\al}(\theta), \Sig_{T,\gam}(\gam))$ is a.s. continuous, hence applying the uniform law of large numbers presented in Lemma \ref{lem.aux.LLN} we can deduce the convergences $\hat{\Sig}_{T,\al,n} \cip \Sig_{T,\al}(\tz)$ and $\hat{\Sig}_{T,\gam,n} \cip \Sig_{T,\gam}(\gam_{0})$. Then it is straightforward to derive \eqref{aqmle_amn_sn} from \eqref{amn.proof-5}, \eqref{amn.proof-DGjoint}, and \eqref{amn.proof-4}.

\subsection{Proof of Theorem \ref{sqmle_ergo_thm}}\label{sec_ergo.an.proof}

Most parts are essentially the same as in the proof of Theorem \ref{sqmle.iv_bda_thm.cf} (hence as in Theorem \ref{sqmle.iv_bda_thm}). We only sketch a brief outline.

The convergences \eqref{llt.imp.1} and \eqref{llt.imp.2} are valid under the present assumptions.
As in Theorem \ref{sqmle.iv_bda_thm.cf}, the localization introduced in Section \ref{sec_localization} is not necessary here, since, under the moment boundedness $\sup_{t}\E(|X_{t}|^{q})<\infty$ for any $q>0$ and the global Lipschitz property of $(a,c)$, we can deduce the large-time version of the latter inequality in \eqref{loc.X.ineq} by the standard argument:
for any $q\ge 2$ we have $\E(|X_{t+h}-X_{t}|^{q}|\mcf_{t}) \lesssim h (1+|X_{t}|^{C})$, hence in particular
\begin{equation}
\sup_{t\in\mbbrp}\E(|X_{t+h}-X_{t}|^{q}) \lesssim h \sup_{t\in\mbbrp}\E(1+|X_{s}|^{C}) \lesssim h.
\nn
\end{equation}
Obviously, \eqref{LusPag.ineq.result} remains the same and Lemmas \ref{lem_aux.se1} and \ref{lem_aux.se_r1} stay valid as well.

\medskip

As for the uniform low of large numbers under $T_{n}\to\infty$, we have the following ergodic counterpart to Lemma \ref{lem.aux.LLN}:

\begin{lem}
For any measurable function $f:\mbbr\times\overline{\Theta}\to\mbbr$ such that
\begin{equation}
\sup_{\theta}\left\{ |f(x,\theta)| + |\p_{x}f(x,\theta)| \right\} \lesssim 1+|x|^{C},
\nonumber
\end{equation}
we have
\begin{equation}
\sup_{\theta}\bigg|\frac{1}{n}\sum_{j=1}^{n} f_{j-1}(\theta) - \int f(x,\theta)\pi_{0}(dx)\bigg|\cip 0.
\nonumber
\end{equation}
\label{lem.aux.LLN_ergo}
\end{lem}

\begin{proof}
Write $\D_{n}^{f}(\theta)=n^{-1}\sum_{j=1}^{n} f_{j-1}(\theta) - \int f(x,\theta)\pi_{0}(dx)$. By \eqref{ergo_thm} we have $\D_{n}^{f}(\theta)\cip 0$ for each $\theta$, hence it suffices to show the tightness of $\{\sup_{\theta}|\p_{\theta}\D_{n}^{f}(\theta)|\}_{n}$ in $\mbbr$, which implies the tightness of $\{\D_{n}^{f}(\cdot)\}_{n}$ in $\mcc(\overline{\Theta})$. But this is obvious since
\begin{equation}
\sup_{\theta}|\p_{\theta}\D_{n}^{f}(\theta)|\lesssim \frac{1}{n}\sumj \sup_{\theta}|f_{j-1}(\theta)-\pi_{0}(f(\cdot,\theta))| \lesssim \frac{1}{n}\sumj (1+|X_{t_{j-1}}|^{C}).
\nonumber
\end{equation}
\end{proof}

\medskip

Having Lemma \ref{lem.aux.LLN_ergo} in hand, we can follow the contents of Sections \ref{sec_proof.consistency}, \ref{sec_proof.amn}, and \ref{sec_lln.scle.proof}. The proof of the central limit theorem is much easier than the mixed normal case, for we now have no need for looking at the step processes introduced in Section \ref{sec_score.scle.proof} and also for taking care of the asymptotic orthogonality condition \eqref{sclt.proof-4}. By means of the classical central limit theorem for martingale difference arrays \cite{Dvo77}, it suffices to show, with the same notation as in \eqref{6.4.2_pi.def} and \eqref{6.4.2_eta.def},
\begin{align}
& \sum_{j=1}^{n}\E^{j-1}\bigg(\bigg|\frac{1}{\sqrt{n}}\pi_{j-1}\eta(\ep_{j})\bigg|^{4}\bigg) \cip 0,
\nn
\\
& \frac{1}{n}\sum_{j=1}^{n}\pi_{j-1} \cdot \E^{j-1}\Bigl\{ \Bigl(\eta(\ep_{j}) - \E^{j-1}\{\eta(\ep_{j})\}\Bigr)^{\otimes 2}
\Bigr\} \pi_{j-1} \cip \diag\big(V_{\al}(\theta_{0};\beta),V_{\gam}(\theta_{0};\beta)\big),
\nn
\\
& \frac{1}{\sqrt{n}}\sum_{j=1}^{n}\pi_{j-1}\E^{j-1}\{\eta(\ep_{j})\} \cip 0,
\nn
\end{align}
all of which can be deduced from the same arguments as in Section \ref{sec_score.scle.proof}.

\bigskip

\noindent
{\bf Acknowledgement.}
The author is grateful to Professor Jean Jacod and to the two anonymous reviewers for their helpful comments, most of which have led to substantial improvements.
This work was partly supported by JSPS KAKENHI Grant Number JP26400204 and JP17K05367, and also JST CREST Grant Number JPMJCR14D7, Japan.

\bigskip


\def\cprime{$'$} \def\polhk#1{\setbox0=\hbox{#1}{\ooalign{\hidewidth
  \lower1.5ex\hbox{`}\hidewidth\crcr\unhbox0}}} \def\cprime{$'$}
  \def\cprime{$'$}

\end{document}